\documentclass{ipi}
\usepackage{amsmath}
  \usepackage{paralist}
  \usepackage{graphics} 
  \usepackage{epsfig} 
\usepackage{graphicx}  \usepackage{epstopdf}
 \usepackage[colorlinks=true]{hyperref}
\hypersetup{urlcolor=blue, citecolor=red}

\def\currentvolume{X}
 \def\currentissue{X}
  \def\currentyear{200X}
   
    \def\ppages{X--XX}
     \def\DOI{10.3934/xx.xx.xx.xx}

\newtheorem{theorem}{Theorem}[section]
\newtheorem{corollary}{Corollary}

\newtheorem{lemma}[theorem]{Lemma}

\theoremstyle{definition}

\newtheorem{remark}{Remark}

\newcommand{\rank}{{\hbox{rank}}}

\usepackage[usenames]{color}
\usepackage[final,notref,notcite]{showkeys}
\usepackage{subfigure}
\usepackage{multirow}
\usepackage[ruled,vlined]{algorithm2e}

\usepackage{float}
\usepackage[font=footnotesize]{caption}

\newcommand{\bm}[1]{\boldsymbol{#1}}

\newcommand{\st}{\text{ subject to }}

\newcommand{\mbfa}{\mathbf{A}}
\newcommand{\mbfb}{\mathbf{B}}
\newcommand{\mbfc}{\mathbf{C}}

\newcommand{\mbfg}{\mathbf{G}}
\newcommand{\mbfh}{\mathbf{H}}
\newcommand{\mbfi}{\mathbf{I}}
\newcommand{\mbfm}{\mathbf{M}}
\newcommand{\mbfq}{\mathbf{Q}}
\newcommand{\mbfr}{\mathbf{R}}

\newcommand{\mbfu}{\mathbf{U}}
\newcommand{\mbfv}{\mathbf{V}}
\newcommand{\mbfx}{\mathbf{X}}
\newcommand{\mbfy}{\mathbf{Y}}
\newcommand{\mbfz}{\mathbf{Z}}

\newcommand{\mbr}{\mathbb{R}}

\newcommand{\mcb}{\mathcal{B}}
\newcommand{\mcc}{\mathcal{C}}

\newcommand{\mck}{\mathcal{K}}

\newcommand{\mcm}{\mathcal{M}}
\newcommand{\mcn}{\mathcal{N}}
\newcommand{\mcp}{\mathcal{P}}

\newcommand{\mcr}{\mathcal{R}}

\newcommand{\mct}{\mathcal{T}}

\newcommand{\mcw}{\mathcal{W}}
\newcommand{\mcx}{\mathcal{X}}
\newcommand{\mcy}{\mathcal{Y}}
\newcommand{\mcz}{\mathcal{Z}}

\newcount\refnum\refnum=0
\def\myref{{\global\advance\refnum by 1} {\bf \large Lecture \the \refnum. }}
\newcommand{\bfx}{{\bf x}}

\newcommand{\bfy}{{\bf y}}

\newcommand{\unfold}{\mathbf{unfold}}
\newcommand{\fold}{\mathbf{fold}}
\newcommand{\fit}{\mathbf{fit}}

\DeclareMathOperator*{\argmin}{argmin}
\DeclareMathOperator*{\argmax}{argmax}

\usepackage[normalem]{ulem}

\title[Parallel matrix factorization for low-rank tensor completion]
      {Parallel matrix factorization for low-rank tensor completion}

\author[Y. Xu, R. Hao, W. Yin and Z. Su]{}


 \email{yangyang.xu@rice.edu}
\email{haoruru.math@gmail.com}
 \email{wotaoyin@math.ucla.edu}
 \email{zxsu@dlut.edu.cn}

\begin{document}

\maketitle

\centerline{\scshape Yangyang Xu}
\medskip
{\footnotesize
 \centerline{ Department of Computational and Applied Mathematics, Rice University}
   \centerline{ Houston, TX 77005, USA}
} 

\medskip

\centerline{\scshape Ruru Hao}
\medskip
{\footnotesize
   \centerline{School of Mathematical Sciences, Dalian University of Technology, Dalian, China.}
}

\medskip

\centerline{\scshape Wotao Yin}
\medskip
{\footnotesize
 \centerline{Department of Mathematics, University of California, Los Angeles, CA 90095, USA}
}

\medskip

\centerline{\scshape Zhixun Su}
\medskip
{\footnotesize
 \centerline{School of Mathematical Sciences, Dalian University of Technology, Dalian, China.}
}

\bigskip


\begin{abstract}
Higher-order low-rank tensors naturally arise in many applications including hyperspectral data recovery, video inpainting, seismic data reconstruction, and so on. 
We propose a new model to recover a low-rank tensor by simultaneously performing low-rank matrix factorizations to the all-mode matricizations of the underlying tensor. An alternating minimization algorithm is applied to solve the model, along with two adaptive rank-adjusting strategies when the exact rank is not known. 

Phase transition plots reveal that  our algorithm can recover    a variety of synthetic low-rank tensors  from significantly fewer samples than the compared methods, which include a matrix completion method applied to tensor recovery and two state-of-the-art tensor completion methods.  Further tests on real-world data show similar advantages. Although our model is non-convex, our algorithm performs consistently throughout the tests and gives better results than the compared methods, some of which are based on convex models. In addition, subsequence convergence of our algorithm can be established in the sense that any limit point of the iterates satisfies the KKT condtions.
\end{abstract}

\section{Introduction}
\emph{Tensor} is a generalization of \emph{vector} and \emph{matrix}. A vector is a first-order (also called one-way or one-mode) tensor, and a matrix is a second-order tensor. Higher-order tensor arises in many applications such as 3D image reconstruction \cite{sauve19993d}, video inpainting \cite{patwardhan2007video}, hyperspectral data recovery \cite{li2010tensor, xing2012dictionary}, higher-order web link analysis \cite{kolda2005higher}, personalized web search \cite{sun2005cubesvd}, and seismic data reconstruction \cite{kreimer2012tensor}. In this paper, we focus on the recovery of higher-order tensors that are (exactly or approximately) low-rank and have missing entries. We dub the problem as \emph{low-rank tensor completion} (LRTC). The introduced model and algorithm can be extended in a rather straightforward way to recovering low-rank tensors from their linear measurements.

LRTC can be regarded as an extension of low-rank matrix completion \cite{candes2009exact}. To recover a low-rank tensor from its partially observed entries, one can unfold it into a matrix and apply a low-rank matrix completion algorithm such as FPCA \cite{ma2011fixed}, APGL \cite{toh2010accelerated}, LMaFit \cite{wen2012lmafit}, the alternating direction method \cite{chen2012matrix, Xu-Yin-Wen-Zhang-11}, the $\ell_q$ minimization method \cite{lai2013improved}, and so on. However, this kind of method utilizes only one mode low-rankness of the underlying tensor. We are motivated and convinced by the results \cite{liu2013tensor} that utilizing all mode low-ranknesses of the tensor gives much better performance.

Existing methods for LRTC  in \cite{liu2013tensor, gandy2011tensor} employ matrix nuclear-norm minimization and use the singular value decomposition (SVD) in their algorithms, which become very slow or even not applicable for large-scale problems. To tackle this difficulty, we apply \emph{low-rank matrix factorization} to each mode unfolding of the tensor in order to enforce low-rankness and  update the matrix factors alternatively, which is computationally much cheaper than SVD.

Our approach is non-convex, and the sizes of the matrix factors must be specified in the algorithm. Non-convexity makes it difficult for us to predict the performance of our approach in a theoretical way, and in general, the performance can vary to the choices of algorithm and starting point. We found cyclic updates of the unknown variables in the model to perform well enough. The sizes of the matrix factors dictate the rank of the recovered tensor. If they are fixed to values significantly different from the true rank, the recovery can overfit or underfit. On the other hand, during the run time of our algorithms, there are simple ways to adaptively adjust the factor sizes. 
In short, the all-mode matricizations, cyclic block minimization, and adaptive adjustment are the building blocks of our approach. 

Before introducing  our model and algorithm, we review some notation and tensor operations.

\subsection{Notation}
Following \cite{kolda2009tensor}, we use bold lower-case letters $\bfx,\bfy,\ldots$ for vectors, bold upper-case letters $\mbfx,\mbfy,\ldots$ for matrices, and bold caligraphic letters $\bm{\mcx},\bm{\mcy},\ldots$ for tensors. The $(i_1,\ldots,i_N)$-th component of an $N$-way tensor $\bm{\mcx}$ is denoted as $x_{i_1\ldots i_N}$. For $\bm{\mcx},\bm{\mcy}\in\mbr^{I_1\times\ldots\times I_N}$, we define their inner product in the same way as that for matrices, i.e.,
$$\langle\bm{\mcx},\bm{\mcy}\rangle=\sum_{i_1=1}^{I_1}\cdots\sum_{i_N=1}^{I_N}x_{i_1\ldots i_N}y_{i_1\ldots i_N}.$$
The Frobenius norm of $\bm{\mcx}$ is defined as $\|\bm{\mcx}\|_F=\sqrt{\langle\bm{\mcx},\bm{\mcx}\rangle}.$

A \emph{fiber} of $\bm{\mcx}$ is a vector obtained by fixing all indices of $\bm{\mcx}$ except one, and a \emph{slice} of $\bm{\mcx}$ is a matrix by fixing all indices of $\bm{\mcx}$ except two. For example, if $\bm{\mcx}\in\mbr^{2\times2\times2}$ has two frontal slices (with the third index fixed)
\begin{equation}\label{eq:example}
\mbfx(:,:,1)=\begin{bmatrix}
1 & 3\\
2 & 4
\end{bmatrix},\quad
\mbfx(:,:,2)=\begin{bmatrix}
5 & 7\\
6 & 8
\end{bmatrix},
\end{equation}
then $[1,2]^\top$ is a mode-1 fiber (with all but the first indices fixed), $[1,3]^\top$ is a mode-2 fiber (with all but the second indices fixed), and $[1,5]^\top$ is a mode-3 fiber (with all but the third indices fixed). Its two horizontal (with the first index fixed) and two lateral slices (with the second index fixed) are respectively
$$
\mbfx(1,:,:)=\begin{bmatrix}
1 & 5\\
3 & 7
\end{bmatrix},
\mbfx(2,:,:)=\begin{bmatrix}
2 & 6\\
4 & 8
\end{bmatrix},\ \text{ and }\
\mbfx(:,1,:)=\begin{bmatrix}
1 & 5\\
2 & 6
\end{bmatrix},
\mbfx(:,2,:)=\begin{bmatrix}
3 & 7\\
4 & 8
\end{bmatrix}.
$$

The mode-$n$ \emph{matricization} (also called \emph{unfolding}) of  $\bm{\mcx}\in\mbr^{I_1\times\ldots\times I_N}$ is denoted as $\mbfx_{(n)}\in\mbr^{I_n\times \Pi_{j\neq n}I_j}$, which is a matrix with columns being the mode-$n$ fibers of $\bm{\mcx}$ in the lexicographical order. Take the tensor in \eqref{eq:example} for example. Its mode-1 and mode-3 matricizations are respectively
$$X_{(1)}=\begin{bmatrix}
1 & 3 & 5 & 7\\
2 & 4 & 6 & 8
\end{bmatrix},\ \text{ and }\
X_{(3)}=\begin{bmatrix}
1 & 2 & 3 & 4\\
5 & 6 & 7 & 8
\end{bmatrix}.$$
Relating to the matricization process, we define $\unfold_n(\bm{\mcx})=\mbfx_{(n)}$ and $\fold_n$ to reverse the process, i.e., $\fold_n(\unfold_n(\bm{\mcx}))=\bm{\mcx}$.
The $n$-rank of an $N$-way tensor $\bm{\mcx}$, denoted as $\rank_n(\bm{\mcx})$, is the rank of $\mbfx_{(n)}$, and we define the rank\footnote{Our definition relates to the Tucker decomposition \cite{tucker1966some}. Another popularly used definition is based on the CANDECOMP/PARAFAC (CP) decomposition \cite{kiers2000towards}.} of $\bm{\mcx}$ as an array: $\rank(\bm{\mcx})=(\rank(\mbfx_{(1)}),\ldots,\rank(\mbfx_{(N)}))$. We say $\bm{\mcx}$ is (approximately) low-rank if $\mbfx_{(n)}$ is (approximately) low-rank for all $n$.

\subsection{Problem formulation}
We aim at recovering an (approximately) low-rank tensor $\bm{\mcm}\in\mbr^{I_1\times\ldots\times I_N}$ from partial observations $\bm{\mcb}=\mcp_\Omega(\bm{\mcm})$, where $\Omega$ is the index set of observed entries, and $\mcp_\Omega$ keeps the entries in $\Omega$ and zeros out others. 
We  apply low-rank matrix factorization to each mode unfolding of $\bm{\mcm}$ by finding matrices $\mbfx_n\in\mbr^{I_n\times r_n},\mbfy_n\in \mbr^{r_n\times\Pi_{j\neq n}I_j}$ such that $\mbfm_{(n)}\approx\mbfx_n\mbfy_n$ for $n=1,\ldots,N$, where $r_n$ is the estimated rank, either fixed or adaptively updated. Introducing one common variable $\bm{\mcz}$ to relate these matrix factorizations, we solve the following model to recover $\bm{\mcm}$
\begin{equation}\label{eq:main}
\min_{\mbfx,\mbfy,\bm{\mcz}}\sum_{n=1}^N\frac{\alpha_n}{2}\|\mbfx_n\mbfy_n-\mbfz_{(n)}\|_F^2,\st \mcp_\Omega(\bm{\mcz})=\bm{\mcb},
\end{equation}
where $\mbfx=(\mbfx_1,\ldots,\mbfx_N)$ and $\mbfy=(\mbfy_1,\ldots,\mbfy_N)$. In the model, $\alpha_n$, $n=1,\ldots,N$, are weights and satisfy $\sum_n\alpha_n=1$. The constraint $\mcp_\Omega(\bm{\mcz})=\bm{\mcb}$ enforces consistency with the observations
and can be replaced with $\|\mcp_\Omega(\bm{\mcz})-\bm{\mcb}\|_F\le \delta$ if $\bm{\mcb}$ is contaminated by noise with a known Frobenius norm equal to $\delta$. 
In this paper, we do not assume the knowledge of $\delta$ and thus use \eqref{eq:main} for both noiseless and noisy cases.

The ranks $r_1,\ldots,r_N$ in \eqref{eq:main} must be specified, yet we do not assume the knowledge of their true values. To address this issue, we dynamically adjust the rank estimates in two schemes. One scheme starts from overestimated ranks and then decreases them by checking the singular values of the factor matrices in each mode. When a large gap between the $\hat{r}_n$th and  $(\hat{r}_n+1)$th singular values of the factors is found, $r_n$ is reduced to $\hat{r}_n$.   The other scheme starts from underestimated ranks and then gradually increases them if the algorithm detects slow progress.

We try to solve \eqref{eq:main} by cyclically updating $\mbfx$, $\mbfy$, and $\bm{\mcz}$. Although a global solution is not guaranteed, we demonstrate by numerical experiments that our algorithm can reliably recover a wide variety of low-rank tensors. In addition, we show that any limit point of the iterates satisfies the KKT conditions. 

The details will be given in Section \ref{sec:alg}.

\subsection{Related work} Our model \eqref{eq:main} can be regarded as an extension of the following model \cite{wen2012lmafit} from matrix completion to tensor completion
\begin{equation}\label{eq:lmafit}
\min_{\mbfx,\mbfy}\frac{1}{2}\|\mbfx\mbfy-\mbfz\|_F^2,\st \mcp_\Omega(\mbfz)=\mbfb,
\end{equation}
where $\mbfb=\mcp_\Omega(\mbfm)$ contains partially observed entries of the underlying (approximately) low-rank matrix $\mbfm$. If $N=2$ in \eqref{eq:main}, i.e., the underlying tensor $\bm{\mcm}$ is two-way, then it is easy to see that \eqref{eq:main} reduces to \eqref{eq:lmafit} by noting $\unfold_1(\bm{\mcm})=\unfold_2(\bm{\mcm})^\top$. The problem \eqref{eq:lmafit} is solved in \cite{wen2012lmafit} by a successive over-relaxation (SOR) method, named as LMaFit. Although \eqref{eq:lmafit} is non-convex, extensive experiments on both synthetic and real-world data demonstrate that \eqref{eq:lmafit} solved by LMaFit performs significantly better than nuclear norm\footnote{The matrix nuclear norm is the convex envelope of matrix rank function \cite{recht2010guaranteed}, and the nuclear norm minimization can promote the low-rank structure of the solution.} based convex models such as
\begin{equation}\label{eq:fpca}
\min_\mbfz \|\mbfz\|_*, \st \mcp_\Omega(\mbfz)=\mbfb,
\end{equation}
where $\|\mbfz\|_*$ denotes the nuclear norm of $\mbfz$, defined as the sum of its singular values.

The work \cite{liu2013tensor} generalizes \eqref{eq:fpca} to the tensor case, and to recover the (approximately) low-rank tensor $\bm{\mcm}$, it proposes to solve \begin{equation}\label{eq:conv}
\min_{\bm{\mcz}}\sum_{n=1}^N\alpha_n\|\mbfz_{(n)}\|_*, \st \mcp_\Omega(\bm{\mcz})=\bm{\mcb},
\end{equation}
where $\alpha_n\ge0,n=1,\ldots,N$ are preselected weights satisfying $\sum_n\alpha_n=1$. Different from our model \eqref{eq:main}, the problem \eqref{eq:conv} is convex, and in \cite{liu2013tensor}, various methods are applied to solve it such as block coordinate descent method, proximal gradient method, and alternating direction method of multiplier (ADMM).
The model \eqref{eq:conv} utilizes low-rankness of all mode unfoldings of the tensor, and as demonstrated in \cite{liu2013tensor}, it can significantly improve the solution quality over that obtained by solving \eqref{eq:fpca}, where the matrix $\mbfz$ corresponds to some mode unfolding of the tensor.

The recent work \cite{mu2013square} proposes a more ``square'' convex model for recovering $\bm{\mcm}$ as follows:
\begin{equation}\label{eq:square}
\min_{\bm{\mcz}}\|\hat{\mbfz}_{[j]}\|_*, \st \mcp_\Omega(\bm{\mcz})=\bm{\mcb},
\end{equation}
where $\hat{\bm{\mcz}}\in\mbr^{I_{i_1}\times\ldots\times I_{i_N}}$ is a tensor by relabeling mode $i_n$ of $\bm{\mcz}$ to mode $n$ for $n=1,\ldots,N$,
$$\hat{\mbfz}_{[j]}=\text{reshape}\left({\hat{\mbfz}_{(1)},\prod_{n\le j}I_{i_n},\prod_{n>j}I_{i_n}}\right),$$ and $j$ and the permutation $(i_1,\ldots,i_N)$ are chosen to make $\prod_{n\le j}I_{i_n}$ as close as to $\prod_{n>j}I_{i_n}$. The idea of reshaping a tensor into a ``square'' matrix has also appeared in \cite{jiang2014tensor} for tensor principal component analysis. As the order of $\bm{\mcm}$ is no more than three, \eqref{eq:square} is the same as \eqref{eq:fpca} with $\mbfz$ corresponding to some mode unfolding of the tensor, and it may not perform as well as \eqref{eq:conv}. However, for a low-rank tensor of more than three orders, it is shown in \cite{mu2013square} that \eqref{eq:square} can exactly recover the tensor from far fewer observed entries than those required by \eqref{eq:conv}.


There are some other models proposed recently for LRTC. For example, the one in \cite{romera2013new} uses, as a regularization term, a tight convex relaxation of the average rank function $\frac{1}{N}\sum_n\rank_n(\bm{\mcm})$ and applies the ADMM method to solve the problem.  The work \cite{kressner2013low} directly constrains the solution in some low-rank manifold and employs the Riemannian optimization to solve the problem. Different from the above discussed models that use tensor $n$-rank, the model in \cite{zhang2013novel} employs the so-called \emph{tubal-rank} based on the recently proposed  tensor singular value decomposition (t-SVD) \cite{kilmer2013third}. For details about these models, we refer the readers to the papers where they are proposed.

\subsection{Organization} The rest of the paper is organized as follows. Section \ref{sec:phase} shows the phase transition of our proposed method and some existing ones. Section \ref{sec:alg} gives our algorithm with two different rank-adjusting strategies, and the convergence result of the algorithm is given in section \ref{sec:convg}. In section \ref{sec:numerical}, we compare the proposed method with some state-of-the-art methods for tensor completion on both synthetic and real-world data. Section \ref{sec:conclusion} conludes the paper, and finally, section \ref{sec:fig-table} shows all figures and tables of our numerical results.

\section{Phase transition plots}\label{sec:phase}
A phase transition plot uses greyscale colors to depict how likely a certain kind of low-rank tensors can be recovered by an algorithm for a range of different ranks and sample ratios. Phase transition plots are important means to compare  the performance of different tensor recovery methods.

We compare our method (called TMac) to the following three methods on random tensors of different kinds. In section \ref{sec:numerical}, we compare them on the real-world data including 3D images and videos.

\begin{itemize}
\item Matrix completion method for recovering low-rank tensors: we unfold the underlying $N$-way tensor $\bm{\mcm}$ along its $N$th mode and apply LMaFit \cite{wen2012lmafit} to \eqref{eq:lmafit}, where $\mbfz$ corresponds to $\unfold_N(\bm{\mcm})$. If the output is $(\tilde{\mbfx},\tilde{\mbfy})$, then we use $\fold_N(\tilde{\mbfx}\tilde{\mbfy})$ to estimate $\bm{\mcm}$.
\item Nuclear norm minimization method for tensor completion: we apply FaLRTC \cite{liu2013tensor} to \eqref{eq:conv} and use the output $\bm{\mcz}$ to estimate $\bm{\mcm}$.
\item Square deal method: we apply FPCA \cite{ma2011fixed} to \eqref{eq:square} and use the output $\bm{\mcz}$ to estimate $\bm{\mcm}$.
\end{itemize}
We call the above three methods as MatComp, FaLRTC, and SquareDeal, respectively. We chose these methods due to their popularity and code availability. LMaFit has been demonstrated superior over many other matrix completion solvers such as APGL \cite{toh2010accelerated}, SVT \cite{MC-SVT2008}, and FPCA \cite{ma2011fixed}; \eqref{eq:conv} appears to be the first convex model for tensor completion, and FaLRTC is the first efficient and also reliable\footnote{In \cite{liu2013tensor}, the ADMM  is also coded up for solving \eqref{eq:conv} and claimed to give high accurate solutions. However, we found that it was not as reliable as FaLRTC. In addition, if the smoothing parameter $\mu$ for FaLRTC was set small, FaLRTC could also produce solutions of high accuracy.} solver of \eqref{eq:conv}; the work \cite{mu2013square} about SquareDeal appears the first work to give theoretical guarantee for low-rank higher-order tensor completion. We set the stopping tolerance to $10^{-5}$ for all algorithms except FPCA that uses $10^{-8}$ since $10^{-5}$ appears too loose for FPCA. Note that the tolerances used here are tighter than those in section \ref{sec:numerical} because we care more about the models' recoverability instead of algorithms' efficiency.

If the relative error $$\text{relerr}=\frac{\|\bm{\mcm}^{rec}-\bm{\mcm}\|_F}{\|\bm{\mcm}\|_F}\le 10^{-2},$$ the recovery was regarded as successful, where $\bm{\mcm}^{rec}$ denotes the recovered tensor.


\subsection{Gaussian random data}\label{sec:randtest}

Two Gaussian random datasets were tested. Each tensor in the first dataset was 3-way and had the form $\bm{\mcm}=\bm{\mcc}\times_1\mbfa_1\times_2\mbfa_2\times_3\mbfa_3$, where $\bm{\mcc}$ was generated by MATLAB command \verb|randn(r,r,r)| and $\mbfa_n$ by \verb|randn(50,r)| for $n=1,2,3$. We generated $\Omega$ uniformly at random. The rank $r$ varies from 5 to 35 with increment 3 and the sample ratio $$\text{SR}=\frac{|\Omega|}{\Pi_n I_n}$$ from 10\% to 90\% with increment 5\%. In the second dataset, each tensor was 4-way and had the form $\bm{\mcm}=\bm{\mcc}\times_1\mbfa_1\times_2\mbfa_2\times_3\mbfa_3\times_4\mbfa_4$,  where $\bm{\mcc}$ was generated by MATLAB command \verb|randn(r,r,r,r)| and $\mbfa_n$ by \verb|randn(20,r)| for $n=1,2,3,4$. The rank $r$ varies from 4 to 13 with increment 1 and SR from 10\% to 90\% with increment 5\%.  For each setting, 50 independent trials were run.

Figure \ref{fig:rate-3G} depicts the phase transition plots of TMac, MatComp, and FaLRTC for the 3-way dataset, and Figure \ref{fig:rate-4G} depicts the phase transition plots of TMac and SquareDeal for the 4-way dataset. Since LMaFit usually works better than FPCA for matrix completion, we also show the result by applying LMaFit to
\begin{equation}\label{eq:sqm-lmafit}
\min_{\mbfx,\mbfy,\bm{\mcz}} \|\mbfx\mbfy-\hat{\mbfz}_{[j]}\|_F^2, \st \mcp_\Omega(\bm{\mcz})=\bm{\mcb},
\end{equation}
where $\hat{\mbfz}_{[j]}$ is the same as that in \eqref{eq:square}. From the figures, we see that TMac performed much better than all the other compared methods. Note that our figure (also in Figures \ref{fig:4way-rand} and \ref{fig:4way-plaw}) for SquareDeal looks different from that shown in \cite{mu2013square}, because we fixed the dimension of $\bm{\mcm}$ and varied the rank $r$ while \cite{mu2013square} fixes $\rank(\bm{\mcm})$ to an array of very small values and varies the dimension of $\bm{\mcm}$. The solver and the stopping tolerance also affect the results. For example, the ``square'' model \eqref{eq:sqm-lmafit} solved by LMaFit gives much better results but still worse than those given by TMac. 

In addition, Figure \ref{fig:diff-mode} depicts the phase transition plots of TMac utilizing 1, 2, 3, and 4 modes of matricization on the 4-way dataset. We see that TMac can recover more tensors as it uses more modes.

\subsection{Uniformly random data} This section tests the recoverability of TMac, MatComp, FaLRTC, and SquareDeal on two datasets in which the tensor factors have uniformly random entries. In the first dataset, each tensor had the form $\bm{\mcm}=\bm{\mcc}\times_1\mbfa_1\times_2\mbfa_2\times_3\mbfa_3$, where $\bm{\mcc}$ was generated by MATLAB command \verb|rand(r,r,r)-0.5| and $\mbfa_n$ by \verb|rand(50,r)-0.5|. Each tensor in the second dataset had the form  $\bm{\mcm}=\bm{\mcc}\times_1\mbfa_1\times_2\mbfa_2\times_3\mbfa_3\times_4\mbfa_4$, where $\bm{\mcc}$ was generated by MATLAB command \verb|rand(r,r,r,r)-0.5| and $\mbfa_n$ by \verb|rand(20,r)-0.5|.
 Figure \ref{fig:3way-rand} shows the recoverability of each method on the first dataset and Figure \ref{fig:4way-rand} on the second dataset. We see that TMac with both rank-fixing and rank-increasing strategies performs significantly better than the other compared methods.

\subsection{Synthetic data with power-law decaying singular values} This section tests the recoverability of TMac, MatComp, FaLRTC, and SquareDeal on two more difficult synthetic datasets. In the first dataset, each tensor had the form $\bm{\mcm}=\bm{\mcc}\times_1\mbfa_1\times_2\mbfa_2\times_3\mbfa_3$, where $\bm{\mcc}$ was generated by MATLAB command \verb|rand(r,r,r)| and $\mbfa_n$ by \verb|orth(randn(50,r))*diag([1:r].^(-0.5))|. Note that the core tensor $\bm{\mcc}$ has nonzero-mean entries and each factor matrix has power-law decaying singular values. This kind of low-rank tensor appears more difficult to recover compared to the previous random low-rank tensors. Each tensor in the second dataset had the form $\bm{\mcm}=\bm{\mcc}\times_1\mbfa_1\times_2\mbfa_2\times_3\mbfa_3\times_4\mbfa_4$, where $\bm{\mcc}$ was generated by MATLAB command \verb|rand(r,r,r,r)| and $\mbfa_n$ by \verb|orth(randn(20,r))*diag([1:r].^(-0.5))|. For these two datasets, TMac with rank-decreasing strategy can never decrease $r_n$ to the true rank and thus performs badly.
 Figure \ref{fig:3way-plaw} shows the recoverability of each method on the first dataset and Figure \ref{fig:4way-plaw} on the second dataset. Again, we see that TMac with both rank-fixing and rank-increasing strategies performs significantly better than the other compared methods.

\section{Algorithm}\label{sec:alg}
We apply the alternating least squares method to \eqref{eq:main}. Since the model needs an estimate of $\rank(\bm{\mcm})$, we provide two strategies to dynamically adjust the rank estimates.

\subsection{Alternating minimization} The model \eqref{eq:main}  is convex with respect to each block of the variables $\mbfx,\mbfy$ and $\bm{\mcz}$ while the other two are fixed. Hence, we cyclically update $\mbfx,\mbfy$ and $\bm{\mcz}$ one at a time. 
Let
\begin{equation}\label{eq:obj}
f(\mbfx,\mbfy,\bm{\mcz})=\sum_{n=1}^N\frac{\alpha_n}{2}\|\mbfx_n\mbfy_n-\mbfz_{(n)}\|_F^2
\end{equation}
be the objective of \eqref{eq:main}. We perform the updates as
\begin{subequations}\label{eq:update}
\begin{align}
\mbfx^{k+1}&=\argmin_\mbfx f(\mbfx,\mbfy^k,\bm{\mcz}^k),\label{update-x}\\
\mbfy^{k+1}&=\argmin_\mbfy f(\mbfx^{k+1},\mbfy,\bm{\mcz}^k),\label{update-y}\\
\bm{\mcz}^{k+1}&=\argmin_{\mcp_\Omega(\bm{\mcz})=\bm{\mcb}} f(\mbfx^{k+1},\mbfy^{k+1},\bm{\mcz}).\label{update-z}
\end{align}
\end{subequations}
Note that both \eqref{update-x} and \eqref{update-y} can be decomposed into $N$ independent least squares problems, which can be solved in parallel. The updates in \eqref{eq:update} can be explicitly written as
\begin{subequations}\label{eq:exupdate}
\begin{align}
\mbfx_n^{k+1}&=\mbfz_{(n)}^k(\mbfy_n^k)^\top\big(\mbfy_n^k(\mbfy_n^k)^\top\big)^\dagger,\ n=1,\ldots,N,\label{exupdate-x}\\
\mbfy_n^{k+1}&=\big((\mbfx_n^{k+1})^\top\mbfx_n^{k+1}\big)^\dagger(\mbfx_n^{k+1})^\top\mbfz_{(n)}^k,\ n=1,\ldots,N,\label{exupdate-y}\\
\bm{\mcz}^{k+1}&=\mcp_{\Omega^c}\left(\sum_{n=1}^N\fold_n(\mbfx_n^{k+1}\mbfy_n^{k+1})\right)+\bm{\mcb},\label{exupdate-z}
\end{align}
\end{subequations}
where $\mbfa^\dagger$ denotes the Moore-Penrose pseudo-inverse of $\mbfa$, $\Omega^c$ is the complement of $\Omega$, and we have used the fact that $\mcp_{\Omega^c}(\bm{\mcb})=\mathbf{0}$ in \eqref{exupdate-z}.

No matter how $\mbfx_n$ is computed, only the products $\mbfx_n\mbfy_n, n=1,\ldots,N$, affect $\bm{\mcz}$ and thus the recovery $\bm{\mcm}$. Hence, we shall update $\mbfx$ in the following more efficient way
\begin{equation}\label{exupdate2-x}
\mbfx_n^{k+1}=\mbfz_{(n)}^k(\mbfy_n^k)^\top,\ n=1,\ldots,N,
\end{equation}
which together with \eqref{exupdate-y} gives the same products $\mbfx_n^{k+1}\mbfy_n^{k+1},\forall n$, as those by \eqref{exupdate-x} and \eqref{exupdate-y} according to the following lemma, which is similar to Lemma 2.1 in \cite{wen2012lmafit}. We give a proof here for completeness.

\begin{lemma}\label{lem:equiv}
For any two matrices $\mbfb,\mbfc$, it holds that
\begin{equation}\label{eq:equiv}
\begin{array}{ll}
&\big(\mbfc\mbfb^\top(\mbfb\mbfb^\top)^\dagger\big)\left(\big(\mbfc\mbfb^\top(\mbfb\mbfb^\top)^\dagger\big)^\top
\big(\mbfc\mbfb^\top(\mbfb\mbfb^\top)^\dagger\big)\right)^\dagger\big(\mbfc\mbfb^\top(\mbfb\mbfb^\top)^\dagger\big)^\top \mbfc\\
=&\big(\mbfc\mbfb^\top\big)\left(\big(\mbfc\mbfb^\top\big)^\top
\big(\mbfc\mbfb^\top\big)\right)^\dagger\big(\mbfc\mbfb^\top\big)^\top \mbfc.
\end{array}
\end{equation}
\end{lemma}

\begin{proof}
Let $\mbfb=\mbfu\bm{\Sigma}\mbfv^\top$ be the compact SVD of $\mbfb$, i.e., $\mbfu^\top\mbfu =\mbfi, \mbfv^\top\mbfv=\mbfi$, and $\bm{\Sigma}$ is a diagonal matrix with all positive singular values on its diagonal. It is not difficult to verify that
$\mbfc\mbfb^\top(\mbfb\mbfb^\top)^\dagger=\mbfc\mbfv\bm{\Sigma}^{-1}\mbfu^\top$. Then
\begin{align*}
&\ \text{first line of \eqref{eq:equiv}}\\
=&\ \mbfc\mbfv\bm{\Sigma}^{-1}\mbfu^\top\big(\mbfu\bm{\Sigma}^{-1}\mbfv^\top\mbfc^\top
\mbfc\mbfv\bm{\Sigma}^{-1}\mbfu^\top\big)^\dagger\mbfu\bm{\Sigma}^{-1}\mbfv^\top\mbfc^\top
\mbfc\\
=&\ \mbfc\mbfv\bm{\Sigma}^{-1}\mbfu^\top\mbfu\bm{\Sigma}\mbfv^\top\big(\mbfc^\top
\mbfc\big)^\dagger\mbfv\bm{\Sigma}\mbfu^\top\mbfu\bm{\Sigma}^{-1}\mbfv^\top\mbfc^\top
\mbfc\\
=&\ \mbfc\mbfv\mbfv^\top\big(\mbfc^\top\mbfc\big)^\dagger\mbfv\mbfv^\top\mbfc^\top\mbfc,
\end{align*}
where we have used $(\mbfg\mbfh)^\dagger=\mbfh^\dagger\mbfg^\dagger$ for any $\mbfg,\mbfh$ of appropriate sizes in the second equality. On the other hand,
\begin{align*}
&\ \text{second line of \eqref{eq:equiv}}\\
=&\ \mbfc\mbfv\bm{\Sigma}\mbfu^\top\big(\mbfu\bm{\Sigma}\mbfv^\top\mbfc^\top
\mbfc\mbfv\bm{\Sigma}\mbfu^\top\big)^\dagger\mbfu\bm{\Sigma}\mbfv^\top\mbfc^\top
\mbfc\\
=&\ \mbfc\mbfv\bm{\Sigma}\mbfu^\top\mbfu\bm{\Sigma}^{-1}\mbfv^\top\big(\mbfc^\top
\mbfc\big)^\dagger\mbfv\bm{\Sigma}^{-1}\mbfu^\top\mbfu\bm{\Sigma}\mbfv^\top\mbfc^\top
\mbfc\\
=&\ \mbfc\mbfv\mbfv^\top\big(\mbfc^\top\mbfc\big)^\dagger\mbfv\mbfv^\top\mbfc^\top\mbfc.
\end{align*}
Hence, we have the desired result \eqref{eq:equiv}.
\end{proof}

\subsection{Rank-adjusting schemes} The problem \eqref{eq:main} requires one to specify the ranks $r_1,\ldots,r_N$. If they are fixed, then a good estimate is important for \eqref{eq:main} to perform well. Too small $r_n$'s can cause underfitting and a large recovery error whereas too large $r_n$'s can cause overfitting and large deviation to the underlying tensor $\bm{\mcm}$. Since we do not assume the knowledge of $\rank(\bm{\mcm})$, we provide two schemes to dynamically adjust the rank estimates $r_1,\ldots,r_N$. In our algorithm, we use parameter $\xi_n$ to determine which one of the two scheme to apply. If $\xi_n=-1$, the rank-decreasing scheme is applied to $r_n$; if $\xi_n=1$,  the rank-increasing scheme is applied to $r_n$; otherwise, $r_n$ is fixed to its initial value.

\subsubsection{Rank-decreasing scheme}\label{sec:dec} This scheme starts from an input overestimated rank , i.e., $r_n>\rank_n(\bm{\mcm})$. Following \cite{wen2012lmafit, lai2013improved}, we calculate the eigenvalues of $\mbfx_n^\top\mbfx_n$ after each iteration, which are assumed to be ordered as $\lambda_1^n\ge\lambda_2^n\ge\ldots\ge\lambda_{r_n}^n$. Then we compute the quotients $\bar{\lambda}_i^n=\lambda_i^n/\lambda_{i+1}^n,i=1,\ldots,r_n-1$. Suppose
$$\hat{r}_n=\argmax_{1\le i\le r_n-1}\bar{\lambda}_i.$$
If
\begin{equation}\label{eq:dec-cond}\text{gap}_n=\frac{(r_n-1)\bar{\lambda}_{\hat{r}_n}}{\sum_{i\neq \hat{r}_n}\bar{\lambda}_i}\ge 10,
\end{equation}
 which means a ``big'' gap between $\lambda_{\hat{r}_n}$ and $\lambda_{\hat{r}_n+1}$, then we reduce $r_n$ to $\hat{r}_n$. Assume that the SVD of $\mbfx_n\mbfy_n$ is $\mbfu\bm{\Sigma}\mbfv^\top$. Then we update $\mbfx_n$ to $\mbfu_{\hat{r}_n}\bm{\Sigma}_{\hat{r}_n}$ and $\mbfy_n$ to $\mbfv_{\hat{r}_n}^\top$, where $\mbfu_{\hat{r}_n}$ is a submatrix of $\mbfu$ containing $\hat{r}_n$ columns corresponding to the  largest $\hat{r}_n$ singular values, and $\bm{\Sigma}_{\hat{r}_n}$ and $\mbfv_{\hat{r}_n}$ are obtained accordingly.

We observe in our numerical experiments that this rank-adjusting scheme generally works well for exactly low-rank tensors. Because, for these tensors, $\text{gap}_n$ can be very large and easy to identify, the true rank is typically obtained after just one rank adjustment. For approximately low-rank tensors, however, a large gap may or may not exist, and when it does not, the rank overestimates will not decrease. For these tensors, the rank-increasing scheme below works better.

\subsubsection{Rank-increasing scheme}\label{sec:inc} This scheme starts an underestimated rank, i.e., $r_n\le\rank_n(\bm{\mcm})$. Following \cite{wen2012lmafit}, we increase $r_n$ to $\min(r_n+\Delta r_n,r_n^{\max})$ at iteration $k+1$ if
\begin{equation}\label{eq:inc-cond}\left|1-\frac{\big\|\bm{\mcb}-\mcp_\Omega\big(\fold_n(\mbfx_n^{k+1}\mbfy_n^{k+1})\big)\big\|_F}{\big\|\bm{\mcb}-\mcp_\Omega\big(\fold_n(\mbfx_n^{k}\mbfy_n^{k})\big)\big\|_F}\right|\le 10^{-2},
\end{equation}
which means ``slow'' progress in the $r_n$ dimensional space along the $n$-th mode. Here, $\Delta r_n$ is a positive integer, and $r_n^{\max}$ is the maximal rank estimate. Let the economy QR factorization of $(\mbfy_n^{k+1})^\top$ be $\mbfq\mbfr$. We augment $\mbfq\gets [\mbfq,\hat{\mbfq}]$ where $\hat{\mbfq}$ has $\Delta r_n$ randomly generated columns and then orthonormalize $\mbfq$. Next, we update $\mbfy_n^{k+1}$ to $\mbfq^\top$ and $\mbfx_n^{k+1}\gets [\mbfx_n^{k+1},\mathbf{0}]$, where $\mathbf{0}$ is an $I_n\times \Delta r_n$ \emph{zero} matrix\footnote{Since we update the variables in the order of $\mbfx,\mbfy,\bm{\mcz}$, appending any matrix of appropriate size after $\mbfx_n$ does not make any difference.}.

Numerically, this scheme works well not only for exactly low-rank tensors but also for approximately low-rank ones. However, for exactly low-rank tensors, this scheme causes the algorithm to run longer than the rank-decreasing scheme. Figure \ref{fig:dec-inc} shows the performance of our algorithm equipped with the two rank-adjusting schemes. As in section \ref{sec:randtest}, we randomly generated $\bm{\mcm}$ of size $50\times50\times50$ with $\rank_n(\bm{\mcm})=r,\forall n$, and we started TMac with 25\% overestimated ranks for rank-decreasing scheme and 25\% underestimated ranks for rank-increasing scheme. From Figure \ref{fig:dec-inc}, we see that TMac with the rank-decreasing scheme is better when $r$ is small while TMac with the rank-increasing scheme becomes better when $r$ is large. 
We can also see from the figure that after $r_n$ is adjusted to match $\rank_n(\bm{\mcm})$, our algorithm converges linearly.

In general, TMac with the rank-increasing scheme works no worse than it does with the rank-decreasing scheme in terms of solution quality no matter the underlying tensor is exactly low-rank or not. Numerically, we observed that if the rank is relatively low, TMac with the rank-decreasing scheme can adjust the estimated rank to the true one within just a few iterations and converges fast, while TMac with the rank-increasing scheme may need more time to get a comparable solution. However, if the underlying tensor is approximately low-rank, or it is exactly low-rank but the user concerns more on solution quality, TMac with the rank-increasing scheme is always preferred, and small $\Delta r_n$'s usually give better solutions at the cost of more running time.

\subsection{Pseudocode} The above discussions are distilled in Algorithm \ref{alg:als}. After the algorithm terminates with output $(\mbfx,\mbfy,\bm{\mcz})$, we use $\sum_{n=1}^N\alpha_n\fold_n(\mbfx_n\mbfy_n)$ to estimate the  tensor $\bm{\mcm}$, which is usually better than  $\bm{\mcz}$ when the underlying $\bm{\mcm}$ is only approximately low-rank or the observations are contaminated by noise, or both.
\begin{algorithm}\caption{Low-rank {\bf T}ensor Completion by Parallel {\bf M}atrix F{\bf ac}torization  (TMac)}\label{alg:als}
{\small
\DontPrintSemicolon
{\bf Input:} $\Omega$, $\bm{\mcb}=\mcp_\Omega(\bm{\mcm})$, and $\alpha_n\ge 0, n=1,\ldots, N$ with $\sum_{n=1}^N\alpha_n=1$.\;
{\bf Parameters:} $r_n, \Delta r_n, r_n^{\max}, \xi_n, n=1,\ldots,N$.\;
{\bf Initialization:} $(\mbfx^0,\mbfy^0,\bm{\mcz}^0)$ with $\mcp_\Omega(\bm{\mcz}^0)=\bm{\mcb}$.\;
\For{$k=0,1,\ldots$}{
$\mbfx^{k+1}\gets$ \eqref{exupdate2-x},
$\mbfy^{k+1}\gets$ \eqref{exupdate-y}, and
$\bm{\mcz}^{k+1}\gets$ \eqref{exupdate-z}.\;
\If{stopping criterion is satisfied}{
Output $(\mbfx^{k+1},\mbfy^{k+1},\bm{\mcz}^{k+1})$.
}
\For{$n=1,\ldots,N$}{
\If{$\xi_n=-1$}{
Apply rank-decreasing scheme to $\mbfx_n^{k+1}$ and $\mbfy_n^{k+1}$  in section \ref{sec:dec}.
}
\ElseIf{$\xi_n=1$}{
Apply rank-increasing scheme to $\mbfx_n^{k+1}$ and $\mbfy_n^{k+1}$  in section \ref{sec:inc}.
}
}
}
}
\end{algorithm}

\begin{remark}
In Algorithm \ref{alg:als}, we can have different rank-adjusting schemes, i.e., different $\xi_n$, for different modes. For simplicity, we  set  $\xi_n=-1$ or $\xi_n=1$ uniformly for all $n$ in our experiments.
\end{remark}

\section{Convergence analysis}\label{sec:convg}
Introducing Lagrangian multiplier $\bm{\mcw}$ for the constraint $\mcp_\Omega(\bm{\mcz})=\bm{\mcb}$, we write the Lagrangian function of \eqref{eq:main}
\begin{equation}\label{eq:lag}
L(\mbfx,\mbfy,\bm{\mcz},\bm{\mcw})=f(\mbfx,\mbfy,\bm{\mcz})-\langle\bm{\mcw},\mcp_\Omega(\bm{\mcz})-\bm{\mcb}\rangle.
\end{equation}
Letting $\bm{\mct}=(\mbfx,\mbfy,\bm{\mcz},\bm{\mcw})$ and $\nabla_{\bm{\mct}} L=0$, we have the KKT conditions
\begin{subequations}\label{eq:kkt}
\begin{align}
&(\mbfx_n\mbfy_n-\mbfz_{(n)})\mbfy_n^\top=0,\ n=1,\ldots,N,\label{kkt-x}\\
&\mbfx_n^{\top}(\mbfx_n\mbfy_n-\mbfz_{(n)})=0,\ n=1,\ldots,N,\label{kkt-y}\\
&\bm{\mcz}-\sum_{n=1}^N\alpha_n\cdot\fold_n(\mbfx_n\mbfy_n)-\mcp_\Omega(\bm{\mcw})=0,\label{kkt-z}\\
&\mcp_\Omega(\bm{\mcz})-\bm{\mcb}=0.\label{kkt-w}
\end{align}
\end{subequations}
Our main result is summarized in the following theorem.

\begin{theorem}\label{thm:main}
Suppose $\{(\mbfx^k,\mbfy^k,\bm{\mcz}^k)\}$ is a sequence generated by Algorithm \ref{alg:als} with fixed $r_n$'s and fixed positive $\alpha_n$'s. Let $\bm{\mcw}^k=\bm{\mcb}-\mcp_\Omega\big(\sum_n\alpha_n\cdot\fold_n(\mbfx_n^k\mbfy_n^k)\big)$. 
Then any limit point of $\{\bm{\mct}^k\}$ satisfies the KKT conditions in \eqref{eq:kkt}.
\end{theorem}

Our proof of Theorem \ref{thm:main} mainly follows \cite{wen2012lmafit}. The literature has work that analyzes the convergence of alternating minimization method for non-convex problems such as \cite{GrippoSciandrone2000, Tseng-01, xu-yin-multiconvex, xu-yin-nonconvex}. However, to the best of our knowledge, none of them implies our convergence result.

 Before giving the proof, we establish some important lemmas that are used to establish our main result. We begin with the following lemma, which is similar to Lemma 3.1 of \cite{wen2012lmafit}.

\begin{lemma}\label{lem:diff1}
For any matrices $\mbfa,\mbfb$ and $\mbfc$ of appropriate sizes, letting
$\tilde{\mbfa}=\mbfc\mbfb^\top, \tilde{\mbfb}=(\tilde{\mbfa}^\top\tilde{\mbfa})^\dagger\tilde{\mbfa}^\top\mbfc,$
then we have
\begin{equation}\label{eq:diff2}
\|\mbfa\mbfb-\mbfc\|_F^2-\|\tilde{\mbfa}\tilde{\mbfb}-\mbfc\|_F^2=\|\tilde{\mbfa}\tilde{\mbfb}-\mbfa\mbfb\|_F^2.
\end{equation}
\end{lemma}

\begin{proof}
From Lemma \ref{lem:equiv}, it suffices to prove \eqref{eq:diff2} by letting $\tilde{\mbfa}=\mbfc\mbfb^\top(\mbfb\mbfb^\top)^\dagger,\tilde{\mbfb}=(\tilde{\mbfa}^\top\tilde{\mbfa})^\dagger\tilde{\mbfa}^\top\mbfc$. Assume the compact SVDs of $\tilde{\mbfa}$ and $\mbfb$ are $\tilde{\mbfa}=\mbfu_{\tilde{a}}\bm{\Sigma}_{\tilde{a}}\mbfv_{\tilde{a}}^\top$ and $\mbfb=\mbfu_b\bm{\Sigma}_b\mbfv_b^\top$. Noting $\mbfb=\mbfb\mbfv_b\mbfv_b^\top$ and $\mbfb^\top(\mbfb\mbfb^\top)^\dagger\mbfb=\mbfv_b\mbfv_b^\top$, we have
\begin{equation}\label{eq1}
\tilde{\mbfa}\mbfb-\mbfa\mbfb=(\mbfc-\mbfa\mbfb)\mbfv_b\mbfv_b^\top.
\end{equation}
Similarly, noting $\tilde{\mbfa}=\mbfu_{\tilde{a}}\mbfu_{\tilde{a}}^\top\tilde{\mbfa}$ and $\tilde{\mbfa}(\tilde{\mbfa}^\top\tilde{\mbfa})^\dagger\tilde{\mbfa}^\top=\mbfu_{\tilde{a}}\mbfu_{\tilde{a}}^\top$, we have
\begin{align}
\tilde{\mbfa}\tilde{\mbfb}-\tilde{\mbfa}\mbfb
=&\tilde{\mbfa}\tilde{\mbfb}-\mbfu_{\tilde{a}}\mbfu_{\tilde{a}}^\top\tilde{\mbfa}\mbfb\nonumber\\
=&\tilde{\mbfa}\tilde{\mbfb}-\mbfu_{\tilde{a}}\mbfu_{\tilde{a}}^\top
(\tilde{\mbfa}\mbfb-\mbfa\mbfb+\mbfa\mbfb)\nonumber\\
=& \mbfu_{\tilde{a}}\mbfu_{\tilde{a}}^\top\mbfc-\mbfu_{\tilde{a}}\mbfu_{\tilde{a}}^\top
\big[(\mbfc-\mbfa\mbfb)\mbfv_b\mbfv_b^\top-\mbfa\mbfb\big]\nonumber\\
=&\mbfu_{\tilde{a}}\mbfu_{\tilde{a}}^\top(\mbfc-\mbfa\mbfb)(\mbfi-\mbfv_b\mbfv_b^\top)\label{eq2}
\end{align}
where the third equality is from \eqref{eq1}. Summing \eqref{eq1} and \eqref{eq2} gives
\begin{equation}\label{eq3}
\tilde{\mbfa}\tilde{\mbfb}-\mbfa\mbfb=(\mbfi-\mbfu_{\tilde{a}}\mbfu_{\tilde{a}}^\top)(\mbfc-\mbfa\mbfb)
\mbfv_b\mbfv_b^\top+\mbfu_{\tilde{a}}\mbfu_{\tilde{a}}^\top(\mbfc-\mbfa\mbfb).
\end{equation}
Since $(\mbfi-\mbfu_{\tilde{a}}\mbfu_{\tilde{a}}^\top)(\mbfc-\mbfa\mbfb)$ is orthogonal to $\mbfu_{\tilde{a}}\mbfu_{\tilde{a}}^\top(\mbfc-\mbfa\mbfb)$, we have
\begin{equation}\label{eq4}
\big\|\tilde{\mbfa}\tilde{\mbfb}-\mbfa\mbfb\big\|_F^2=\big\|(\mbfi-\mbfu_{\tilde{a}}\mbfu_{\tilde{a}}^\top)(\mbfc-\mbfa\mbfb)
\mbfv_b\mbfv_b^\top\big\|_F^2+\big\|\mbfu_{\tilde{a}}\mbfu_{\tilde{a}}^\top(\mbfc-\mbfa\mbfb)\big\|_F^2.
\end{equation}
In addition, note
$$\langle(\mbfi-\mbfu_{\tilde{a}}\mbfu_{\tilde{a}}^\top)(\mbfc-\mbfa\mbfb)
\mbfv_b\mbfv_b^\top,\mbfc-\mbfa\mbfb\rangle=\|(\mbfi-\mbfu_{\tilde{a}}\mbfu_{\tilde{a}}^\top)(\mbfc-\mbfa\mbfb)
\mbfv_b\mbfv_b^\top\|_F^2,$$
and
$$\langle\mbfu_{\tilde{a}}\mbfu_{\tilde{a}}^\top(\mbfc-\mbfa\mbfb),\mbfc-\mbfa\mbfb\rangle=\|\mbfu_{\tilde{a}}\mbfu_{\tilde{a}}^\top(\mbfc-\mbfa\mbfb)\|_F^2.$$
Hence,
$$\langle\tilde{\mbfa}\tilde{\mbfb}-\mbfa\mbfb,\mbfc-\mbfa\mbfb\rangle=\big\|(\mbfi-\mbfu_{\tilde{a}}\mbfu_{\tilde{a}}^\top)(\mbfc-\mbfa\mbfb)
\mbfv_b\mbfv_b^\top\big\|_F^2+\big\|\mbfu_{\tilde{a}}\mbfu_{\tilde{a}}^\top(\mbfc-\mbfa\mbfb)\big\|_F^2,$$
and thus
$\|\tilde{\mbfa}\tilde{\mbfb}-\mbfa\mbfb\|_F^2=\langle\tilde{\mbfa}\tilde{\mbfb}-\mbfa\mbfb, \mbfc-\mbfa\mbfb\rangle.$ Then \eqref{eq:diff2} can be shown by noting
$$\|\tilde{\mbfa}\tilde{\mbfb}-\mbfc\|_F^2=\|\tilde{\mbfa}\tilde{\mbfb}-\mbfa\mbfb\|_F^2+2\langle\tilde{\mbfa}\tilde{\mbfb}-\mbfa\mbfb,\mbfa\mbfb-\mbfc\rangle+\|\mbfa\mbfb-\mbfc\|_F^2.$$
This completes the proof.
\end{proof}

We also need the following lemma.
\begin{lemma}\label{lem:range}
For any two matrices $\mbfb$ and $\mbfc$ of appropriate sizes, it holds that
\begin{subequations}
\begin{align}
&\mcr_c\left(\mbfc\mbfb^\top\right)=\mcr_c\left(\mbfc\mbfb^\top(\mbfb\mbfb^\top)^\dagger\right),\label{range-col}\\
&\mcr_r\left(\big[(\mbfc\mbfb^\top)^\top
(\mbfc\mbfb^\top)\big]^\dagger
(\mbfc\mbfb^\top)^\top\mbfc\right)\cr
=&\ \mcr_r\left(\big[(\mbfc\mbfb^\top(\mbfb\mbfb^\top)^\dagger)^\top
(\mbfc\mbfb^\top(\mbfb\mbfb^\top)^\dagger)\big]^\dagger
\big[\mbfc\mbfb^\top(\mbfb\mbfb^\top)^\dagger\big]^\top\mbfc\right),\label{range-row}
\end{align}
\end{subequations}
where $\mcr_c(\mbfa)$ and $\mcr_r(\mbfa)$ denote the column and row space of $\mbfa$, respectively.
\end{lemma}

\begin{proof}
Following the proof of Lemma \ref{lem:equiv}, we assume the compact SVD of $\mbfb$ to be $\mbfb=\mbfu\bm{\Sigma}\mbfv^\top$. Then $$\mbfc\mbfb^\top=\mbfc\mbfv\bm{\Sigma}\mbfu^\top\text{ and } \mbfc\mbfb^\top(\mbfb\mbfb^\top)^\dagger=\mbfc\mbfv\bm{\Sigma}^{-1}\mbfu^\top.$$ For any vector $\bfx$ of appropriate size, there must be another vector $\bfy$ such that $\mbfu^\top\bfy=\bm{\Sigma}^2\mbfu^\top\bfx$ or equivalently $\bm{\Sigma}^{-1}\mbfu^\top\bfy=\bm{\Sigma}\mbfu^\top\bfx$. Hence $\mbfc\mbfb^\top\bfx=\mbfc\mbfb^\top(\mbfb\mbfb^\top)^\dagger\bfy$, which indicates $\mcr_c\big(\mbfc\mbfb^\top\big)\subset\mcr_c\big(\mbfc\mbfb^\top(\mbfb\mbfb^\top)^\dagger\big)$.  In the same way, one can show the reverse inclusion, and hence $\mcr_c\big(\mbfc\mbfb^\top\big)=\mcr_c\big(\mbfc\mbfb^\top(\mbfb\mbfb^\top)^\dagger\big)$.
The result \eqref{range-row} can be shown similarly by noting that
{\small\begin{align*}
\big[(\mbfc\mbfb^\top)^\top
(\mbfc\mbfb^\top)\big]^\dagger
(\mbfc\mbfb^\top)^\top\mbfc&=\mbfu\bm{\Sigma}\mbfv^\top(\mbfc^\top\mbfc)^\dagger\mbfv
\mbfv^\top\mbfc^\top\mbfc,\\
\big[(\mbfc\mbfb^\top(\mbfb\mbfb^\top)^\dagger)^\top
(\mbfc\mbfb^\top(\mbfb\mbfb^\top)^\dagger)\big]^\dagger
\big[\mbfc\mbfb^\top(\mbfb\mbfb^\top)^\dagger\big]^\top\mbfc&=\mbfu\bm{\Sigma}^{-1}\mbfv^\top(\mbfc^\top\mbfc)^\dagger\mbfv
\mbfv^\top\mbfc^\top\mbfc.
\end{align*}}
This completes the proof.
\end{proof}

According to Lemma \ref{lem:range}, it is not difficult to get the following corollary.
\begin{corollary}\label{cor:diff}
For any matrices $\mbfa,\mbfb$ and $\mbfc$ of appropriate sizes, let
$\tilde{\mbfa}=\mbfc\mbfb^\top, \tilde{\mbfb}=(\tilde{\mbfa}^\top\tilde{\mbfa})^\dagger\tilde{\mbfa}^\top\mbfc.$ If the compact SVDs of $\tilde{\mbfa}$ and $\mbfb$ are $\tilde{\mbfa}=\mbfu_{\tilde{a}}\bm{\Sigma}_{\tilde{a}}\mbfv_{\tilde{a}}^\top$ and $\mbfb=\mbfu_b\bm{\Sigma}_b\mbfv_b^\top$, then we have \eqref{eq4}.
\end{corollary}

We are now ready to prove Theorem \ref{thm:main}.

\subsection*{Proof of Theorem \ref{thm:main}}
Let $\bar{\bm{\mct}}=(\bar{\mbfx},\bar{\mbfy},\bar{\bm{\mcz}},\bar{\bm{\mcw}})$ be a limit point of $\{\bm{\mct}^k\}$, and thus there is a subsequence $\{\bm{\mct}^k\}_{k\in\mck}$ converging to $\bar{\bm{\mct}}$.
According to \eqref{update-z}, we have $\mcp_\Omega(\bm{\mcz}^k)=\bm{\mcb}$, and
$\mcp_{\Omega^c}(\bm{\mcz}^k)=\mcp_{\Omega^c}(\sum_n\alpha_n\cdot\fold_n(\mbfx_n^k\mbfy_n^k))$. Hence, \eqref{kkt-z} and \eqref{kkt-w} hold at $\bm{\mct}=\bm{\mct}^k$ for all $k$ and thus at $\bar{\bm{\mct}}$.

From Lemma \ref{lem:diff1}, it follows that
\begin{equation}\label{eq5}f(\mbfx^k,\mbfy^k,\bm{\mcz}^k)-f(\mbfx^{k+1},\mbfy^{k+1},\bm{\mcz}^k)=\sum_{n=1}^N\frac{\alpha_n}{2}\|\mbfx^k\mbfy^k-\mbfx^{k+1}\mbfy^{k+1}\|_F^2.
\end{equation}
In addition, it is not difficult to verify
\begin{equation}\label{eq6}
f(\mbfx^{k+1},\mbfy^{k+1},\bm{\mcz}^k)-f(\mbfx^{k+1},\mbfy^{k+1},\bm{\mcz}^{k+1})=\frac{1}{2}\|\bm{\mcz}^k-\bm{\mcz}^{k+1}\|_F^2.
\end{equation}
Summing up \eqref{eq5} and \eqref{eq6} and observing that $f$ is lower bounded by \emph{zero}, we have
$$\sum_{k=0}^\infty\left(\sum_{n=1}^N\frac{\alpha_n}{2}\|\mbfx^k\mbfy^k-\mbfx^{k+1}\mbfy^{k+1}\|_F^2+\frac{1}{2}\|\bm{\mcz}^k-\bm{\mcz}^{k+1}\|_F^2\right)<\infty,$$
and thus
\begin{equation}\label{eq7}
\lim_{k\to\infty}\|\mbfx^k_n\mbfy^k_n-\mbfx_n^{k+1}\mbfy_n^{k+1}\|_F^2= 0,\text{ and }
\lim_{k\to\infty}\|\bm{\mcz}^k-\bm{\mcz}^{k+1}\|_F^2=0.
\end{equation}

For each $n$ and $k$, let the compact SVDs of $\mbfx_n^k$ and $\mbfy_n^k$ be $\mbfx_n^k=\mbfu_{x_n^k}\bm{\Sigma}_{x_n^k}\mbfv_{x_n^k}^\top$ and $\mbfy_n^k=\mbfu_{y_n^k}\bm{\Sigma}_{y_n^k}\mbfv_{y_n^k}^\top$. Letting $\mbfa=\mbfx_n^k,\mbfb=\mbfy_n^k,\tilde{\mbfa}=\mbfx_n^{k+1},\tilde{\mbfb}=\mbfy_n^{k+1}$, and $\mbfc=\mbfz_{(n)}^k$ in \eqref{eq4}, we have
\begin{align*}&\ \big\|\mbfx_n^{k+1}\mbfy_n^{k+1}-\mbfx_n^k\mbfy_n^k\big\|_F^2\\
=&\ \big\|(\mbfi-\mbfu_{x_n^{k+1}}\mbfu_{x_n^{k+1}}^\top)(\mbfz_{(n)}^k-\mbfx_n^k\mbfy_n^k)
\mbfv_{y_n^k}\mbfv_{y_n^k}^\top\big\|_F^2+\big\|\mbfu_{x_n^{k+1}}\mbfu_{x_n^{k+1}}^\top(\mbfz_{(n)}^k-\mbfx_n^k\mbfy_n^k)\big\|_F^2,
\end{align*}
which together with \eqref{eq7} gives
\begin{subequations}
\begin{align}
&\lim_{k\to \infty}\big(\mbfz_{(n)}^k-\mbfx_n^k\mbfy_n^k\big)
\mbfv_{y_n^k}\mbfv_{y_n^k}^\top=0,\label{lim-x}\\
&\lim_{k\to\infty} \mbfu_{x_n^{k+1}}\mbfu_{x_n^{k+1}}^\top\big(\mbfz_{(n)}^k-\mbfx_n^k\mbfy_n^k\big)=0.\label{lim-y}
\end{align}
\end{subequations}
Since $(\mbfy_n^k)^\top=\mbfv_{y_n^k}\mbfv_{y_n^k}^\top(\mbfy_n^k)^\top$ and $\{\mbfy^k_n\}_{k\in\mck}$ is bounded, then right multiplying $(\mbfy_n^k)^\top$ for $k\in\mck$ on both sides of \eqref{lim-x} yields
$$\big(\bar{\mbfz}_{(n)}-\bar{\mbfx}_n\bar{\mbfy}_n\big)\bar{\mbfy}_n^\top=\lim_{\substack{k\to \infty\\
k\in\mck}}\big(\mbfz_{(n)}^k-\mbfx_n^k\mbfy_n^k\big)(\mbfy_n^k)^\top
=0,$$
which indicates that \eqref{kkt-x} is satisfied at $\bar{\bm{\mct}}$. From \eqref{eq7} and \eqref{lim-y}, we have
$$\lim_{k\to\infty} \mbfu_{x_n^{k}}\mbfu_{x_n^{k}}^\top\big(\mbfz_{(n)}^{k}-\mbfx_n^{k}\mbfy_n^{k}\big)=0,$$
which together with the boundedness of $\{\mbfx_n^k\}_{k\in\mck}$ and $\mbfx_n^k=\mbfu_{x_n^{k}}\mbfu_{x_n^{k}}^\top\mbfx_n^k$ gives
$$\bar{\mbfx}_n^\top\big(\bar{\mbfz}_{(n)}-\bar{\mbfx}_n\bar{\mbfy}_n\big)=
\lim_{\substack{k\to\infty\\k\in\mck}} (\mbfx_n^{k})^\top\big(\mbfz_{(n)}^{k}-\mbfx_n^{k}\mbfy_n^{k}\big)=0,$$
and thus \eqref{kkt-y} is satisfied at $\bar{\bm{\mct}}$. This completes the proof.

\section{Numerical experiments}\label{sec:numerical}
This section tests Algorithm \ref{alg:als},  TMac, for solving \eqref{eq:main}. To demonstrate its effectiveness, we compared it with MatComp and FaLRTC (see section \ref{sec:phase}) on real-world data.

\subsection{Dynamic weights and stopping rules}
The parameters $\alpha_1,\ldots,\alpha_N$ in \eqref{eq:main} were uniformly set to $\frac{1}{N}$ at the beginning of TMac. During the iterations, we either fixed them or dynamically updated them according to the fitting error $$\fit_n(\mbfx_n\mbfy_n)=\|\mcp_\Omega\big(\fold_n(\mbfx_n\mbfy_n)-\bm{\mcb}\big)\|_F.$$ The smaller $\fit_n(\mbfx_n\mbfy_n)$ is, the larger $\alpha_n$ should be. Specifically, if the current iterate is $(\mbfx^k,\mbfy^k,\bm{\mcz}^k)$, we set
\begin{equation}\label{update-a}
\alpha_n^k=\frac{\big[\fit_n(\mbfx_n^k\mbfy_n^k)\big]^{-1}}{\sum_{i=1}^N\big[\fit_i(\mbfx_i^k\mbfy_i^k)\big]^{-1}},\ n=1,\ldots,N.
\end{equation}
As demonstrated below, dynamic updating  $\alpha_n$'s can  improve the recovery quality for tensors that have better low-rankness in one mode than others.
TMac was terminated if one of the following conditions was satisfied for some $k$
\begin{align}
\frac{\left|\sum_{n=1}^N\fit_n(\mbfx_n^{k}\mbfy_n^k)-\sum_{n=1}^N\fit_n(\mbfx_n^{k+1}\mbfy_n^{k+1})\right|}{1+\sum_{n=1}^N\fit_n(\mbfx_n^{k}\mbfy_n^{k})}\le tol,\label{cond1}\\
\frac{\sum_{n=1}^N\alpha_n^k\cdot\fit_n(\mbfx_n^{k+1}\mbfy_n^{k+1})}{\|\bm{\mcb}\|_F}\le tol,\label{cond2}
\end{align}
where $tol$ is a small positive value specified below. The condition \eqref{cond1} checks the relative change of the overall fitting, and \eqref{cond2} is satisfied if the weighted fitting is good enough.

\subsection{MRI data}
This section compares TMac, MatComp, and FaLRTC on a $181\times 217\times 181$ brain MRI data, which has been used in \cite{liu2013tensor}. The data is approximately low-rank: for its three mode unfodings, the numbers of singular values larger than 0.1\% of the largest one are 28, 33, and 29, respectively. One slice of the data is shown in Figure \ref{fig:mri}. We tested all three methods on both noiseless and noisy data\footnote{For noisy case, it could be better to relax the equality constraints in \eqref{eq:main}, \eqref{eq:lmafit}, and \eqref{eq:conv} to include some information on the noise level. However, we did not assume such information, and these methods could still work reasonably.}. Specifically, we added scaled Gaussian noise to the original data to have $$\bm{\mcm}^{nois}=\bm{\mcm}+\sigma\frac{\|\bm{\mcm}\|_\infty}{\|\bm{\mcn}\|_\infty}\bm{\mcn},$$ and made noisy observations $\bm{\mcb}=\mcp_\Omega(\bm{\mcm}^{nois})$, where $\|\bm{\mcm}\|_\infty$ denotes the maximum absolute value of $\bm{\mcm}$, and the entries of $\bm{\mcn}$ follow idendically independent standard Gaussian distribution.

We ran all the algorithms to maximum 1000 iterations. The stopping tolerance was set to $tol=10^{-3}$ for TMac and LMaFit. For FaLRTC, $tol=10^{-4}$ was set since we found $10^{-3}$ was too loose. Both TMac and LMaFit used the rank-increasing strategy. For TMac, we initialized $r_n=5$ and set $\Delta r_n=3, r_n^{\max}=50, \forall n$, and for LMaFit, we set initial rank $K=5$, increment $\kappa=3$, and maximal rank $K^{\max}=50$. We tested TMac with fixed parameters $\alpha_n=\frac{1}{3},n=1,2,3,$ and also dynamically updated ones by \eqref{update-a} starting from $\alpha_n^0=\frac{1}{3},n=1,2,3$. The smoothing parameter for FaLRTC was set to its default value $\mu=0.5$ and weight parameters set to $\alpha_n=\frac{1}{3}, n=1,2,3.$
Table \ref{table:mri} shows the average relative errors and running times of five independent trials for each setting of $\sigma$ and SR. Figure \ref{fig:mri} shows one noisy masked slice and the corresponding recovered slices by different methods with the setting of $\sigma=0.05$ and $\text{SR}=10\%$. From the results, we see that TMac consistently reached lower relative errors than those by FaLRTC and cost less time. MatComp used the least time and could achieve low relative error as SR is large. However, for low SR's (e.g., SR=10\%), it performed extremely bad, and even we ran it to more iterations, say 5000, it still performed much worse than TMac and FaLRTC. In addition, TMac using fixed $\alpha_n$'s worked similarly well as that using dynamically updated ones, which should be because the data has similar low-rankness along each mode.

\subsection{Hyperspectral data}
This section compares TMac, MatComp, and FaLRTC on a $205\times246\times96$ hyperspectral data, one slice of which is shown in Figure \ref{fig:hyp}. This data is also approximately low-rank: for its three mode unfoldings, the numbers of singular values larger than 1\% of the largest one are 19,19, and 4, respectively. However, its low-rank property is not as good as that of the above MRI data. Its numbers of singular values larger than 0.1\% of the largest one are 198, 210, and 18. Hence, its mode-3 unfolding has better low-rankness, and we assigned larger weight to the third mode. For FaLRTC, we set $\alpha_1=\alpha_2=0.25,\alpha_3=0.5$. For TMac, we tested it with fixed weights $\alpha_1=\alpha_2=0.25,\alpha_3=0.5$ and also with dynamically updated ones by \eqref{update-a} starting from $\alpha_1^0=\alpha_2^0=0.25,\alpha_3^0=0.5$. All other parameters of the three methods were set as the same as those used in the previous test.

For each setting of $\sigma$ and SR, we made 5 independent runs. Table \ref{table:hyp} reports the average results of each tested method, and Figure \ref{fig:hyp} shows one noisy slice with 90\% missing values and 5\% Gaussian noise, and the corresponding recovered slices by the compared methods. From the results, we see again that TMac outperformed FaLRTC in both solution quality and running time, and MatComp gave the largest relative errors at the cost of least time. In addition, TMac with dynamically updated $\alpha_n$'s worked better than that with fixed $\alpha_n$'s as SR was relatively large (e.g., SR=30\%, 50\%), and this should be because the data has much better low-rankness along the third mode than the other two. As SR was low (e.g., SR=10\%), TMac with dynamically updated $\alpha_n$'s performed even worse. This could be explained by the case that all slices might have missing values at common locations (i.e., some mode-3 fibers were entirely missing) as SR was low, and in this case,  the third mode unfolding had some entire columns missing. In general, it is impossible for any matrix completion solver to recover an entire missing column or row of a matrix, and thus putting more weight on the third mode could worsen the recovery. That also explains why MatComp gave much larger relative errors than those by FaLRTC but the slice recovered by MatComp looks better than that by FaLRTC in Figure \ref{fig:hyp}. Note that there are lots of black points on the slice given by MatComp, and these black points correspond to missing columns of the third mode unfolding. Therefore, we do not recommend to dynamically update $\alpha_n$'s in TMac when SR is low or some fibers are entirely missing.

\subsection{Video inpainting}
In this section, we compared TMac, MatComp, and FaLRTC on both grayscale and color videos. The grayscale video\footnote{http://www.ugcs.caltech.edu/$\sim$srbecker/escalator\_data.mat} has 200 frames with each one of size $130\times 160$, and the color video\footnote{http://media.xiph.org/video/derf/ The original video has 500 frames, and we used its first 150 frames in our test.} has 150 frames with each one of size $144\times 176$. We treated the grayscale video as a $130\times 160\times 200$ tensor and the color video as three $144\times176\times150$ tensors\footnote{One can also treat the color video as a 4th-order tensor. However, we found that recovering a 4th-order tensor cost much more time than recovering three 3rd-order tensors and made no quality improvement.}, one for each channel. For the grayscale video, the numbers of singular values larger than 1\% of the largest one for each mode unfolding are 79, 84, and 35. Hence, its rank is not low, and it is relatively difficult to recover this video. The color video has lower rank. For each of its three channel tensors, the numbers of singular values larger than 1\% of the largest one are about\footnote{More precisely, the numbers are $(51,53,24)$, $(48,51,24)$, and $(49,52,24)$ respectively for three channel tensors.} 50, 50, and 24. During each run, the three channel tensors had the same index set of observed entries, which is the case in practice, and we recovered each channel independently. We set $\alpha_n=\frac{1}{3},n=1,2,3$ for FaLRTC and TMac, while the latter was tested with both fixed $\alpha_n$'s and dynamically updated one by \eqref{update-a}. All the other parameters of the test methods were set as the same as those in the previous test. The average results of 5 independent runs were reported in Table \ref{table:grayvideo} for the grayscale video and Table \ref{table:colorvideo} for the color video. Figure \ref{fig:grayvideo} shows one frame of recovered grayscale video by each method and Figure \ref{fig:colorvideo} one frame of recovered color video. From the tables, we see that the comparisons are similar to those for the previous hyperspectral data recovery.

\section{Discussions}\label{sec:conclusion}
We have proposed a new method for low-rank tensor completion. Our model utilizes low-rank matrix factorizations to all-mode unfoldings of the tensor. Synthetic data tests demonstrate that our model can recover significantly more low-rank tensors than two nuclear norm based models and one model that performs low-rank matrix factorization to only one mode unfolding. In addition, numerical results on 3D images and videos show that our method consistently produces the best solutions among all compared methods and outperforms the nuclear norm minimization method in both solution quality and running time.

Numerically, we have observed that our algorithm converges fast (e.g., linear convergence in Figure \ref{fig:dec-inc}). Papers \cite{Ling-Xu-Yin-mtx-conf-11, wen2012lmafit} demonstrate that the SOR technique can significantly accelerate the progress of alternating least squares. However, we did not observe any acceleration applying the same technique to our algorithm. In the future, we will explore the reason and try to develop other techniques to accelerate our algorithm. We also plan to incorporate the objective term in \eqref{eq:sqm-lmafit} to enrich \eqref{eq:main}, if the underlying low-rank tensor has more than three orders. 

\section{Figures and tables}\label{sec:fig-table} 
We give all figures and tables of our numerical results in this section.
\newpage
 
\begin{figure}[H]
\captionsetup{width=0.99\textwidth}
\caption{Phase transition plots for different methods on \textbf{3-way low-rank tensors whose factors have Gaussian random entries}. (a) FaLRTC: the tensor completion method in \cite{liu2013tensor} that solves \eqref{eq:conv}. (b) MatComp: the matrix completion solver LMaFit in \cite{wen2012lmafit} that solves \eqref{eq:lmafit} with $\mbfz$ corresponding to $\mbfm_{(N)}$. (c) TMac-fix: Algorithm \ref{alg:als} solves \eqref{eq:main} with $\alpha_n=\frac{1}{3}$ and $r_n$ fixed to $r$, $\forall n$. (d) TMac-inc: Algorithm \ref{alg:als} solves \eqref{eq:main} with $\alpha_n=\frac{1}{3},\forall n$ and using rank-increasing strategy starting from $r_n=\text{round}(0.75r),\forall n$. (e) TMac-dec: Algorithm \ref{alg:als} solves \eqref{eq:main} with $\alpha_n=\frac{1}{3},\forall n$ and using rank-decreasing strategy starting from $r_n=\text{round}(1.25r),\forall n$.}\label{fig:rate-3G}
\centering
\subfigure[FaLRTC]{
\begin{minipage}[t]{0.25\textwidth}
\centering
\includegraphics[width=0.99\textwidth]{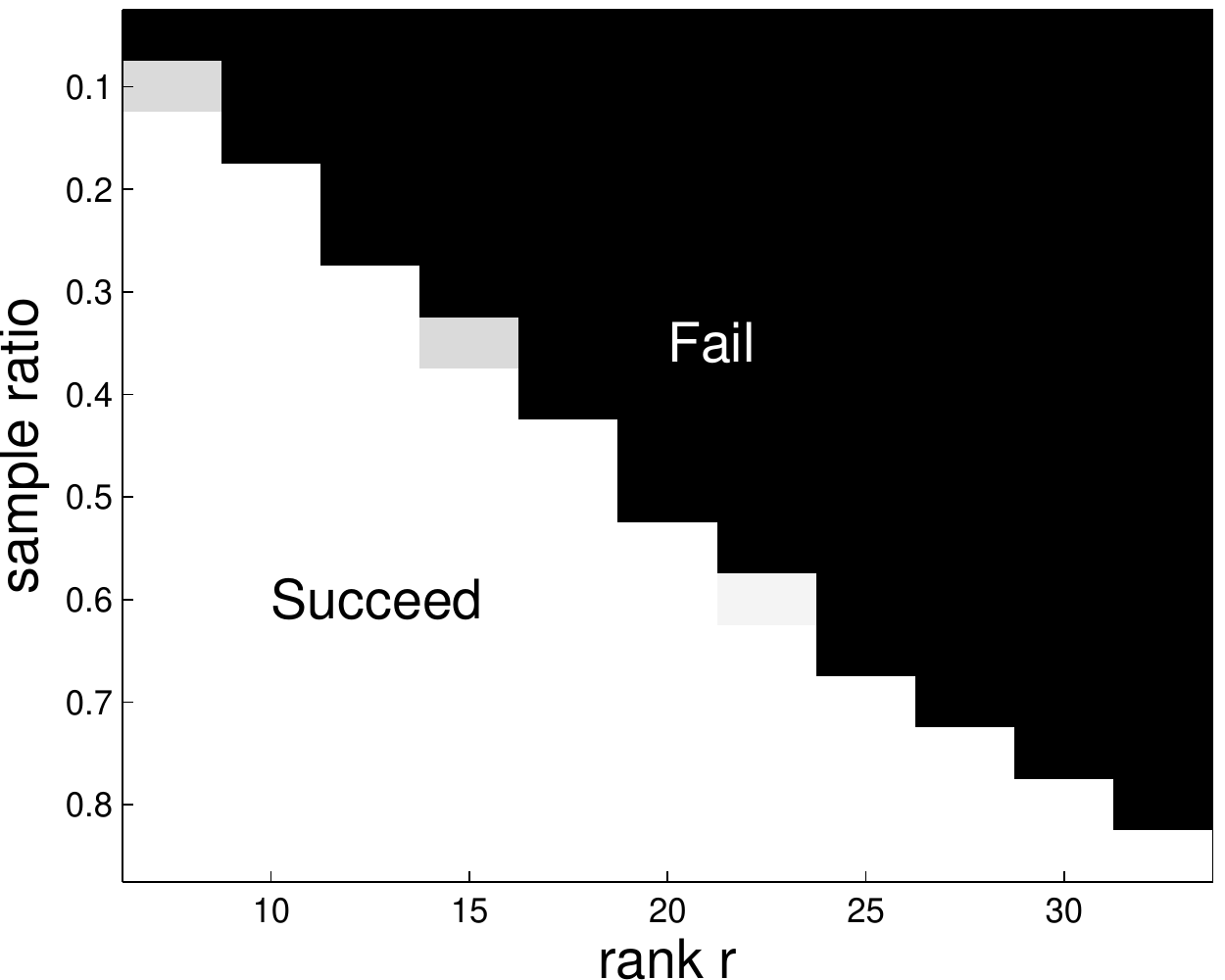}
\end{minipage}
}
\subfigure[MatComp]{
\begin{minipage}[t]{0.25\textwidth}
\centering
\includegraphics[width=0.99\textwidth]{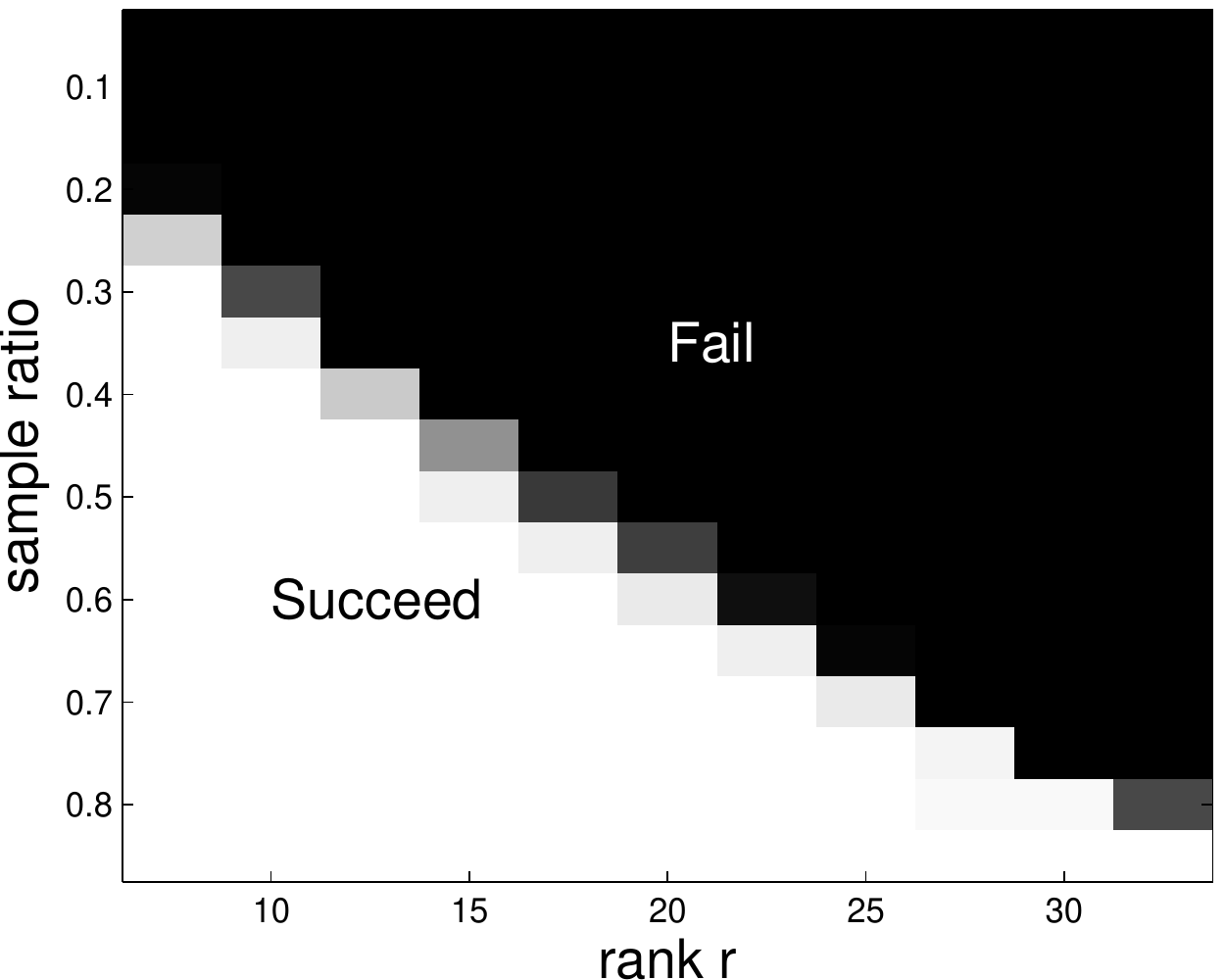}
\end{minipage}
}\\
\subfigure[TMac-fix]{
\begin{minipage}[t]{0.25\textwidth}
\centering
\includegraphics[width=0.99\textwidth]{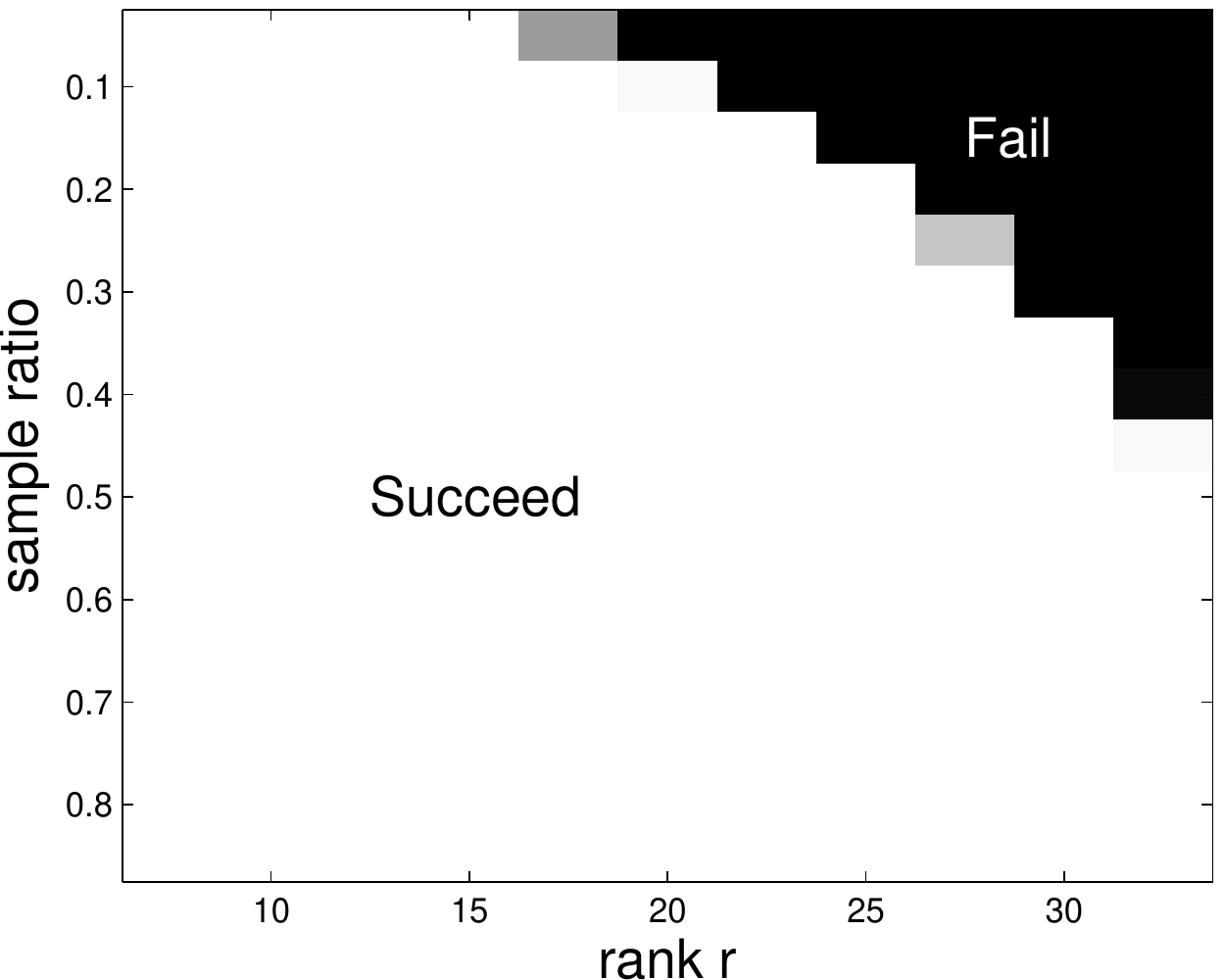}
\end{minipage}
}
\subfigure[TMac-inc]{
\begin{minipage}[t]{0.25\textwidth}
\centering
\includegraphics[width=0.99\textwidth]{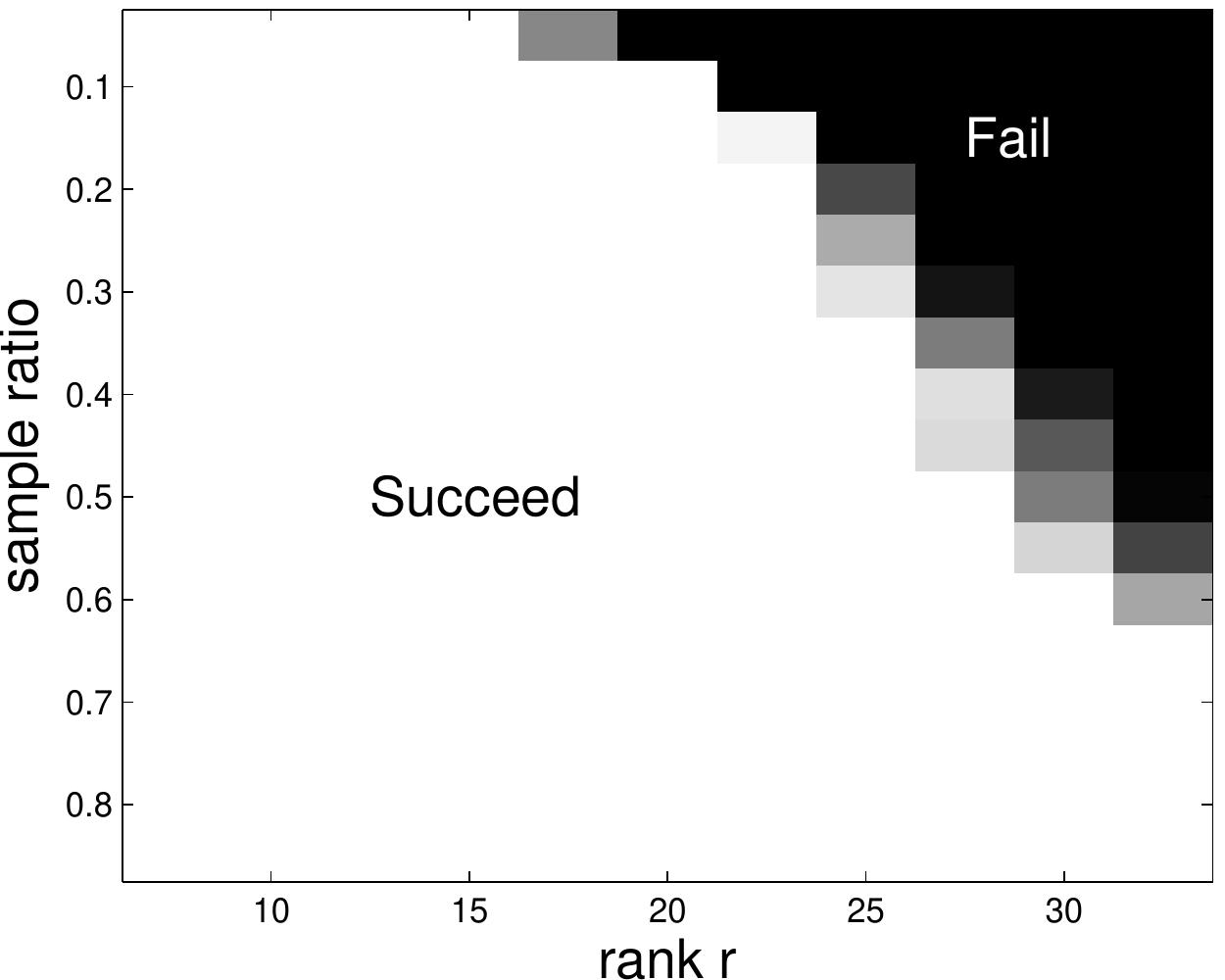}
\end{minipage}
}
\subfigure[TMac-dec]{
\begin{minipage}[t]{0.25\textwidth}
\centering
\includegraphics[width=0.99\textwidth]{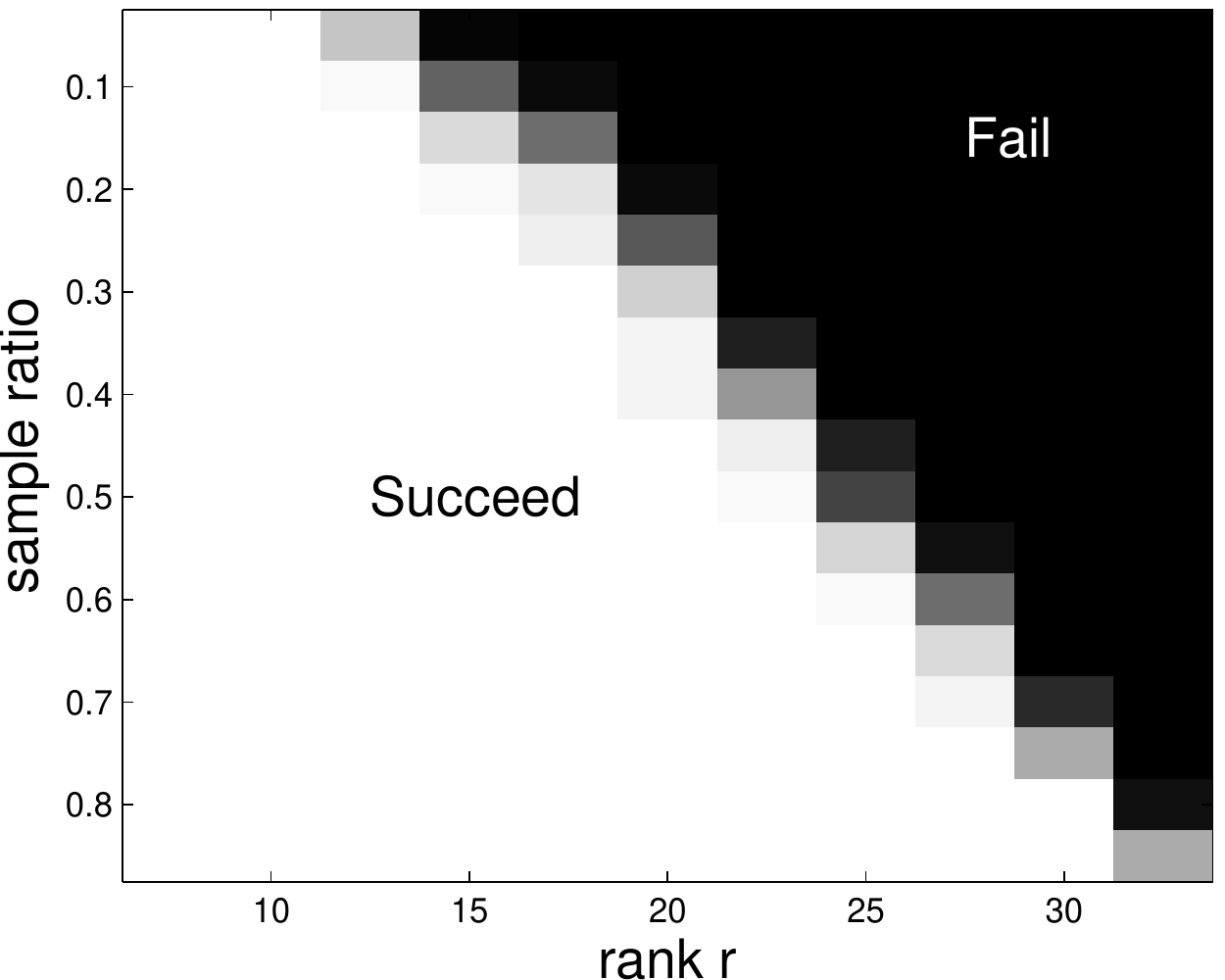}
\end{minipage}
}
\end{figure}

\begin{figure}[H]
\captionsetup{width=0.99\textwidth}
\caption{Phase transition plots for different methods on \textbf{4-way low-rank tensors whose factors have Gaussian random entries}. (a) SquareDeal: the matrix completion solver FPCA \cite{ma2011fixed} solves the model \eqref{eq:square} proposed in \cite{mu2013square}. (b) the matrix completion solver LMaFit \cite{wen2012lmafit} solves \eqref{eq:sqm-lmafit}, which is a non-convex variant of \eqref{eq:square}. (c) TMac-fix: Algorithm \ref{alg:als} solves \eqref{eq:main} with $\alpha_n=\frac{1}{4}$ and $r_n$ fixed to $r$, $\forall n$. (d) TMac-inc: Algorithm \ref{alg:als} solves \eqref{eq:main} with $\alpha_n=\frac{1}{4},\forall n$ and using rank-increasing strategy starting from $r_n=\text{round}(0.75r),\forall n$. (e) TMac-dec: Algorithm \ref{alg:als} solves \eqref{eq:main} with $\alpha_n=\frac{1}{4},\forall n$ and using rank-decreasing strategy starting from $r_n=\text{round}(1.25r),\forall n$.}\label{fig:rate-4G}
\centering
\subfigure[SquareDeal]{
\begin{minipage}[t]{0.25\textwidth}
\centering
\includegraphics[width=0.99\textwidth]{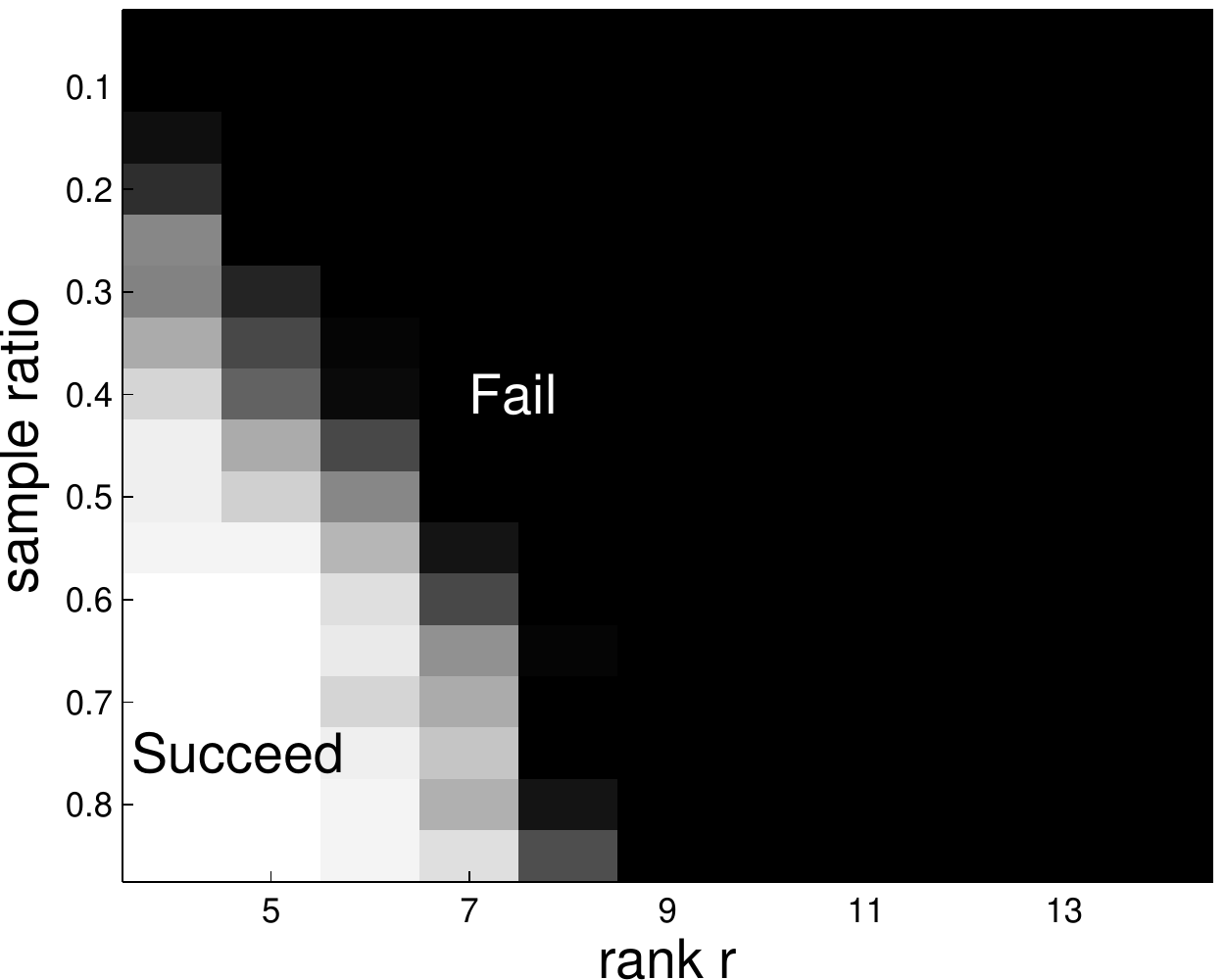}
\end{minipage}
}
\subfigure[\eqref{eq:sqm-lmafit} solved by LMaFit]{
\begin{minipage}[t]{0.25\textwidth}
\centering
\includegraphics[width=0.99\textwidth]{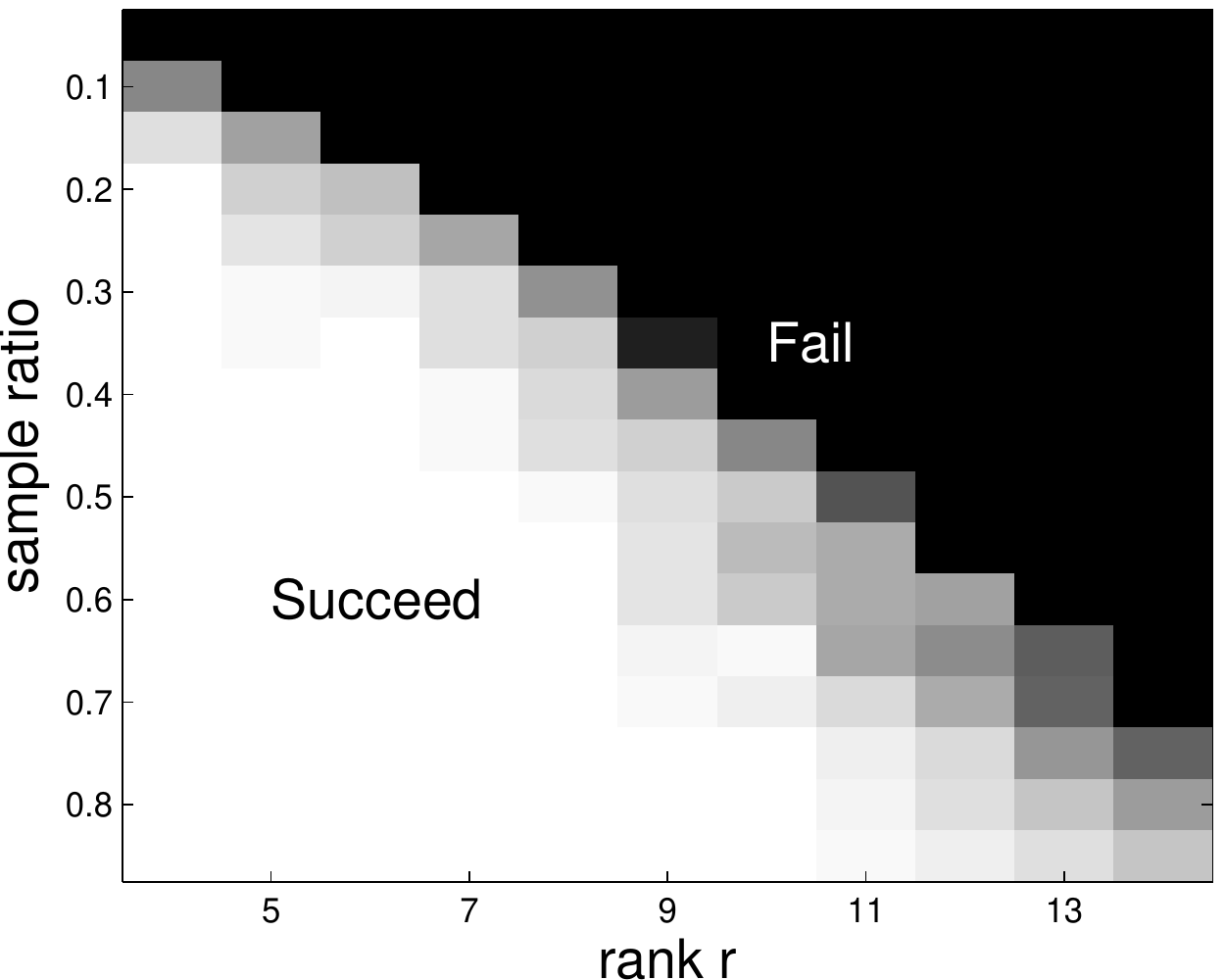}
\end{minipage}
}\\
\subfigure[TMac-fix]{
\begin{minipage}[t]{0.25\textwidth}
\centering
\includegraphics[width=0.99\textwidth]{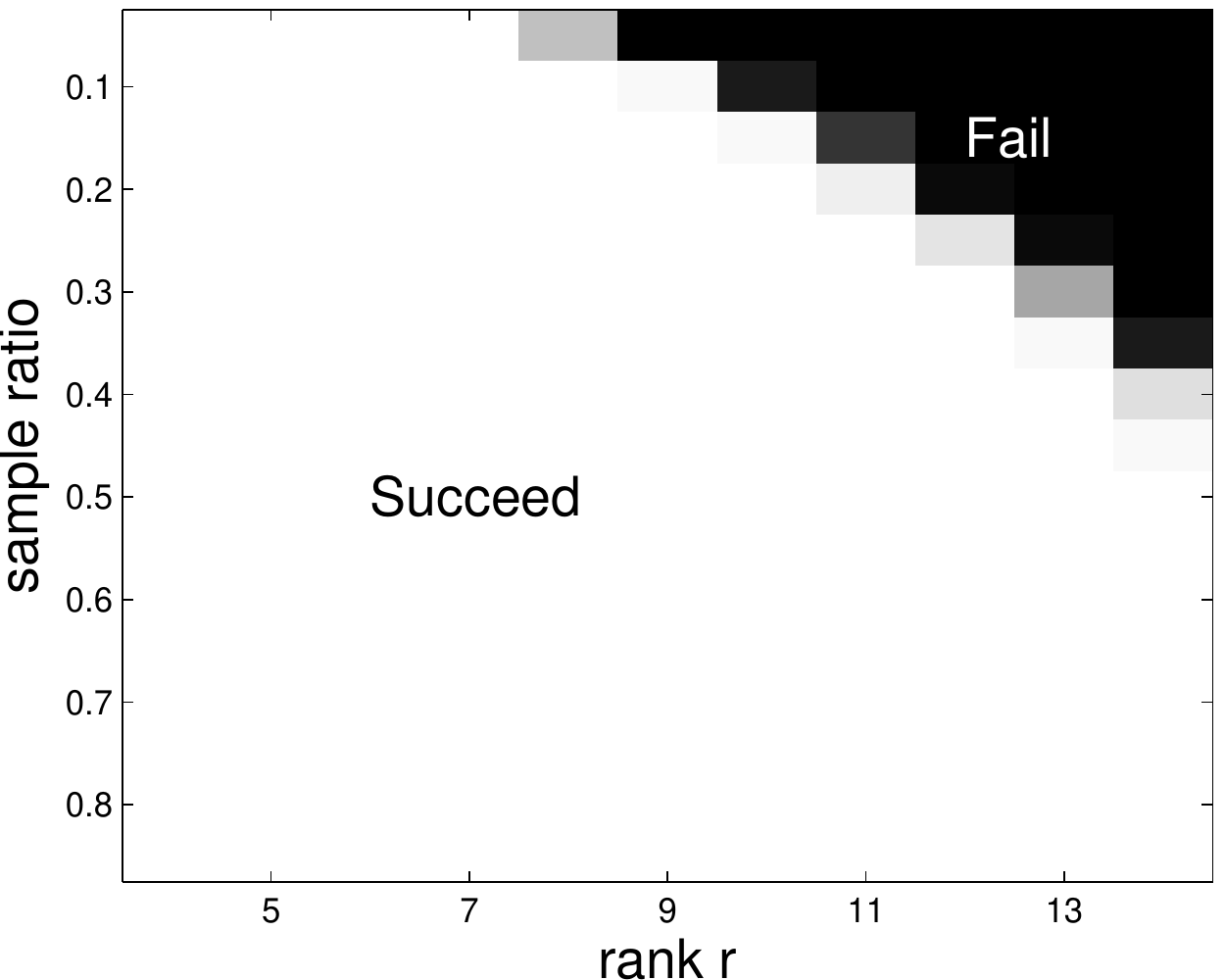}
\end{minipage}
}
\subfigure[TMac-inc]{
\begin{minipage}[t]{0.25\textwidth}
\centering
\includegraphics[width=0.99\textwidth]{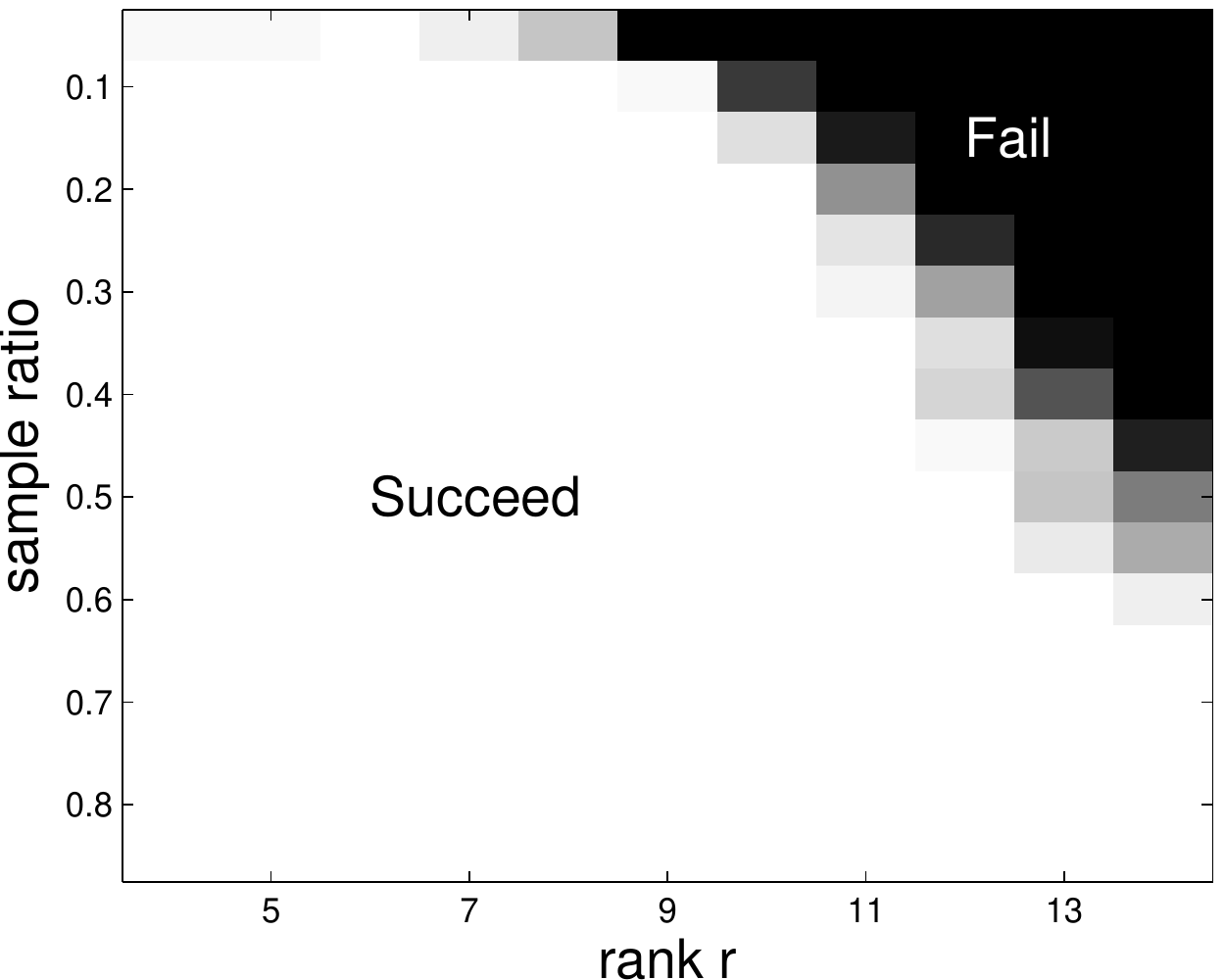}
\end{minipage}
}
\subfigure[TMac-dec]{
\begin{minipage}[t]{0.25\textwidth}
\centering
\includegraphics[width=0.99\textwidth]{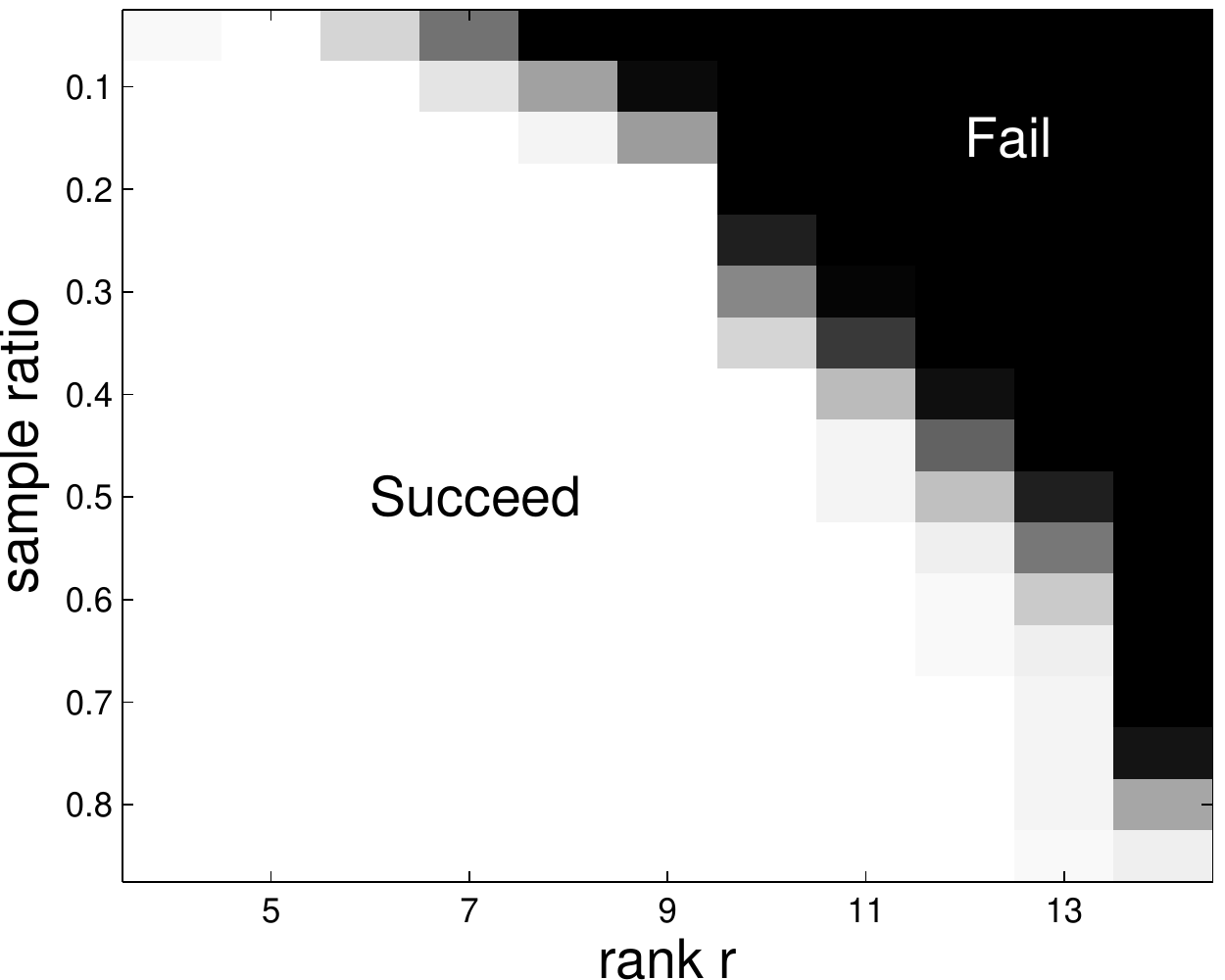}
\end{minipage}
}
\end{figure}

\begin{figure}[H]
\captionsetup{width=0.99\textwidth}
\caption{Phase transition plots for model \eqref{eq:main} utilizing different numbers of mode matricization on \textbf{4-way low-rank tensors whose factors have Gaussian random entries}. (a) 1 mode: Algorithm \ref{alg:als} solves \eqref{eq:main} with $\alpha_1=1,\alpha_n=0,n\ge2$ and each $r_n$ fixed to $r$; (b) 2 modes: Algorithm \ref{alg:als} solves \eqref{eq:main} with $\alpha_1=\alpha_2=0.5,\alpha_3=\alpha_4=0,$ and each $r_n$ fixed to $r$; (c) 3 modes: Algorithm \ref{alg:als} solves \eqref{eq:main} with $\alpha_n=\frac{1}{3}, n\le 3, \alpha_4=0,$ and each $r_n$ fixed to $r$; (d) 4 modes: Algorithm \ref{alg:als} solves \eqref{eq:main} with $\alpha_n=0.25,\forall n$ and each $r_n$ fixed to $r$.}\label{fig:diff-mode}
\centering
\subfigure[1 mode]{
\begin{minipage}[t]{0.25\textwidth}
\centering
\includegraphics[width=0.99\textwidth]{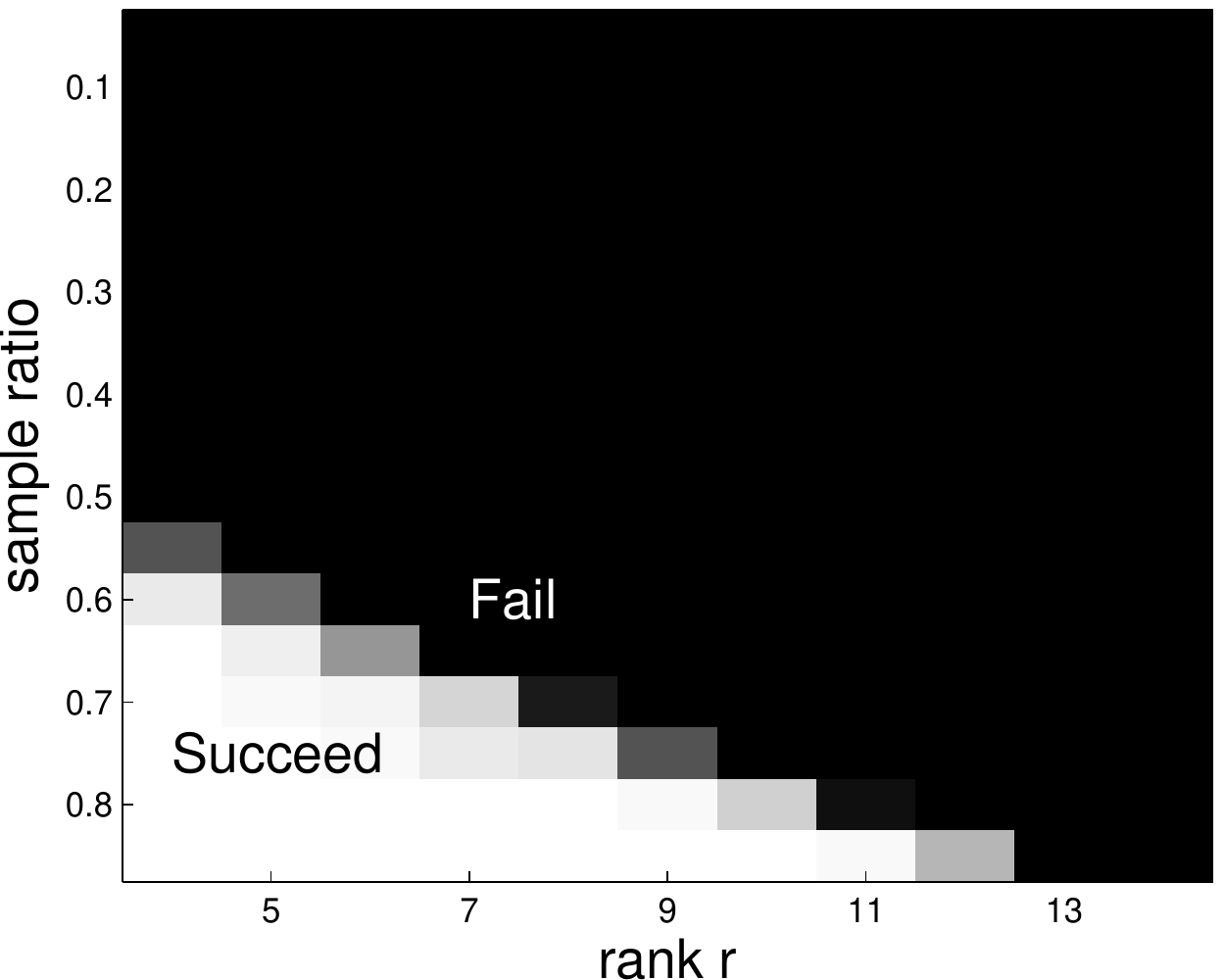}
\end{minipage}
}
\subfigure[2 modes]{
\begin{minipage}[t]{0.25\textwidth}
\centering
\includegraphics[width=0.99\textwidth]{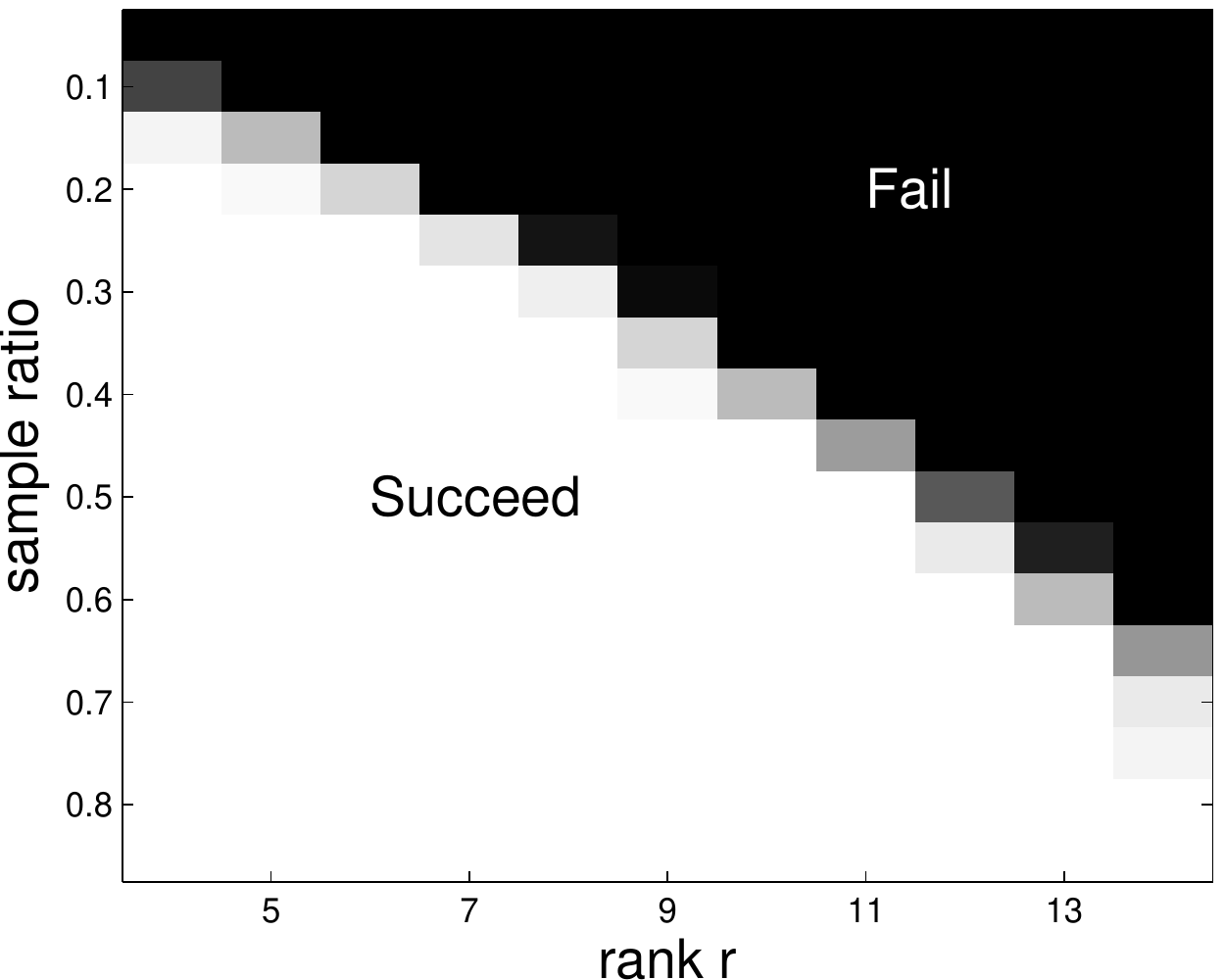}
\end{minipage}
}\\
\subfigure[3 modes]{
\begin{minipage}[t]{0.25\textwidth}
\centering
\includegraphics[width=0.99\textwidth]{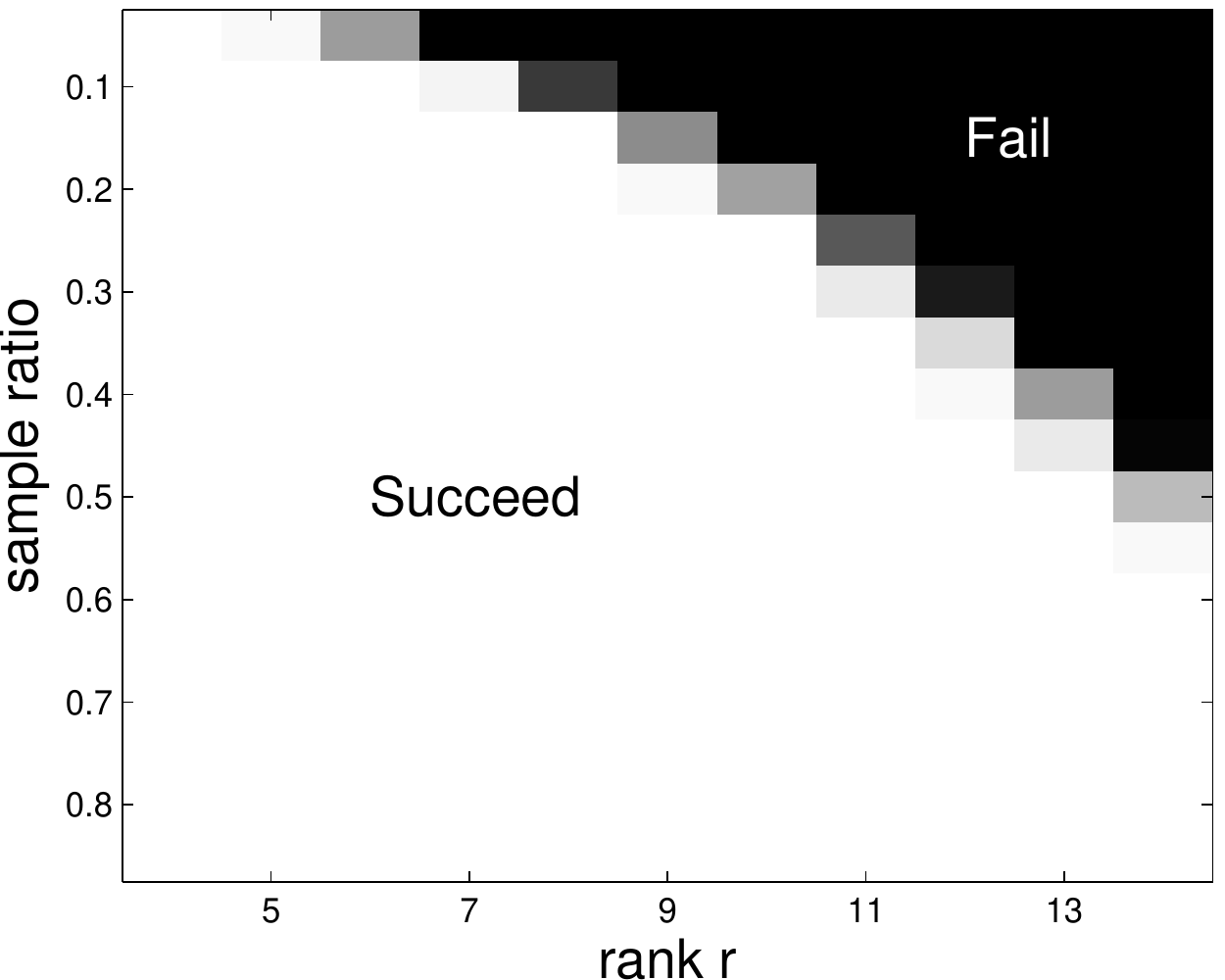}
\end{minipage}
}
\subfigure[4 modes]{
\begin{minipage}[t]{0.25\textwidth}
\centering
\includegraphics[width=0.99\textwidth]{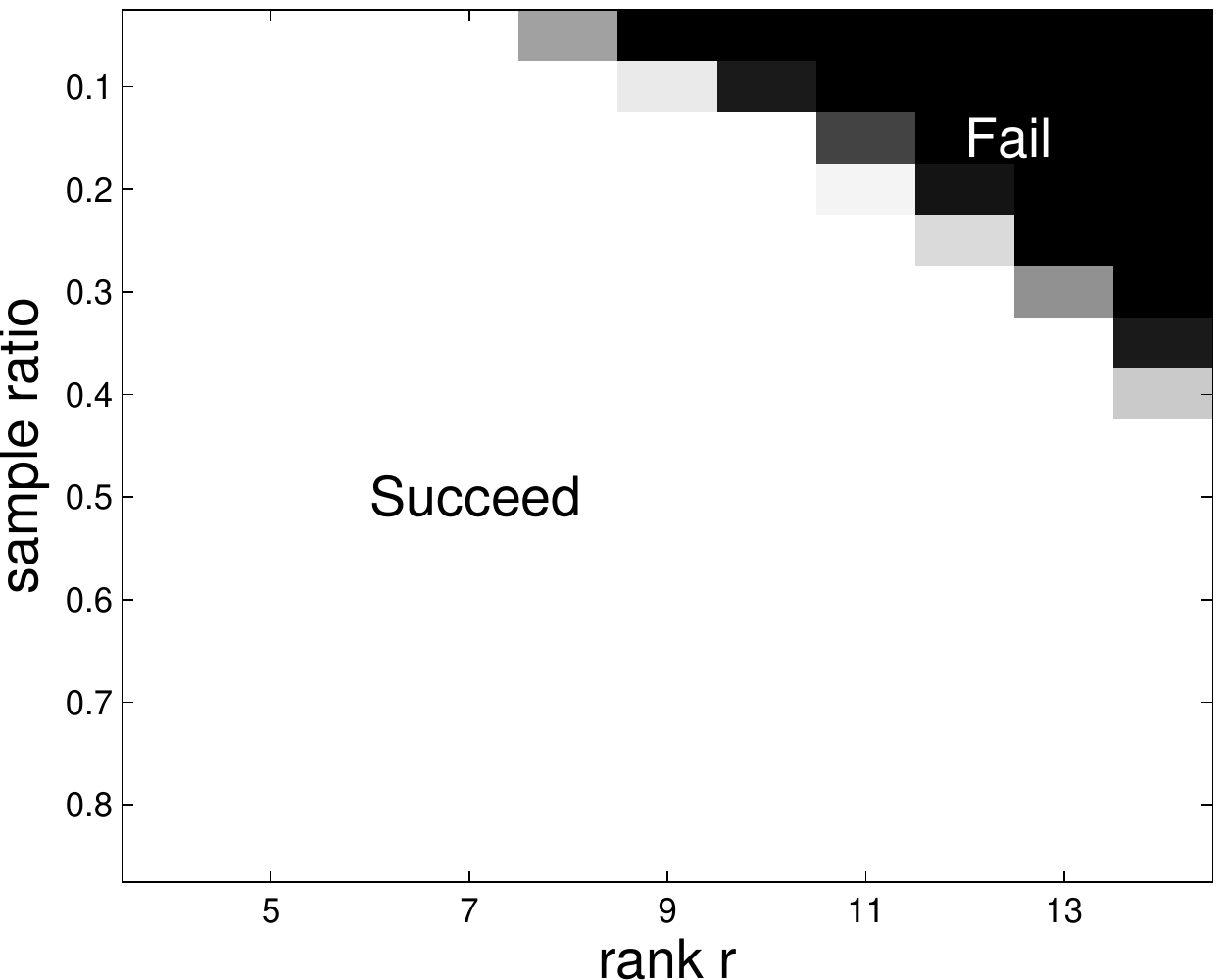}
\end{minipage}
}
\end{figure}

\begin{figure}[H]
\captionsetup{width=0.99\textwidth}
\caption{Phase transition plots for different methods on \textbf{3-way low-rank tensors whose factors have uniformly random entries}. (a) FaLRTC: the tensor completion method in \cite{liu2013tensor} that solves \eqref{eq:conv}. (b) MatComp: the matrix completion solver LMaFit in \cite{wen2012lmafit} that solves \eqref{eq:lmafit} with $\mbfz$ corresponding to $\mbfm_{(N)}$. (c) TMac-fix: Algorithm \ref{alg:als} solves \eqref{eq:main} with $\alpha_n=\frac{1}{3}$ and $r_n$ fixed to $r$, $\forall n$. (d) TMac-inc: Algorithm \ref{alg:als} solves \eqref{eq:main} with $\alpha_n=\frac{1}{3},\forall n$ and using rank-increasing strategy starting from $r_n=\text{round}(0.75r),\forall n$. (e) TMac-dec: Algorithm \ref{alg:als} solves \eqref{eq:main} with $\alpha_n=\frac{1}{3},\forall n$ and using rank-decreasing strategy starting from $r_n=\text{round}(1.25r),\forall n$.}\label{fig:3way-rand}
\centering
\subfigure[MatComp]{
\begin{minipage}[t]{0.25\textwidth}
\centering
\includegraphics[width=0.99\textwidth]{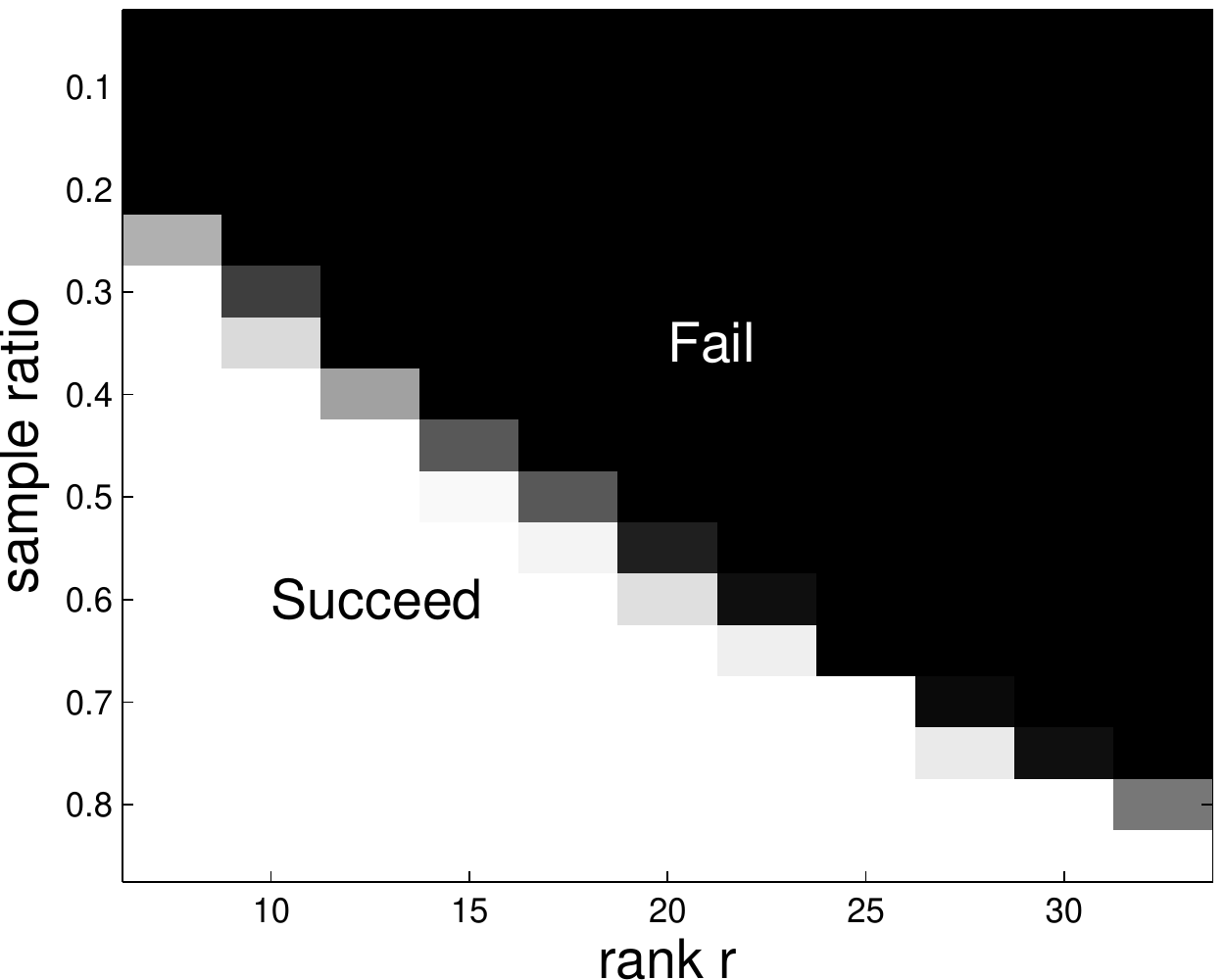}
\end{minipage}
}
\subfigure[FaLRTC]{
\begin{minipage}[t]{0.25\textwidth}
\centering
\includegraphics[width=0.99\textwidth]{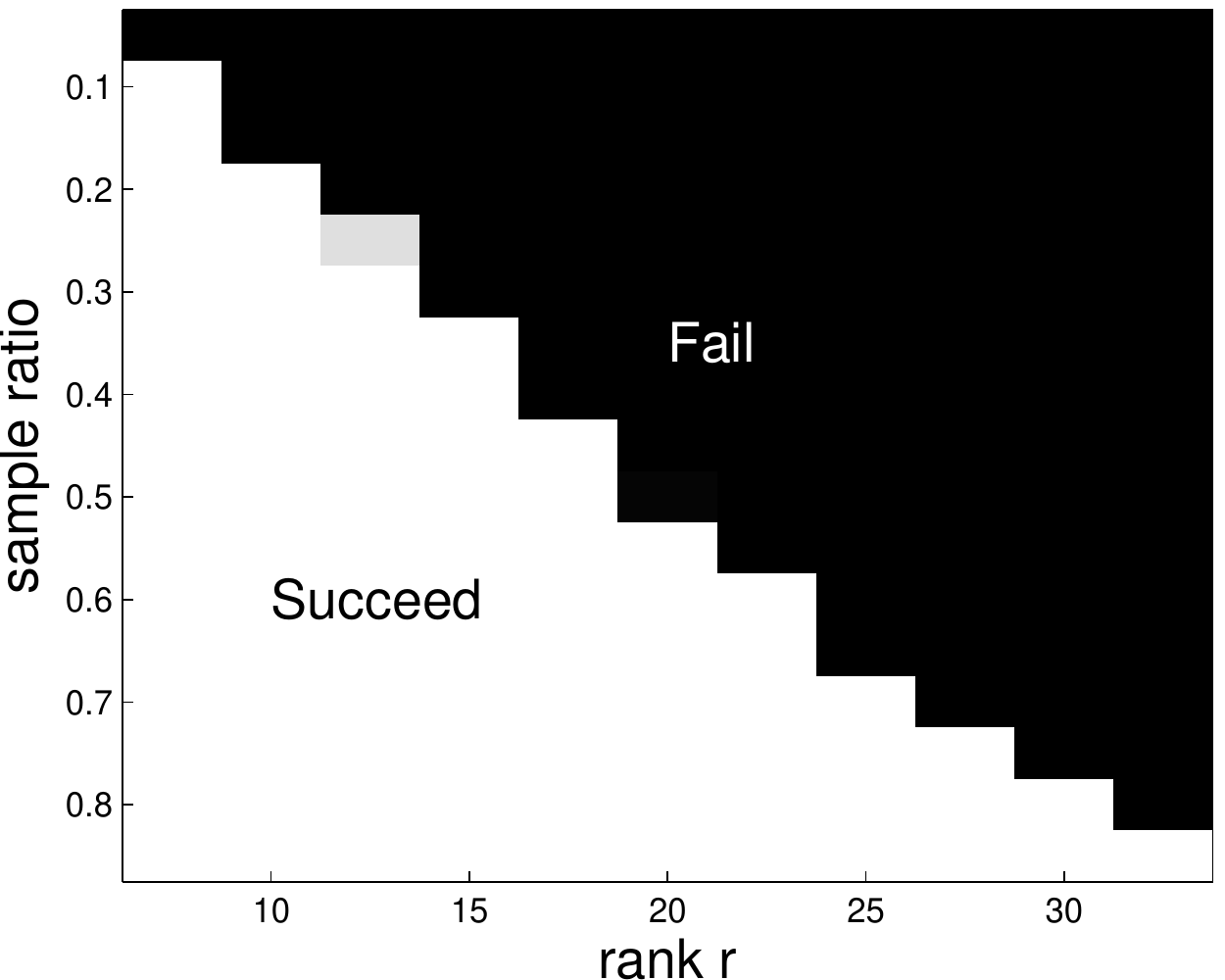}
\end{minipage}
}\\
\subfigure[TMac-fix]{
\begin{minipage}[t]{0.25\textwidth}
\centering
\includegraphics[width=0.99\textwidth]{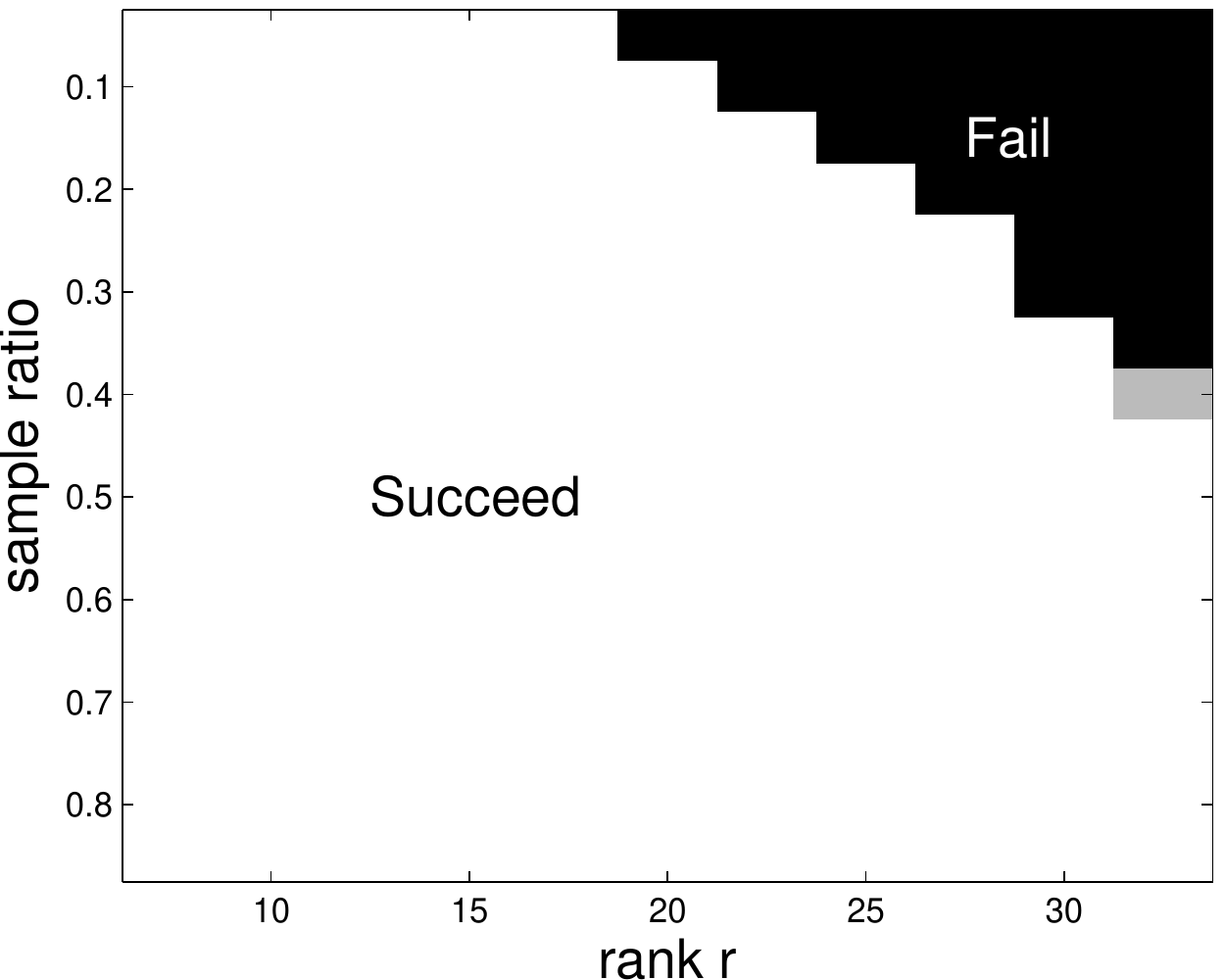}
\end{minipage}
}
\subfigure[TMac-inc]{
\begin{minipage}[t]{0.25\textwidth}
\centering
\includegraphics[width=0.99\textwidth]{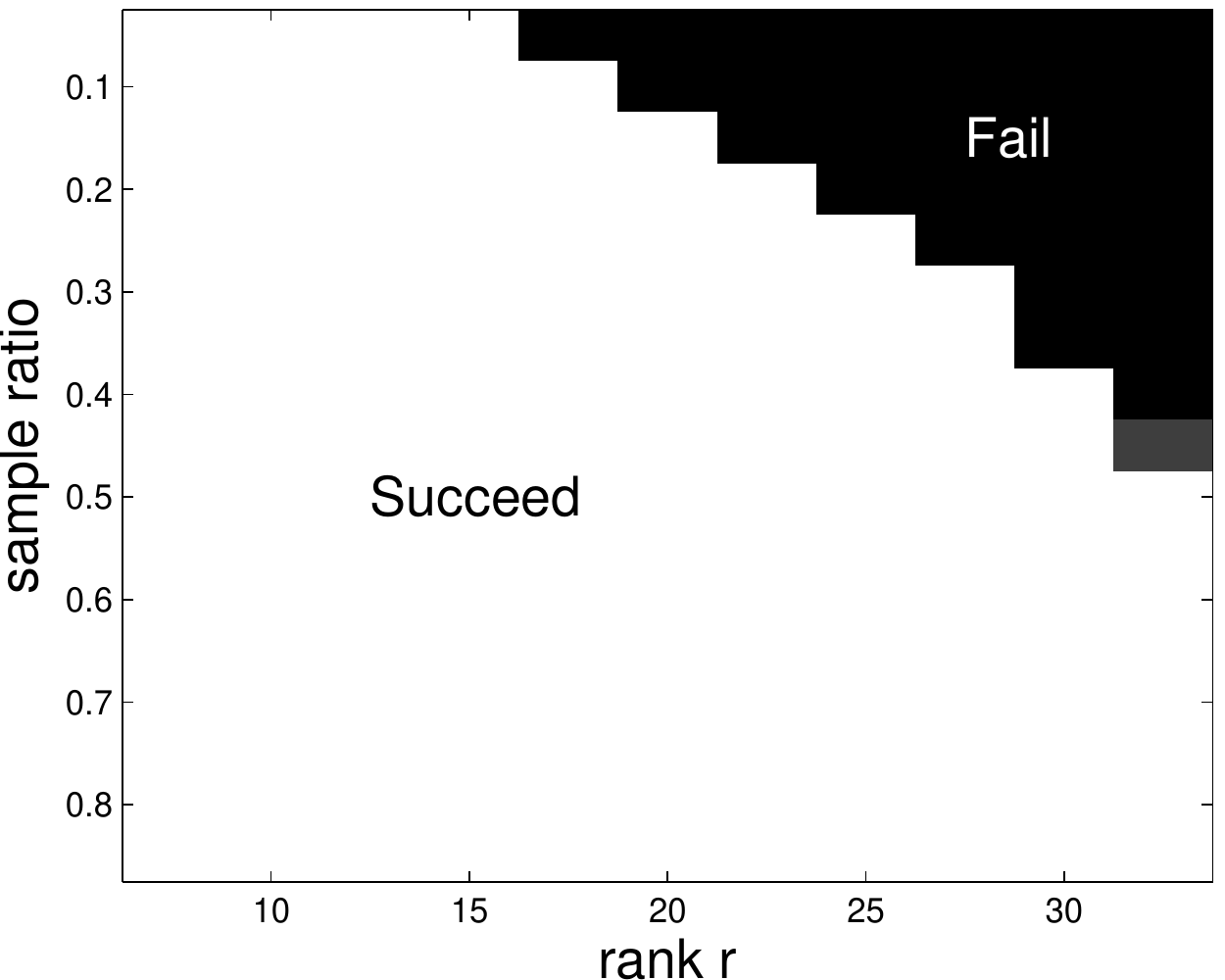}
\end{minipage}
}
\subfigure[TMac-dec]{
\begin{minipage}[t]{0.25\textwidth}
\centering
\includegraphics[width=0.99\textwidth]{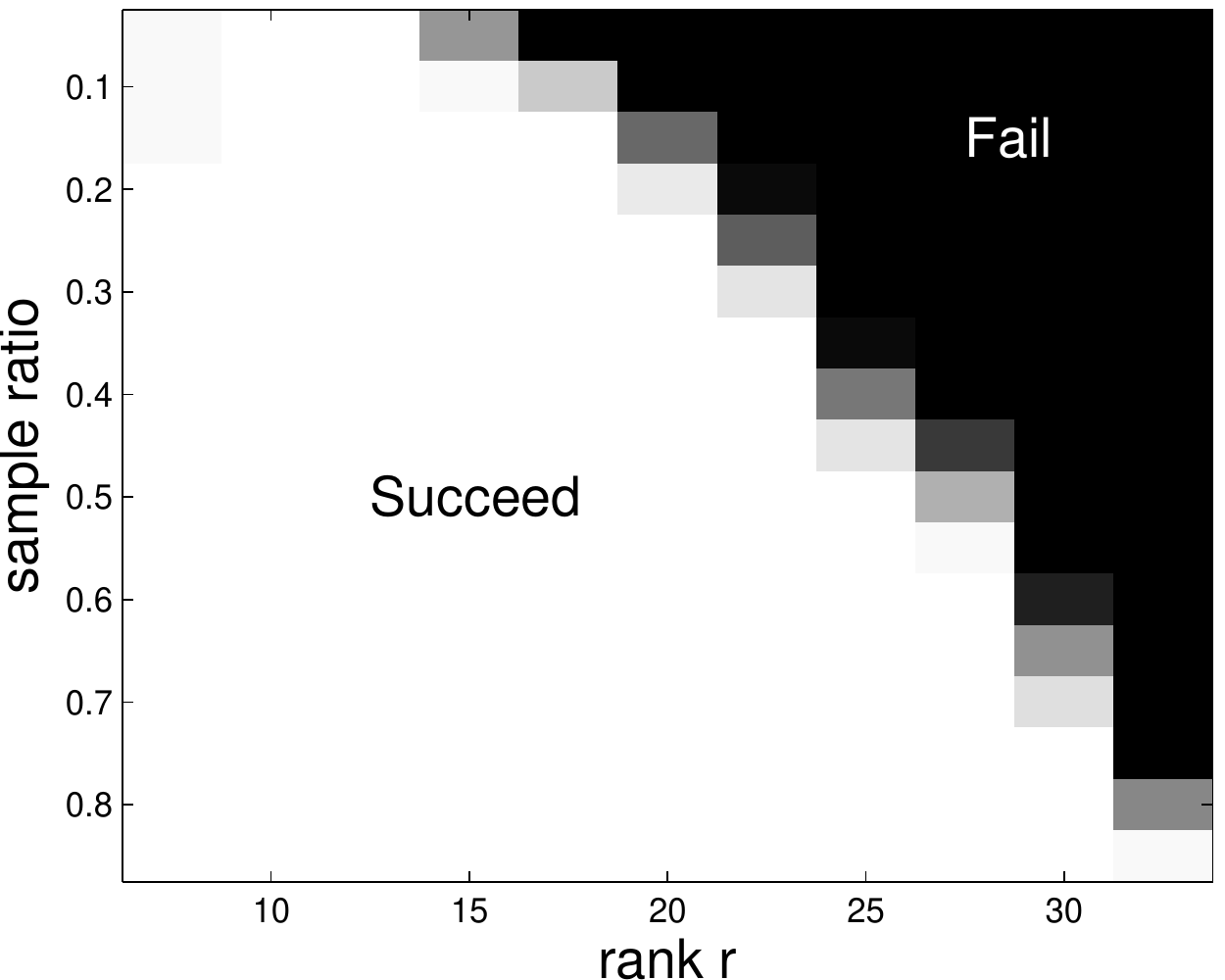}
\end{minipage}
}
\end{figure}

\begin{figure}[H]
\captionsetup{width=0.99\textwidth}
\caption{Phase transition plots for different methods on \textbf{4-way low-rank tensors whose factors have uniformly random entries}. (a) SquareDeal: the matrix completion solver FPCA \cite{ma2011fixed} solves the model \eqref{eq:square} proposed in \cite{mu2013square}. (b) the matrix completion solver LMaFit \cite{wen2012lmafit} solves \eqref{eq:sqm-lmafit}, which is a non-convex variant of \eqref{eq:square}. (c) TMac-fix: Algorithm \ref{alg:als} solves \eqref{eq:main} with $\alpha_n=\frac{1}{4}$ and $r_n$ fixed to $r$, $\forall n$. (d) TMac-inc: Algorithm \ref{alg:als} solves \eqref{eq:main} with $\alpha_n=\frac{1}{4},\forall n$ and using rank-increasing strategy starting from $r_n=\text{round}(0.75r),\forall n$. (e) TMac-dec: Algorithm \ref{alg:als} solves \eqref{eq:main} with $\alpha_n=\frac{1}{4},\forall n$ and using rank-decreasing strategy starting from $r_n=\text{round}(1.25r),\forall n$.}\label{fig:4way-rand}
\centering
\subfigure[SquareDeal]{
\begin{minipage}[t]{0.25\textwidth}
\centering
\includegraphics[width=0.99\textwidth]{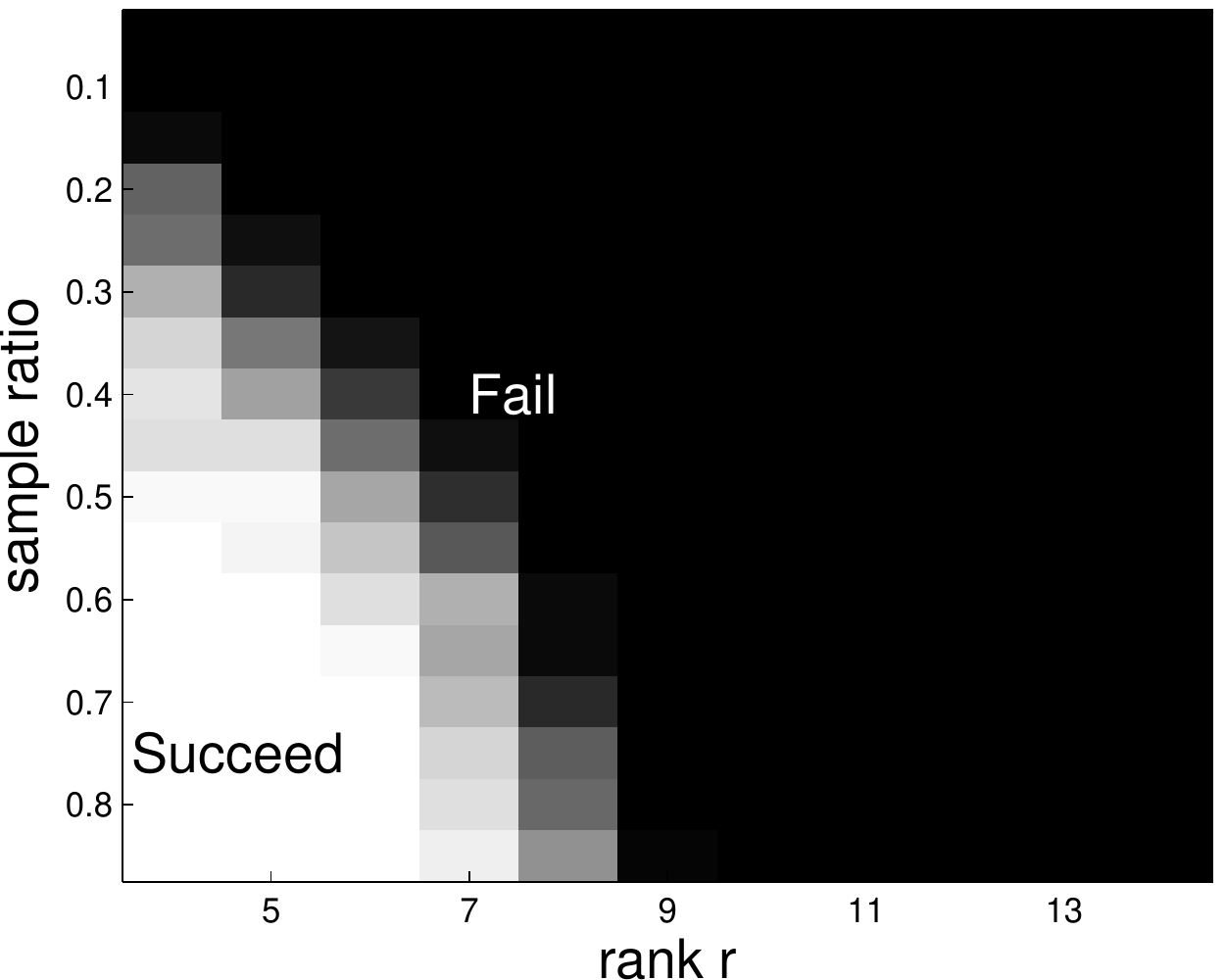}
\end{minipage}
}
\subfigure[\eqref{eq:sqm-lmafit} solved by LMaFit]{
\begin{minipage}[t]{0.25\textwidth}
\centering
\includegraphics[width=0.99\textwidth]{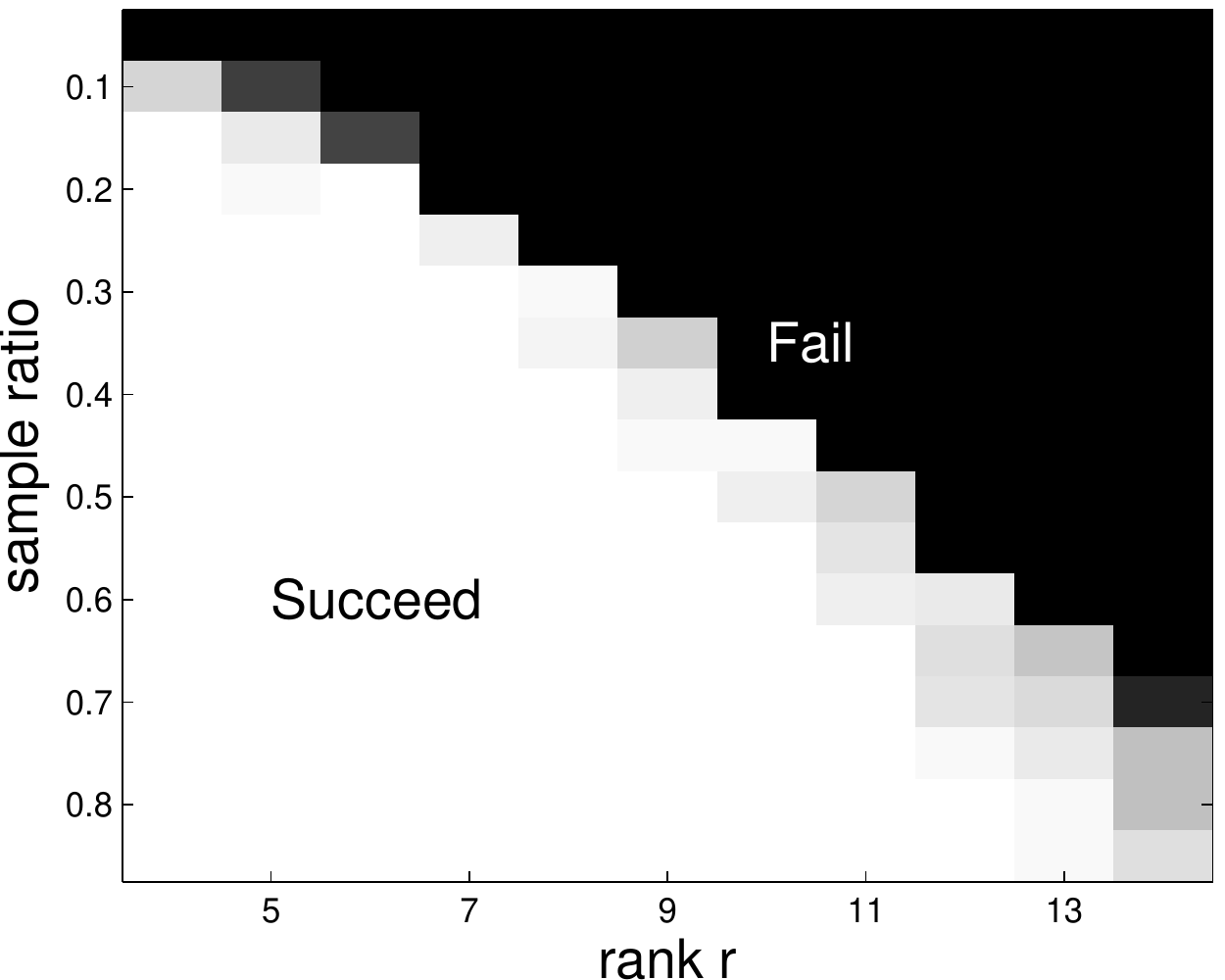}
\end{minipage}
}\\
\subfigure[TMac-fix]{
\begin{minipage}[t]{0.25\textwidth}
\centering
\includegraphics[width=0.99\textwidth]{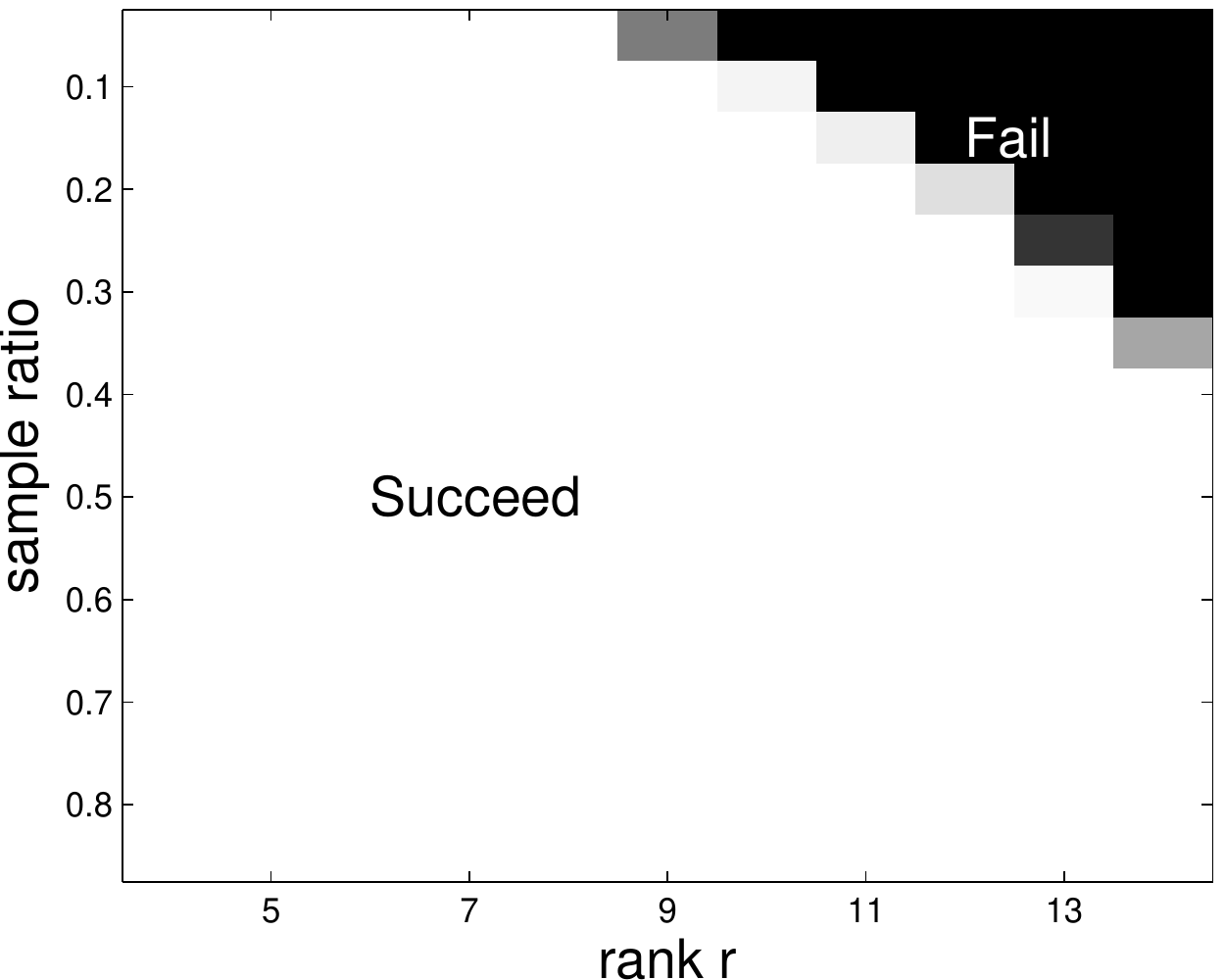}
\end{minipage}
}
\subfigure[TMac-inc]{
\begin{minipage}[t]{0.25\textwidth}
\centering
\includegraphics[width=0.99\textwidth]{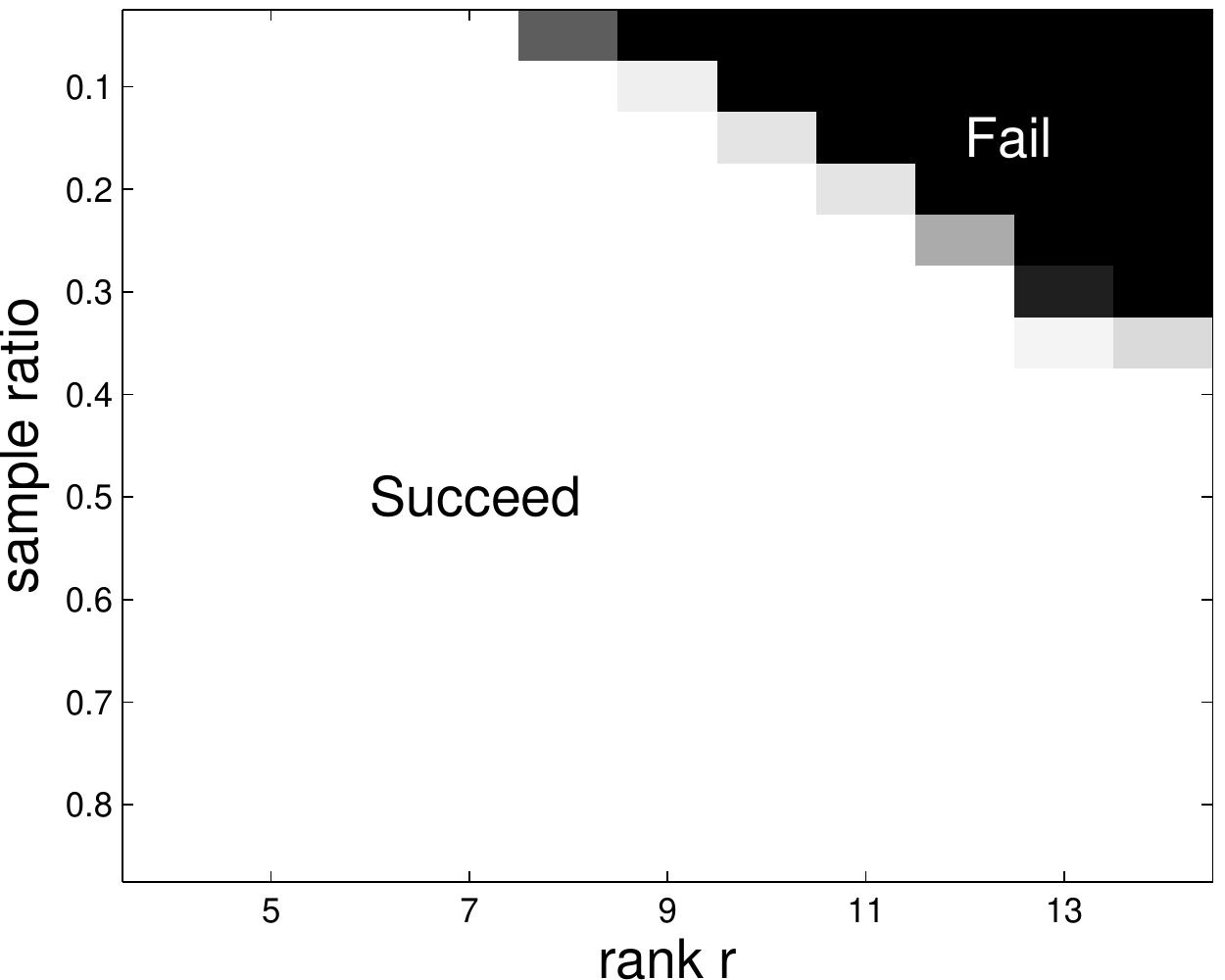}
\end{minipage}
}
\subfigure[TMac-dec]{
\begin{minipage}[t]{0.25\textwidth}
\centering
\includegraphics[width=0.99\textwidth]{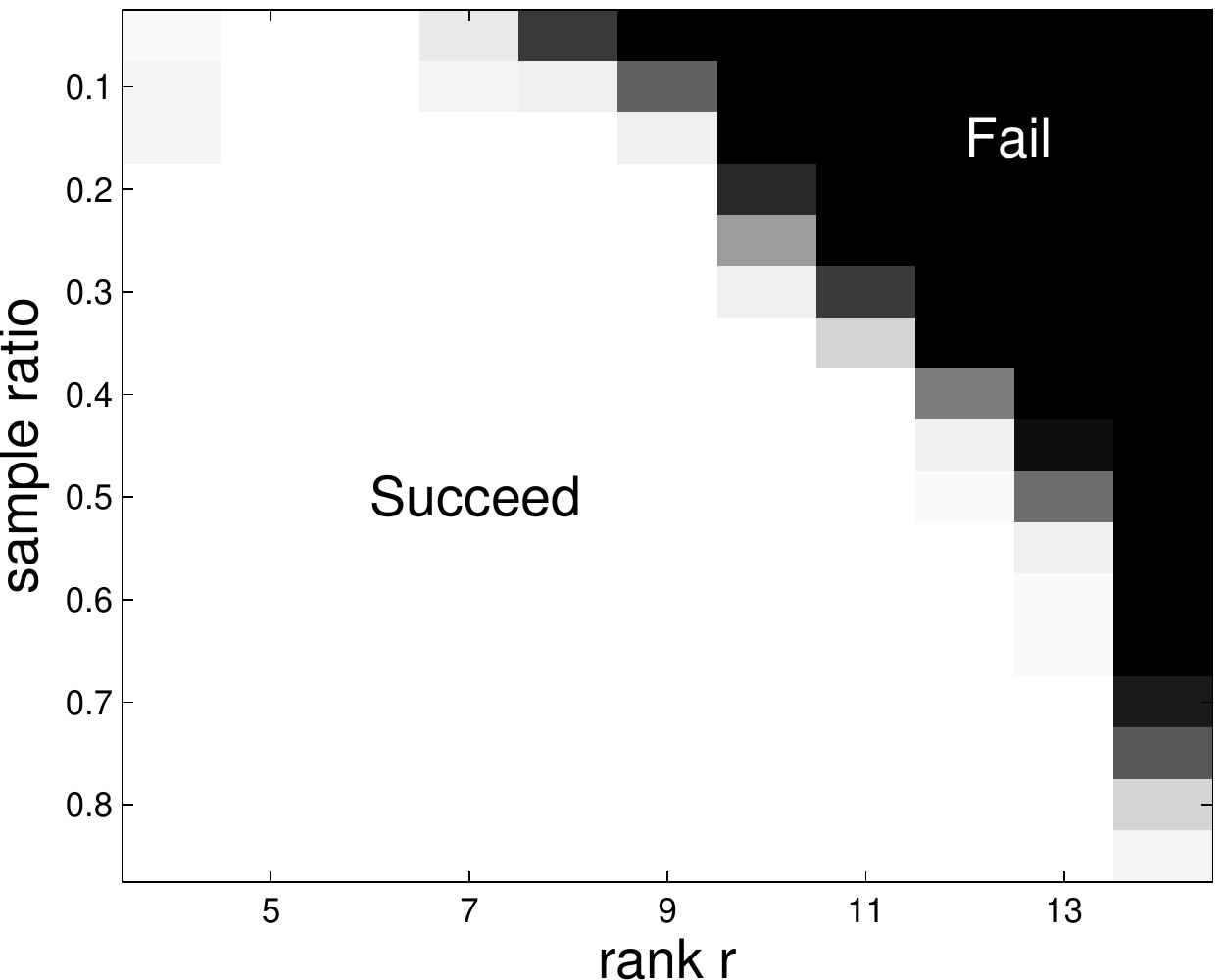}
\end{minipage}
}
\end{figure}

\begin{figure}[H]
\captionsetup{width=0.99\textwidth}
\caption{Phase transition plots for different methods on \textbf{3-way random low-rank tensors whose factors have power-law decaying singular values}. (a) FaLRTC: the tensor completion method in \cite{liu2013tensor} that solves \eqref{eq:conv}. (b) MatComp: the matrix completion solver LMaFit in \cite{wen2012lmafit} that solves \eqref{eq:lmafit} with $\mbfz$ corresponding to $\mbfm_{(N)}$. (c) TMac-fix: Algorithm \ref{alg:als} solves \eqref{eq:main} with $\alpha_n=\frac{1}{3}$ and $r_n$ fixed to $r$, $\forall n$. (d) TMac-inc: Algorithm \ref{alg:als} solves \eqref{eq:main} with $\alpha_n=\frac{1}{3},\forall n$ and using rank-increasing strategy starting from $r_n=\text{round}(0.75r),\forall n$.}\label{fig:3way-plaw}
\centering
\subfigure[MatComp]{
\begin{minipage}[t]{0.25\textwidth}
\centering
\includegraphics[width=0.99\textwidth]{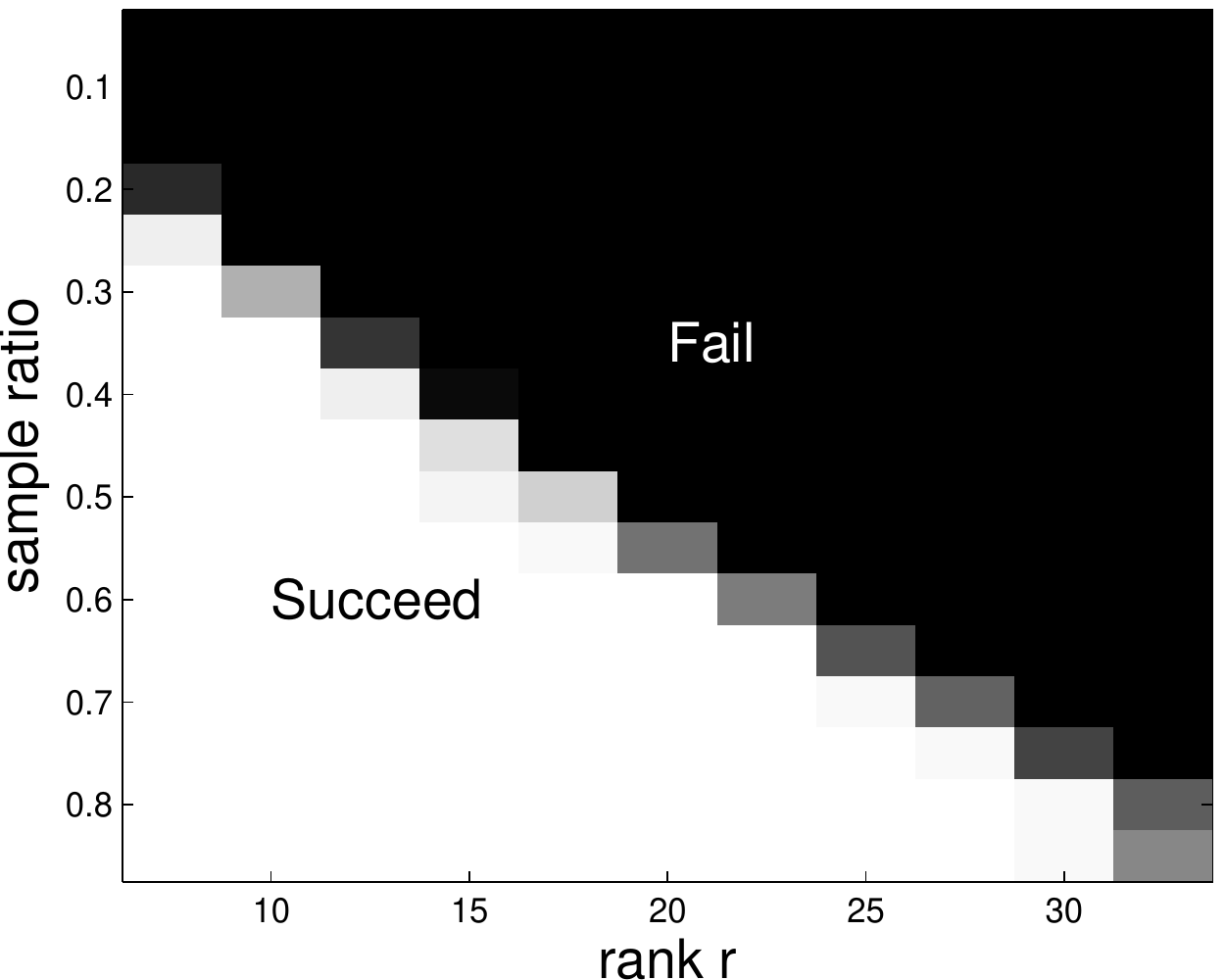}
\end{minipage}
}
\subfigure[FaLRTC]{
\begin{minipage}[t]{0.25\textwidth}
\centering
\includegraphics[width=0.99\textwidth]{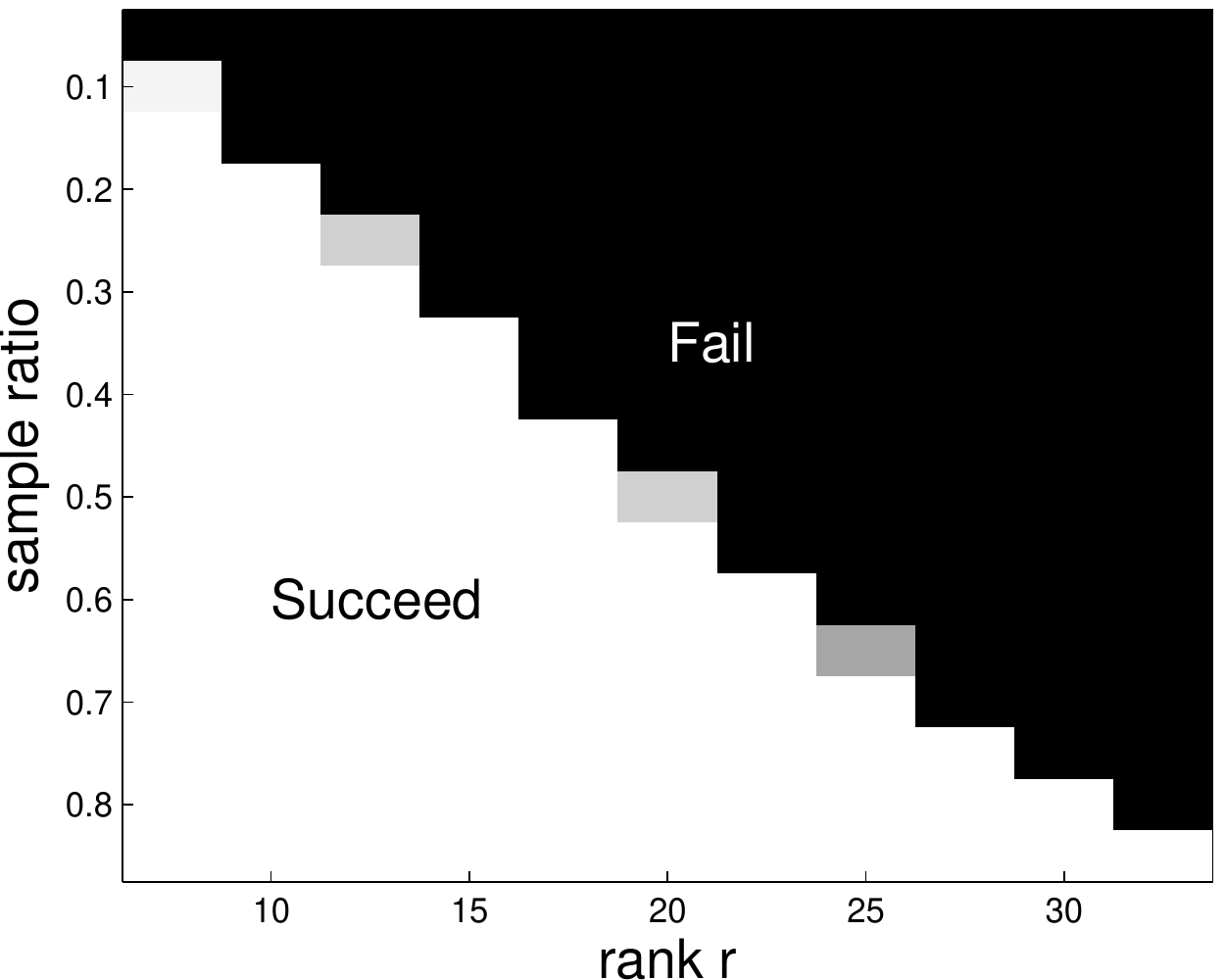}
\end{minipage}
}\\
\subfigure[TMac-fix]{
\begin{minipage}[t]{0.25\textwidth}
\centering
\includegraphics[width=0.99\textwidth]{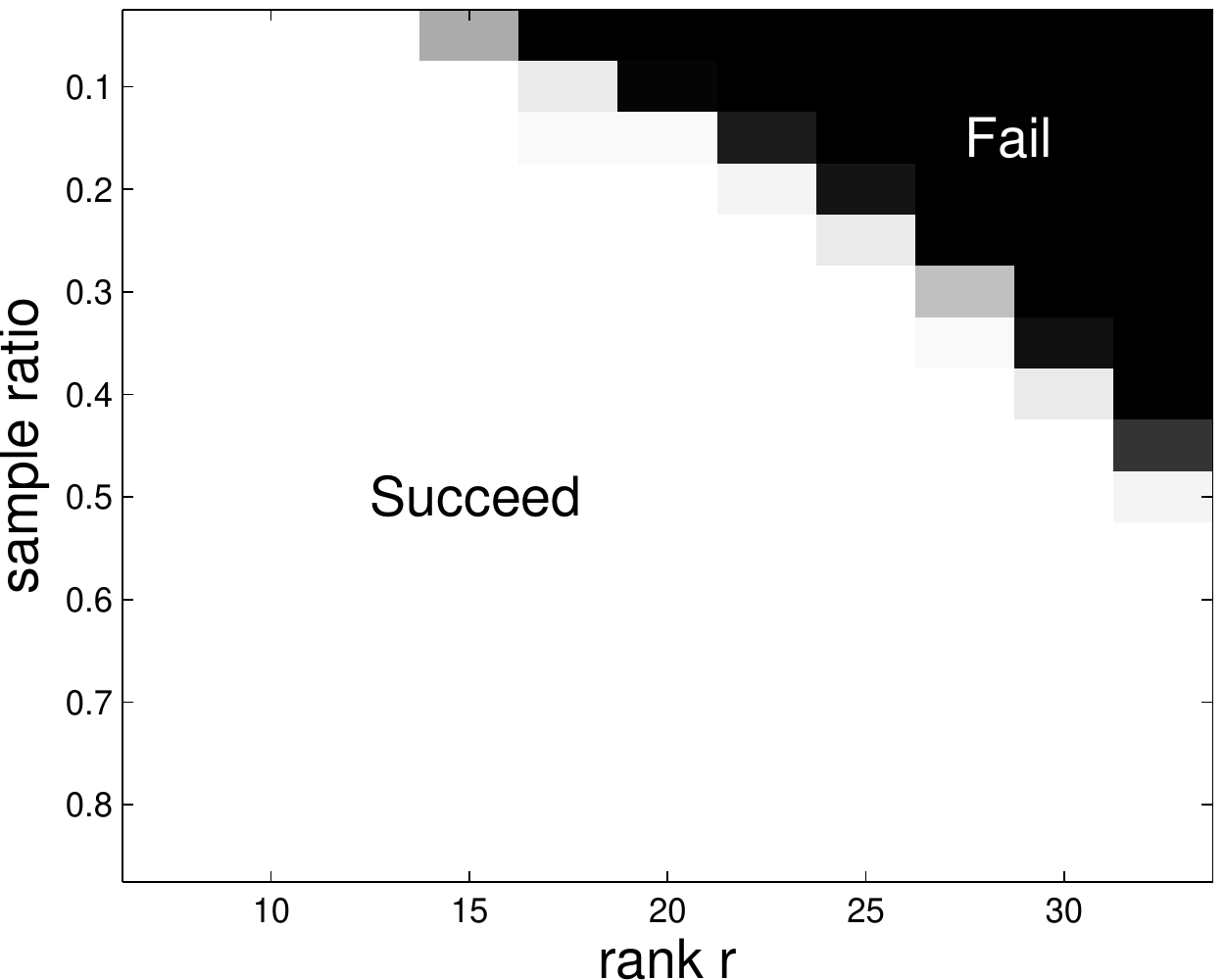}
\end{minipage}
}
\subfigure[TMac-inc]{
\begin{minipage}[t]{0.25\textwidth}
\centering
\includegraphics[width=0.99\textwidth]{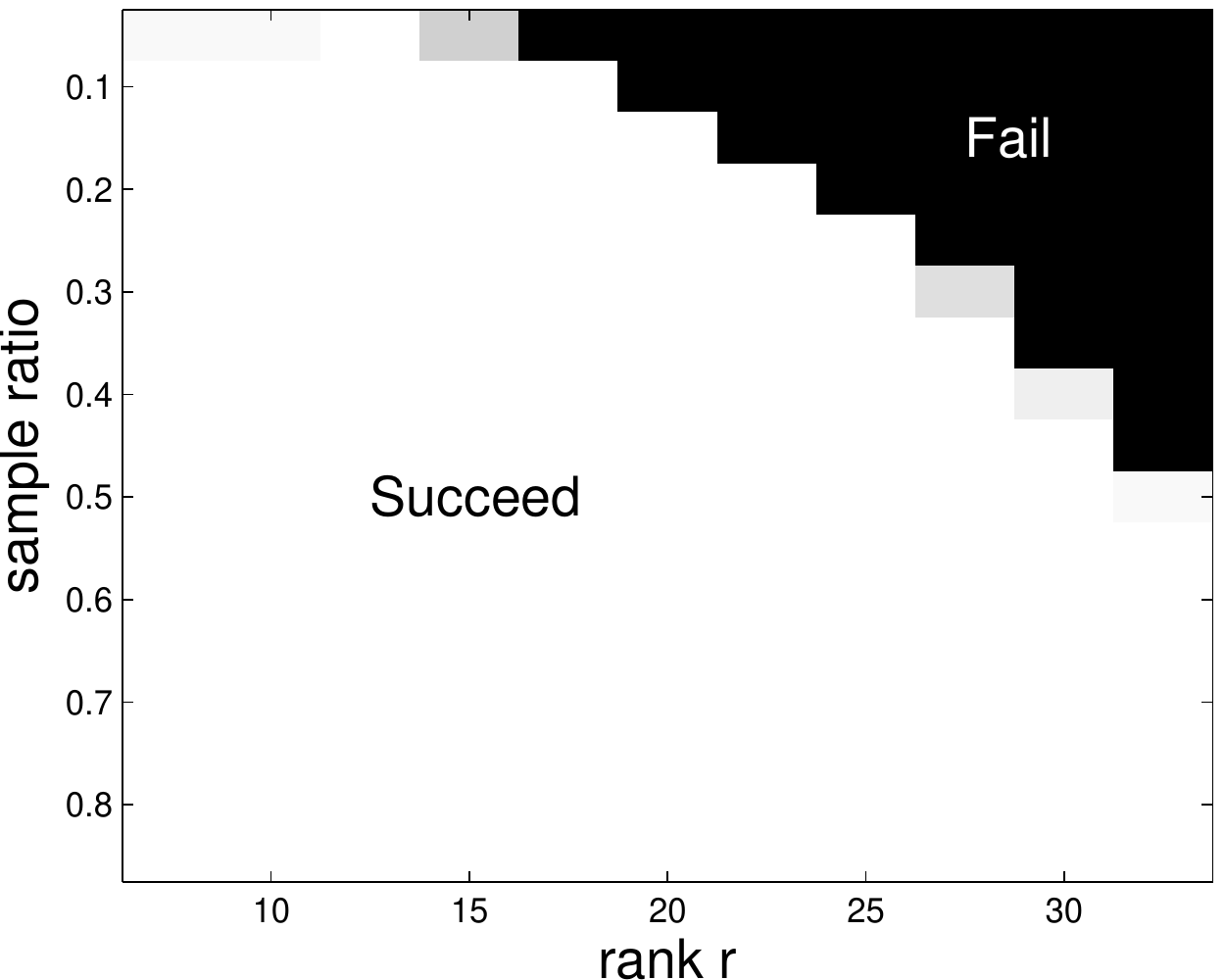}
\end{minipage}
}
\end{figure}

\begin{figure}[H]
\captionsetup{width=0.99\textwidth}
\caption{Phase transition plots for different methods on \textbf{4-way random low-rank tensors whose factors have power-law decaying singular values}. (a) SquareDeal: the matrix completion solver FPCA \cite{ma2011fixed} solves the model \eqref{eq:square} proposed in \cite{mu2013square}. (b) the matrix completion solver LMaFit \cite{wen2012lmafit} solves \eqref{eq:sqm-lmafit}, which is a non-convex variant of \eqref{eq:square}. (c) TMac-fix: Algorithm \ref{alg:als} solves \eqref{eq:main} with $\alpha_n=\frac{1}{4}$ and $r_n$ fixed to $r$, $\forall n$. (d) TMac-inc: Algorithm \ref{alg:als} solves \eqref{eq:main} with $\alpha_n=\frac{1}{4},\forall n$ and using rank-increasing strategy starting from $r_n=\text{round}(0.75r),\forall n$. }\label{fig:4way-plaw}
\centering
\subfigure[SquareDeal]{
\begin{minipage}[t]{0.25\textwidth}
\centering
\includegraphics[width=0.99\textwidth]{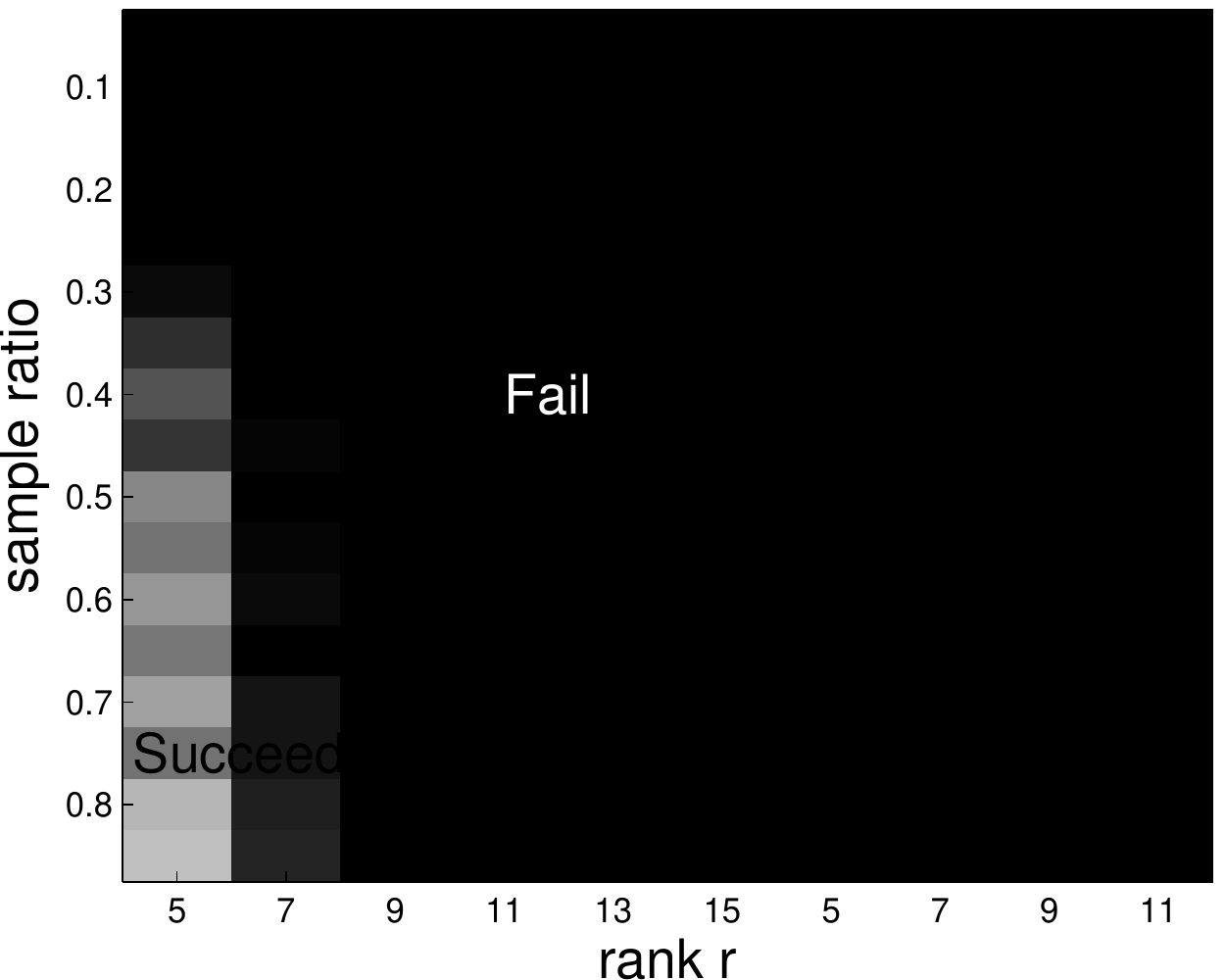}
\end{minipage}
}
\subfigure[\eqref{eq:sqm-lmafit} solved by LMaFit]{
\begin{minipage}[t]{0.25\textwidth}
\centering
\includegraphics[width=0.99\textwidth]{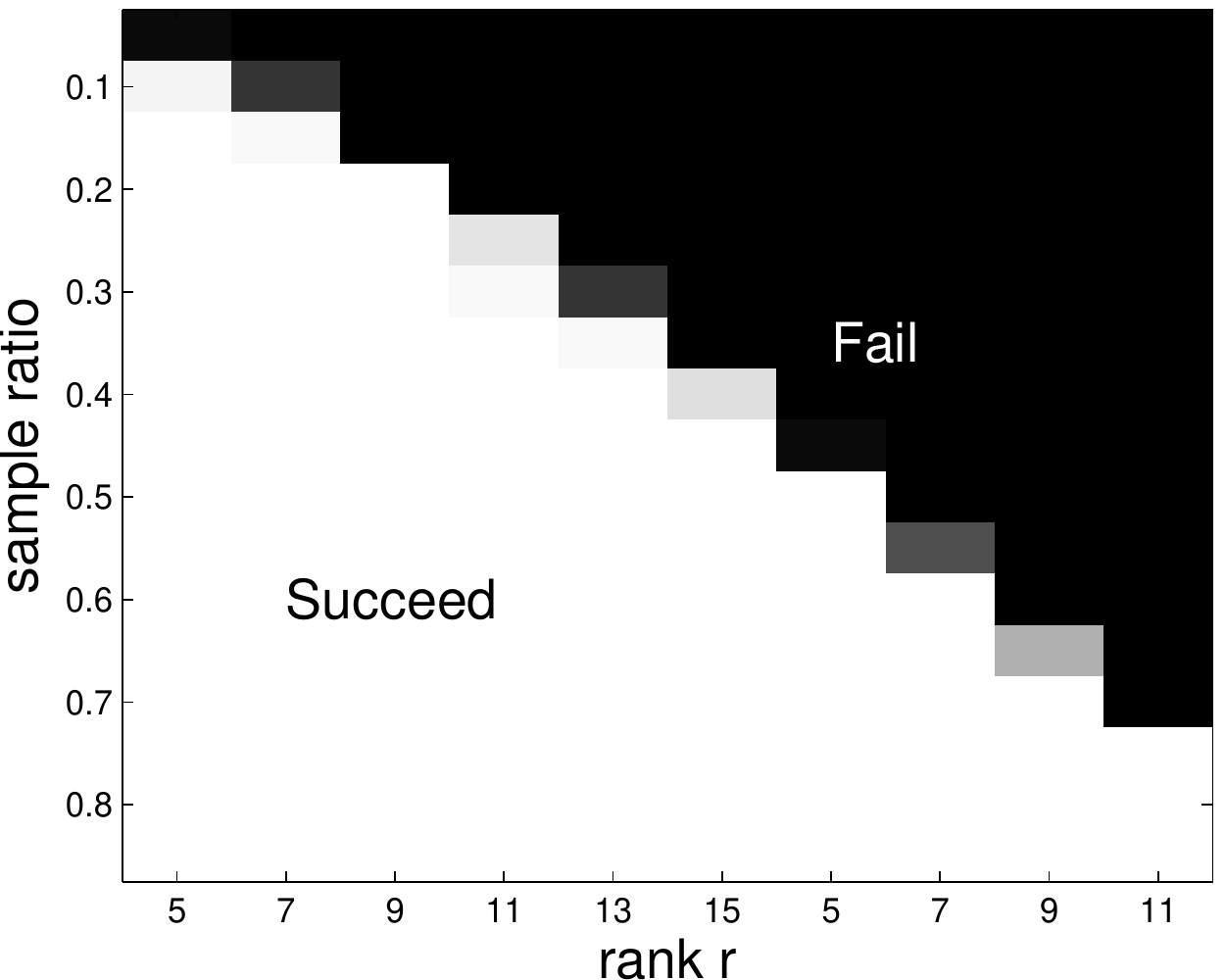}
\end{minipage}
}\\
\subfigure[TMac-fix]{
\begin{minipage}[t]{0.25\textwidth}
\centering
\includegraphics[width=0.99\textwidth]{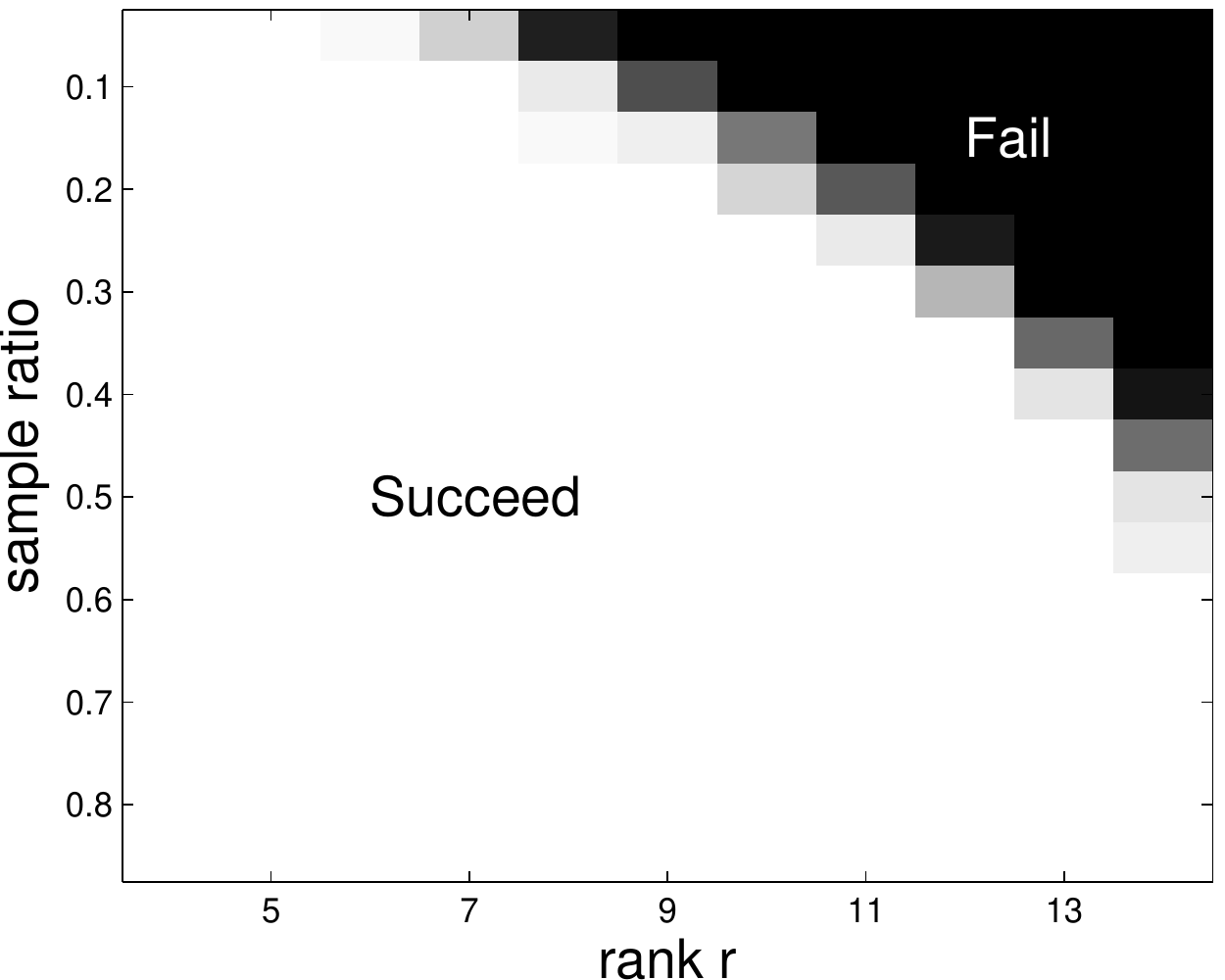}
\end{minipage}
}
\subfigure[TMac-inc]{
\begin{minipage}[t]{0.25\textwidth}
\centering
\includegraphics[width=0.99\textwidth]{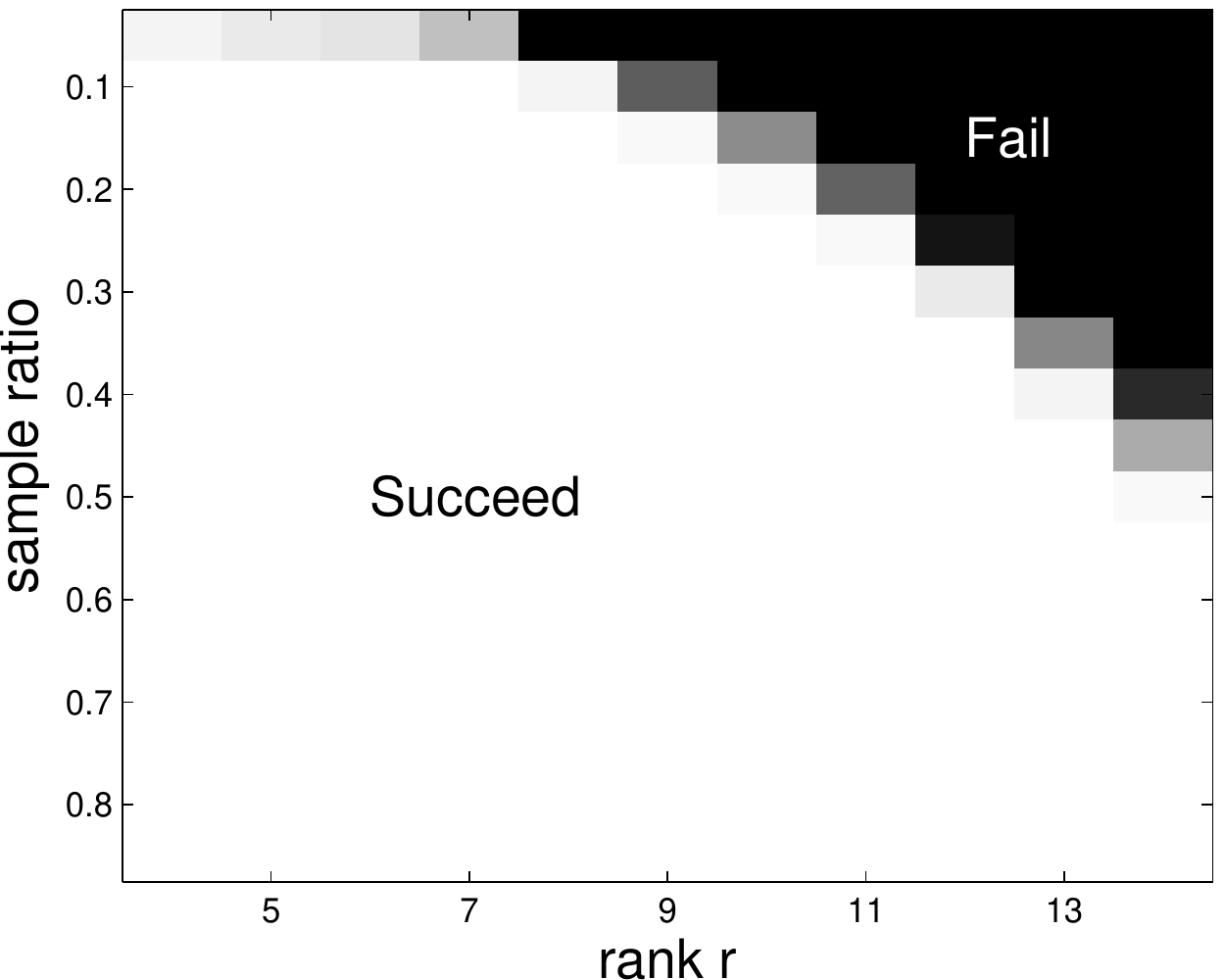}
\end{minipage}
}
\end{figure}

\begin{figure}[H]
\captionsetup{width=0.99\textwidth}
\caption{Convergence behavior of Algorithm \ref{alg:als} with two rank-adjusting strategies on Gaussian randomly generated $50\times50\times50$ tensors that have each mode rank to be $r$.}\label{fig:dec-inc}
\begin{center}
{\small
\begin{tabular}{cc}
true rank $r=10$ & true rank $r=25$\\
\includegraphics[width=0.3\textwidth]{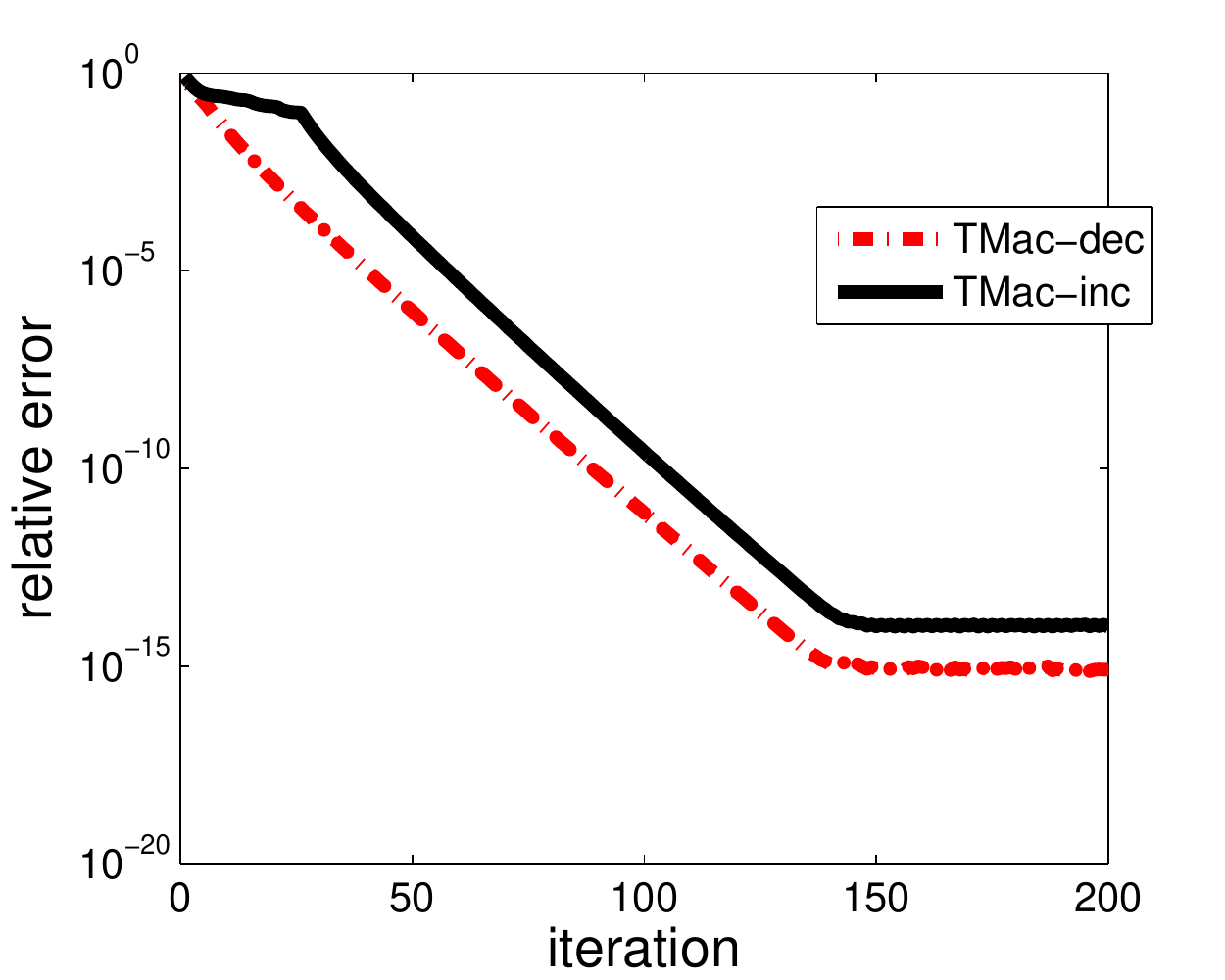}&
\includegraphics[width=0.3\textwidth]{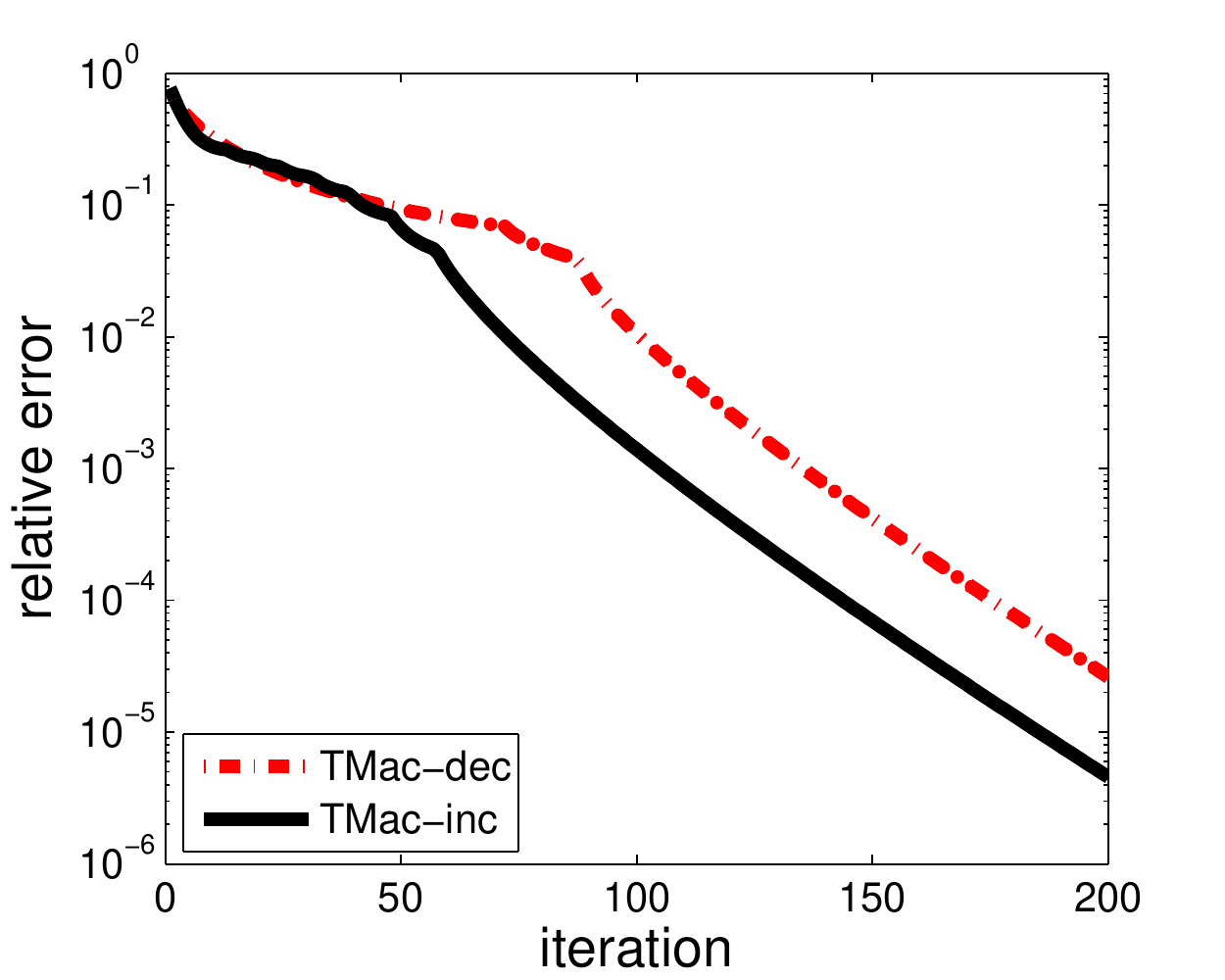}\\
\includegraphics[width=0.3\textwidth]{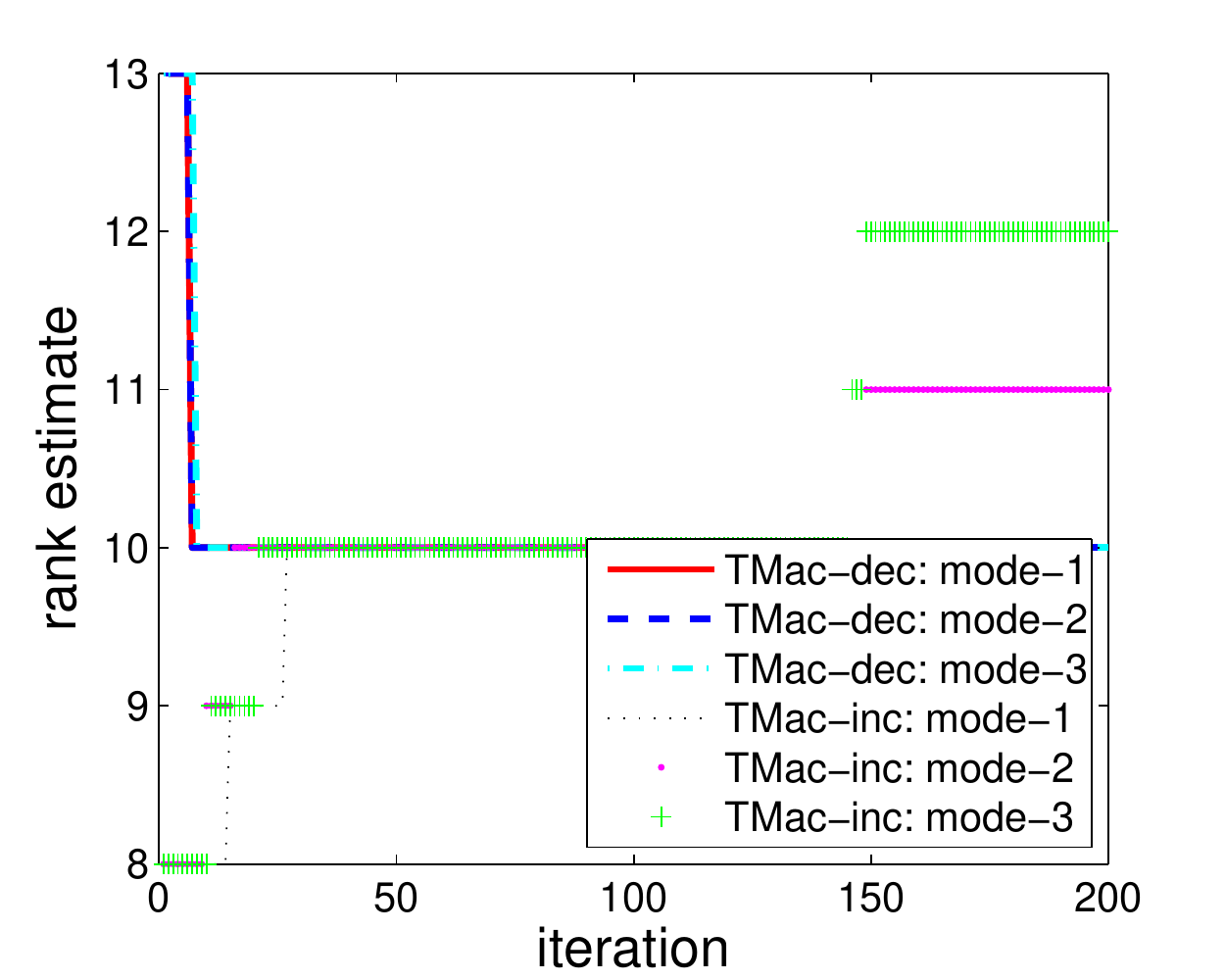}&
\includegraphics[width=0.3\textwidth]{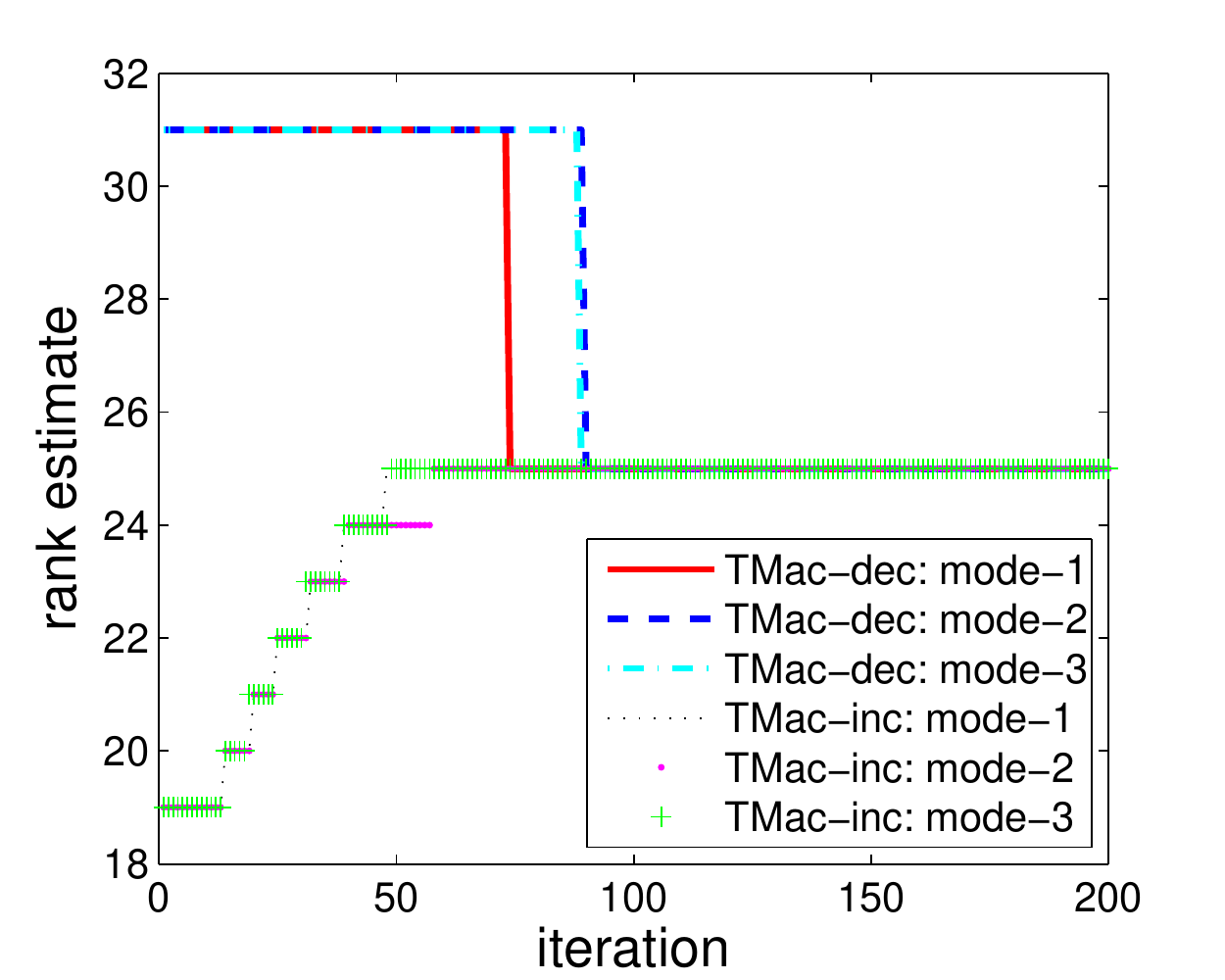}
\end{tabular}}
\end{center}
\end{figure}

\begin{figure}[H]
\captionsetup{width=0.99\textwidth}
\caption{Brain MRI images: one original slice, the corresponding slice with 90\% pixels missing and 5\% Gaussian noise, and the recovered slices by different tensor completion methods.}\label{fig:mri}
\begin{center}
{\small
\begin{tabular}{ccc}
Original & 90\% masked and 5\% noise & MatComp\\
\includegraphics[width=0.25\textwidth]{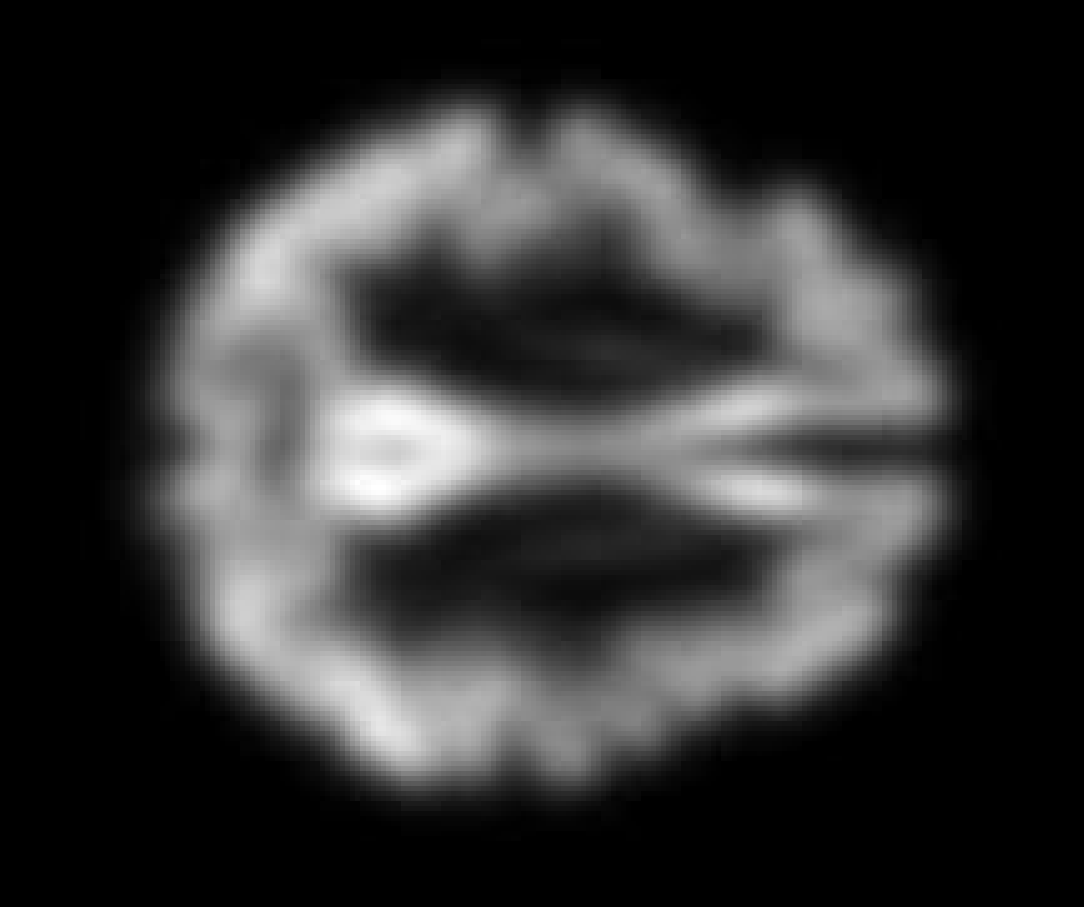}&
\includegraphics[width=0.25\textwidth]{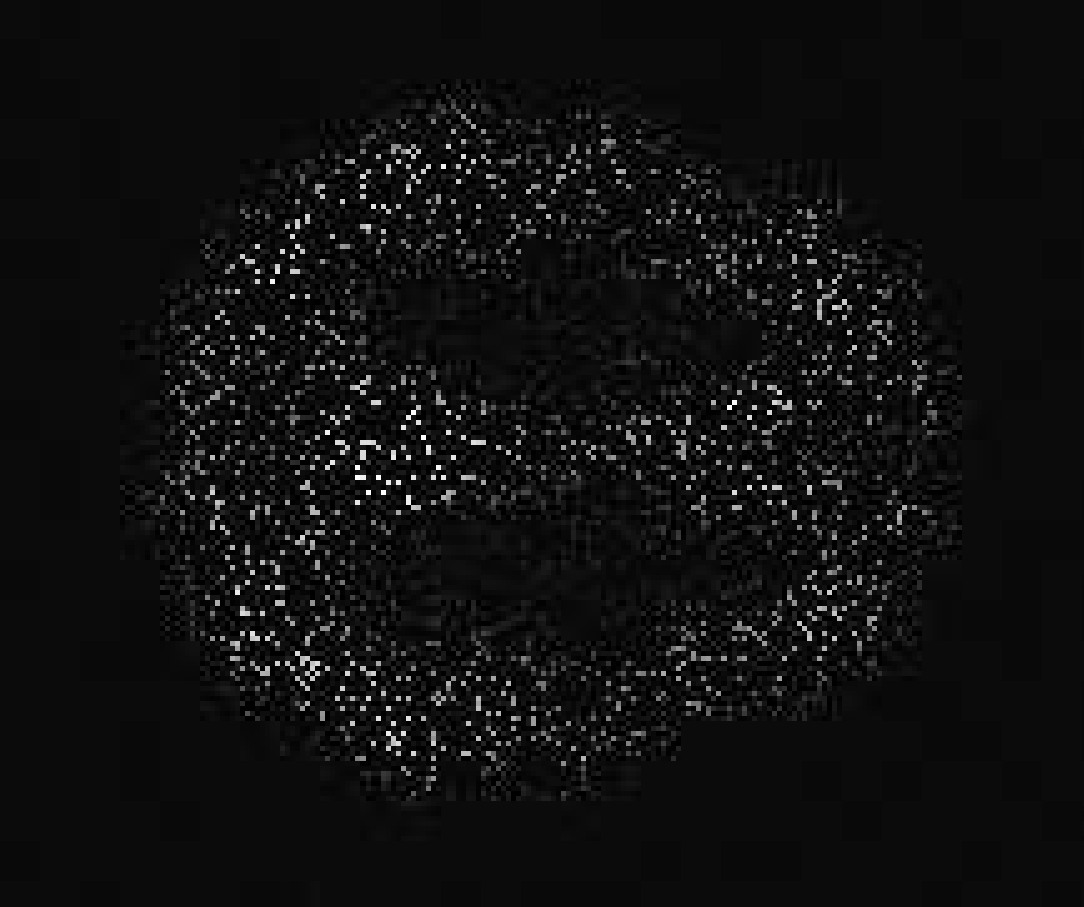}&
\includegraphics[width=0.25\textwidth]{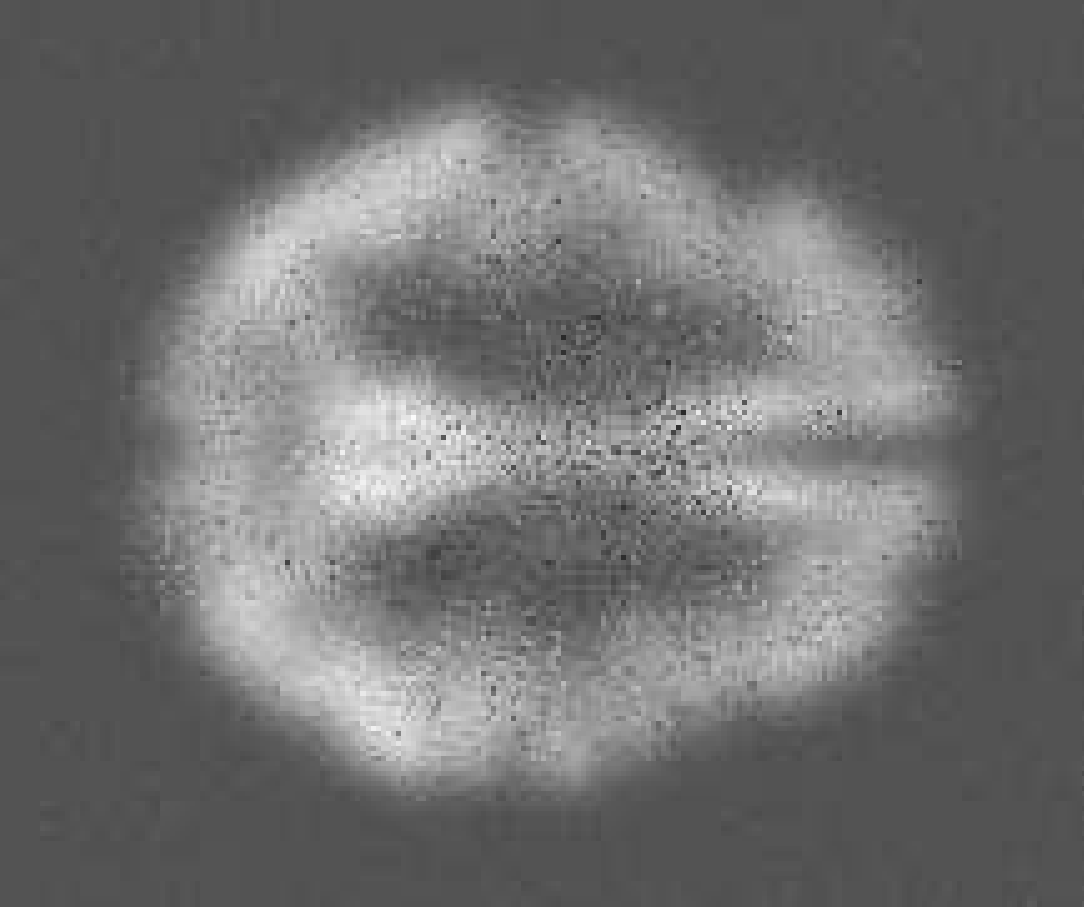}\\
FaLRTC & TMac with dynamic $\alpha_n$'s & TMac with fixed $\alpha_n$'s\\
\includegraphics[width=0.25\textwidth]{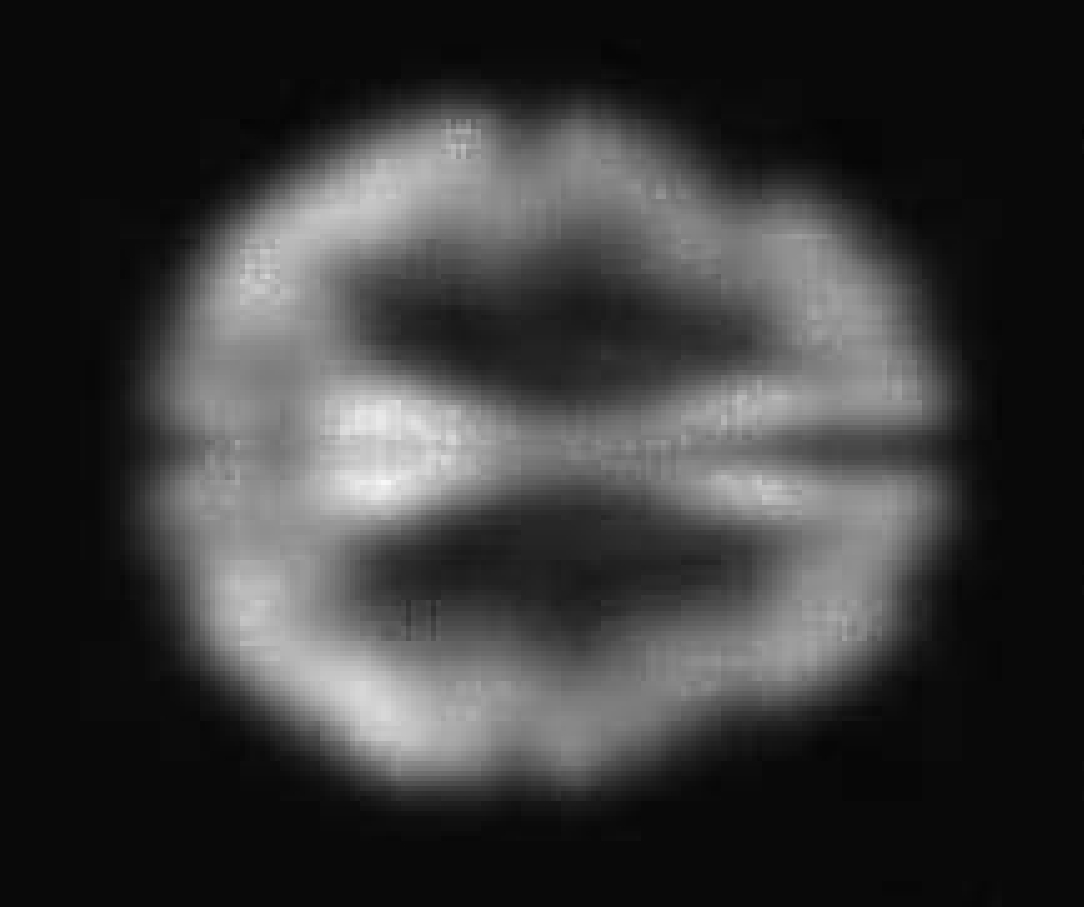}&
\includegraphics[width=0.25\textwidth]{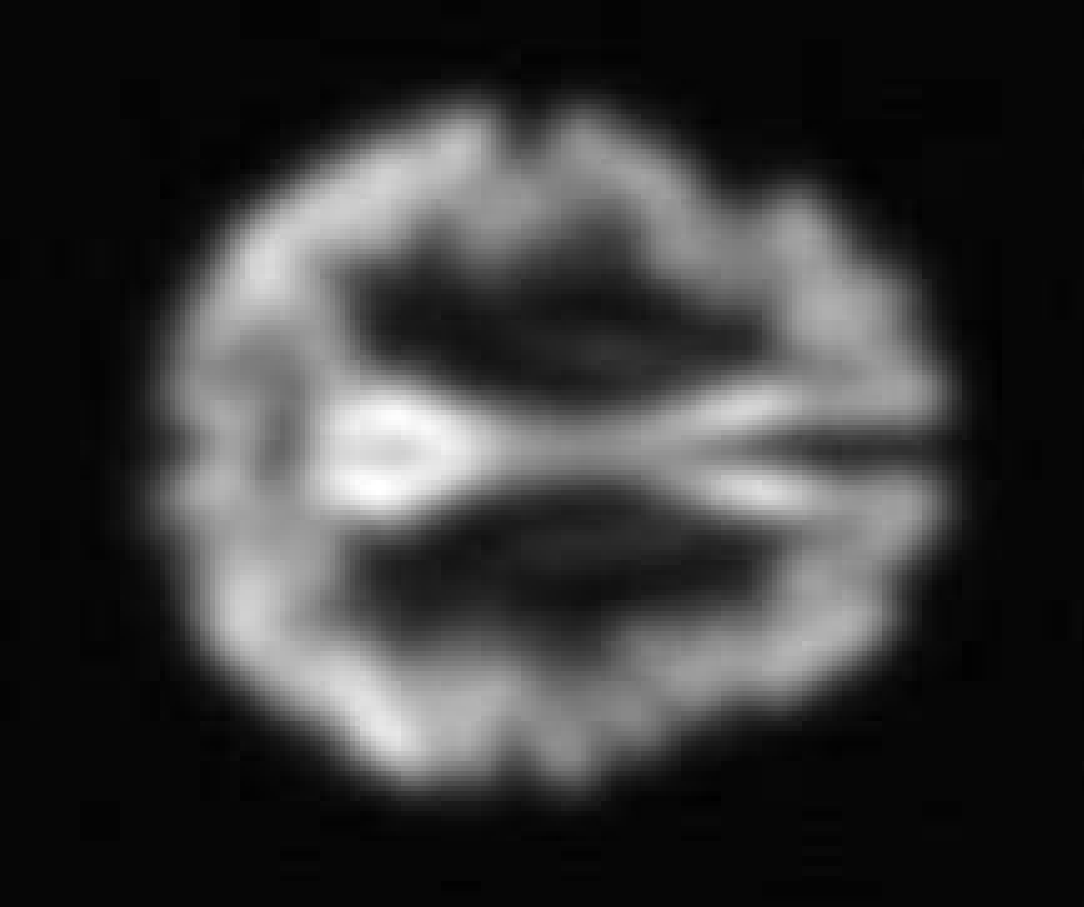}&
\includegraphics[width=0.25\textwidth]{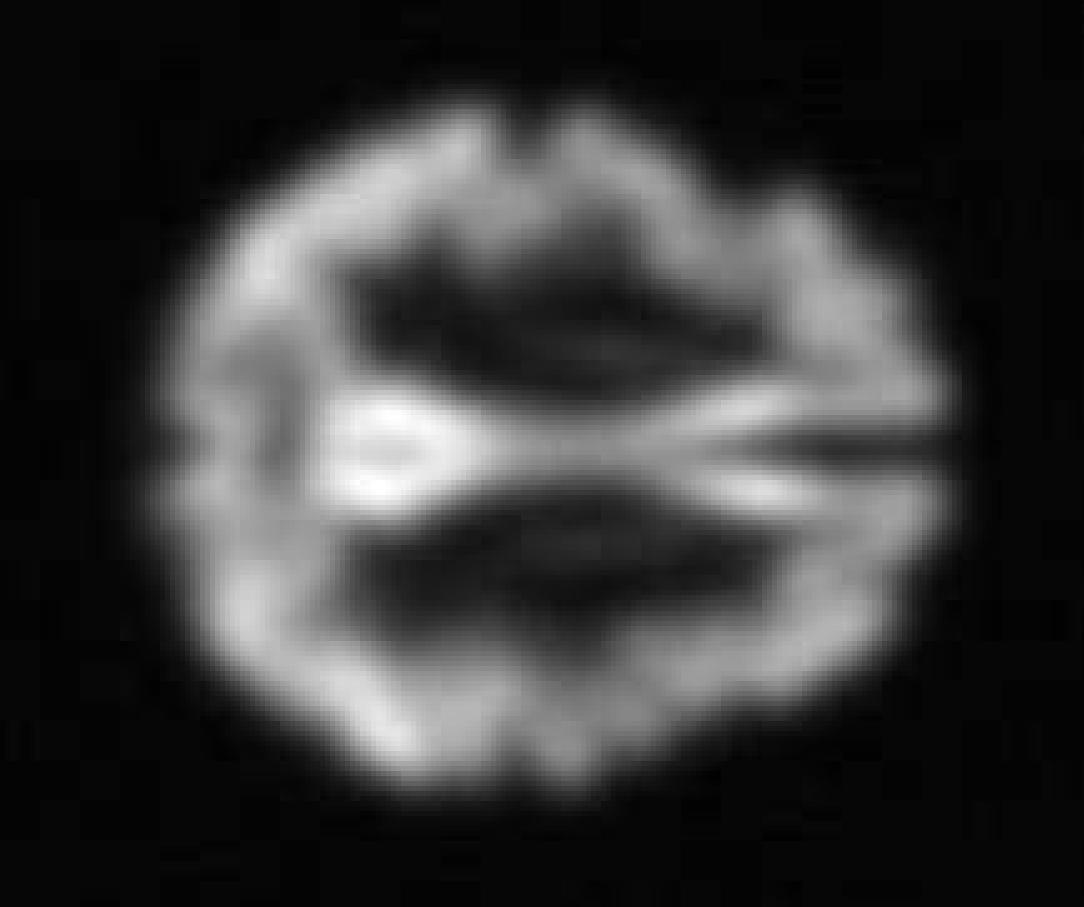}
\end{tabular}}
\end{center}
\end{figure}

\begin{table}[H]
\captionsetup{width=0.99\textwidth}
\caption{Brain MRI images: average results of 5 independent runs by different tensor completion methods for different settings of noise level $\sigma$'s and sample ratio SR's.}\label{table:mri}
\begin{center}{\small
\resizebox{\textwidth}{!}{\begin{tabular}{|c|cc|cc|cc|cc|}
\hline
 & \multicolumn{2}{|c|}{TMac with} & \multicolumn{2}{|c|}{TMac with} & \multicolumn{2}{|c|}{\multirow{2}{*}{MatComp}} & \multicolumn{2}{|c|}{\multirow{2}{*}{FaLRTC}}\\
& \multicolumn{2}{|c|}{dynamic $\alpha_n$'s} & \multicolumn{2}{|c|}{fixed $\alpha_n$'s} & \multicolumn{2}{|c|}{} & \multicolumn{2}{|c|}{}\\\hline
SR& relerr & time & relerr & time & relerr & time & relerr & time\\\hline
\multicolumn{9}{|c|}{noise level $\sigma=0$}\\\hline
10\% & 1.54e-03 & 1.79e+02 & 1.56e-03 & 1.82e+02 & 2.39e-01 & 2.40e+01 & 8.67e-02 & 2.65e+02\\
30\%& 1.14e-03 & 9.89e+01 & 1.15e-03 & 9.28e+01 & 2.13e-02 & 1.98e+01 & 9.59e-03 & 2.47e+02\\
50\%& 8.55e-04 & 8.11e+01 & 1.00e-03 & 7.63e+01 & 3.19e-03 & 1.84e+01 & 7.34e-03 & 2.01e+02\\\hline
\multicolumn{9}{|c|}{noise level $\sigma = 0.05$}\\\hline
10\% & 2.15e-02 & 1.39e+02 & 2.15e-02 & 1.43e+02 & 2.55e-01 & 3.00e+01 & 1.15e-01 & 2.96e+02\\
30\% & 1.67e-02 & 9.04e+01 & 1.66e-02 & 8.71e+01 & 8.10e-02 & 3.12e+01 & 3.86e-02 & 1.43e+02\\
50\% & 1.62e-02 & 8.11e+01 & 1.61e-02 & 7.84e+01 & 4.35e-02 & 2.26e+01 & 3.66e-02 & 1.36e+02\\\hline
\multicolumn{9}{|c|}{noise level $\sigma = 0.10$}\\\hline
10\% & 4.34e-02 & 1.26e+02 & 4.34e-02 & 1.30e+02 & 3.00e-01 & 3.33e+01 & 1.48e-01 & 2.46e+02\\
30\% & 3.37e-02 & 7.69e+01 & 3.33e-02 & 7.81e+01 & 1.66e-01 & 3.16e+01 & 7.19e-02 & 1.05e+02\\
50\% & 3.25e-02 & 7.22e+01 & 3.33e-02 & 7.81e+01 & 8.61e-02 & 2.12e+01 & 7.17e-02 & 1.01e+02\\\hline
\end{tabular}}}
\end{center}
\end{table}

\begin{figure}[H]
\captionsetup{width=0.99\textwidth}
\caption{Hyperspectral images: one original slice, the corresponding slice with 90\% pixels missing and 5\% Gaussian noise, and the recovered slices by different tensor completion methods.}\label{fig:hyp}
\begin{center}
{\small
\begin{tabular}{ccc}
Original & 90\% masked and 5\% noise & MatComp\\
\includegraphics[width=0.25\textwidth]{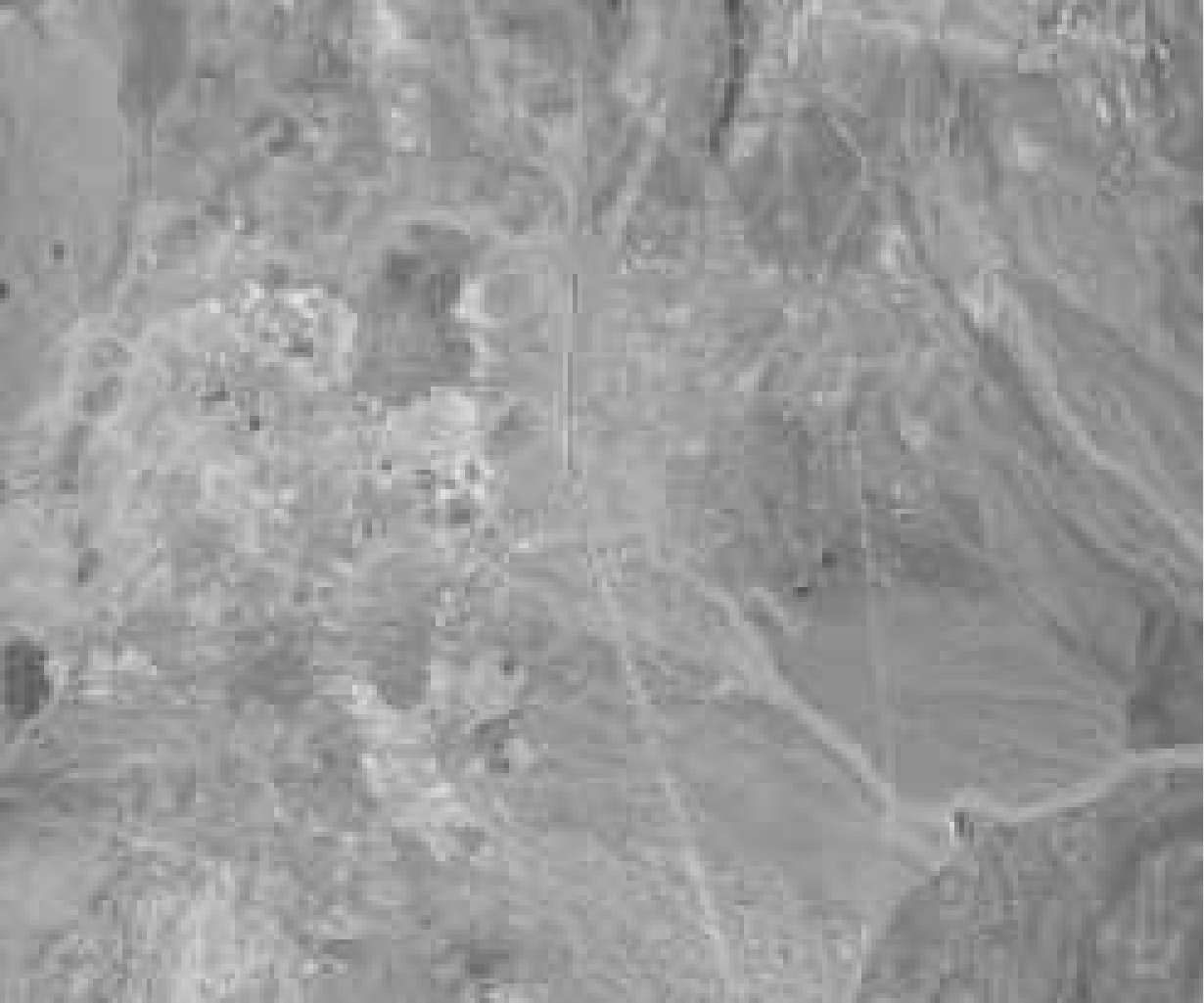}&
\includegraphics[width=0.25\textwidth]{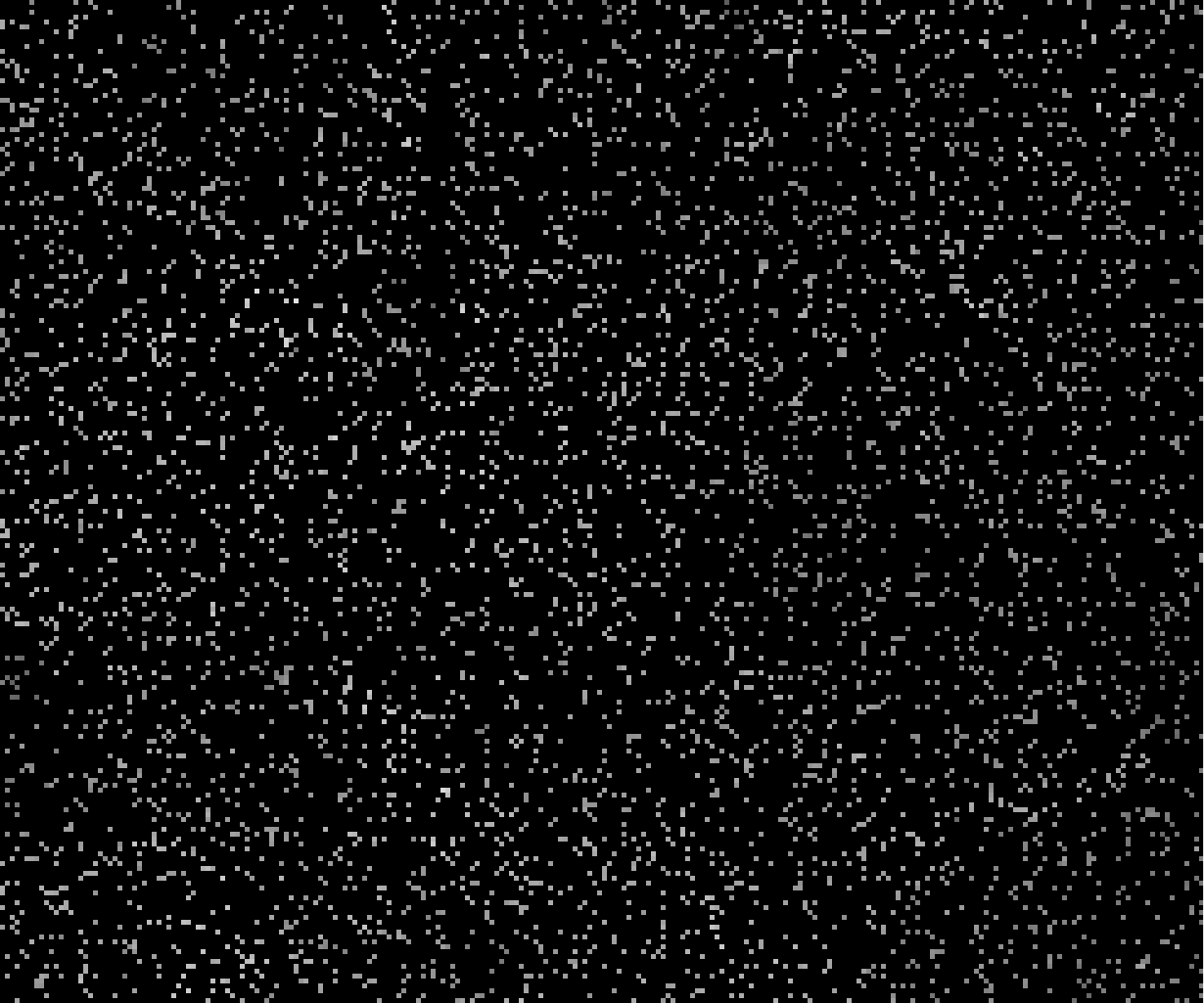}&
\includegraphics[width=0.25\textwidth]{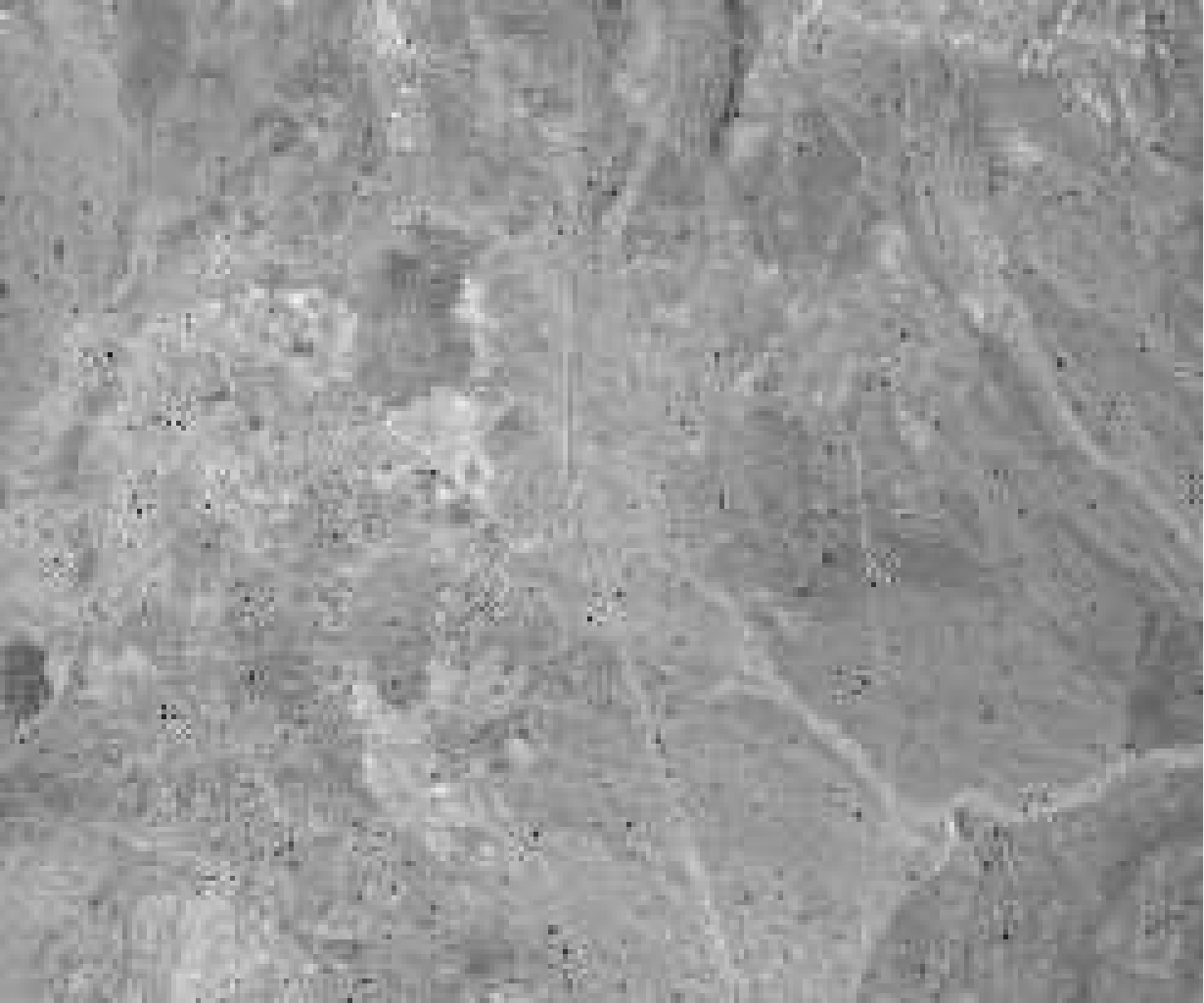}\\
FaLRTC & TMac with dynamic $\alpha_n$'s & TMac with fixed $\alpha_n$'s\\
\includegraphics[width=0.25\textwidth]{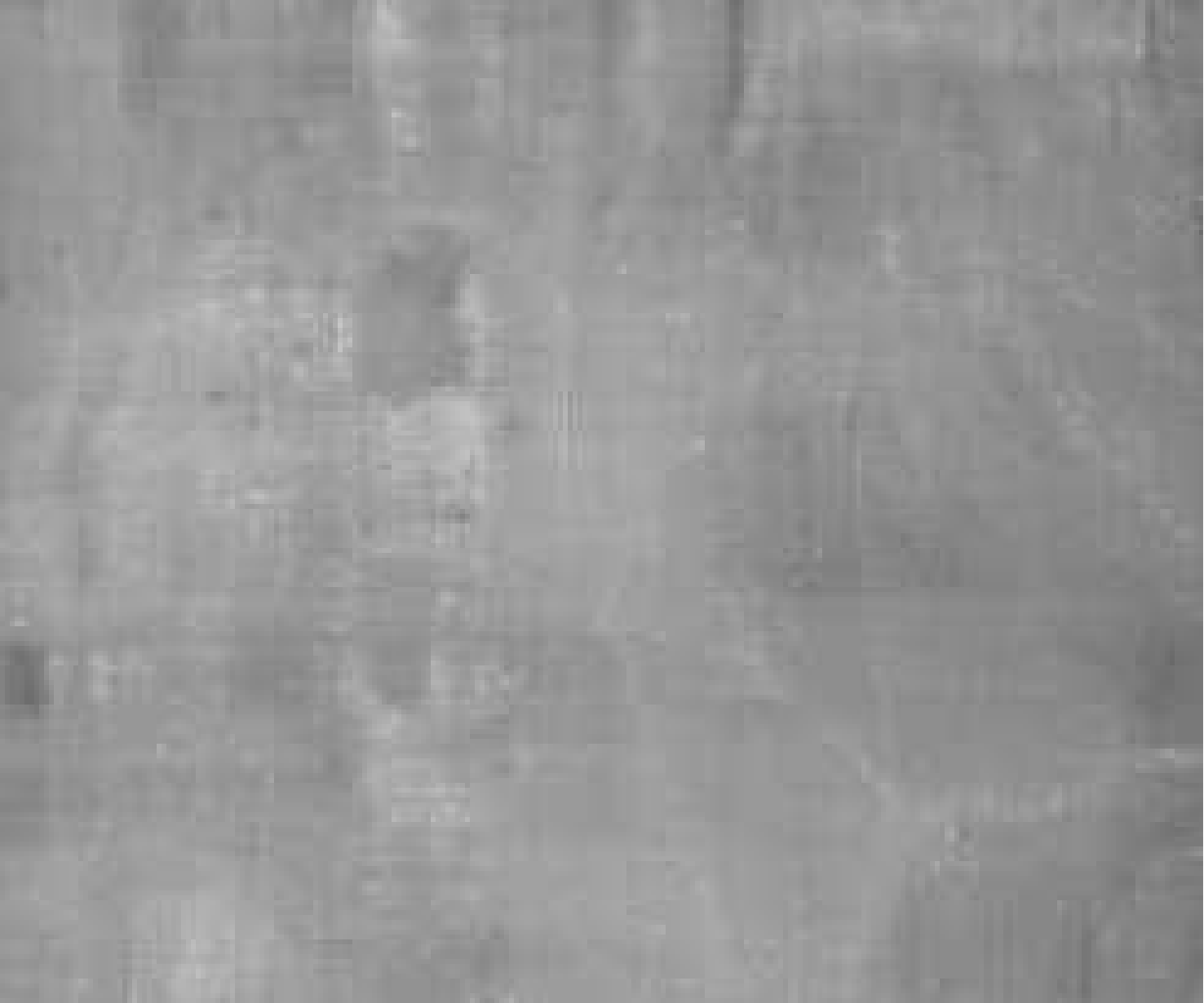}&
\includegraphics[width=0.25\textwidth]{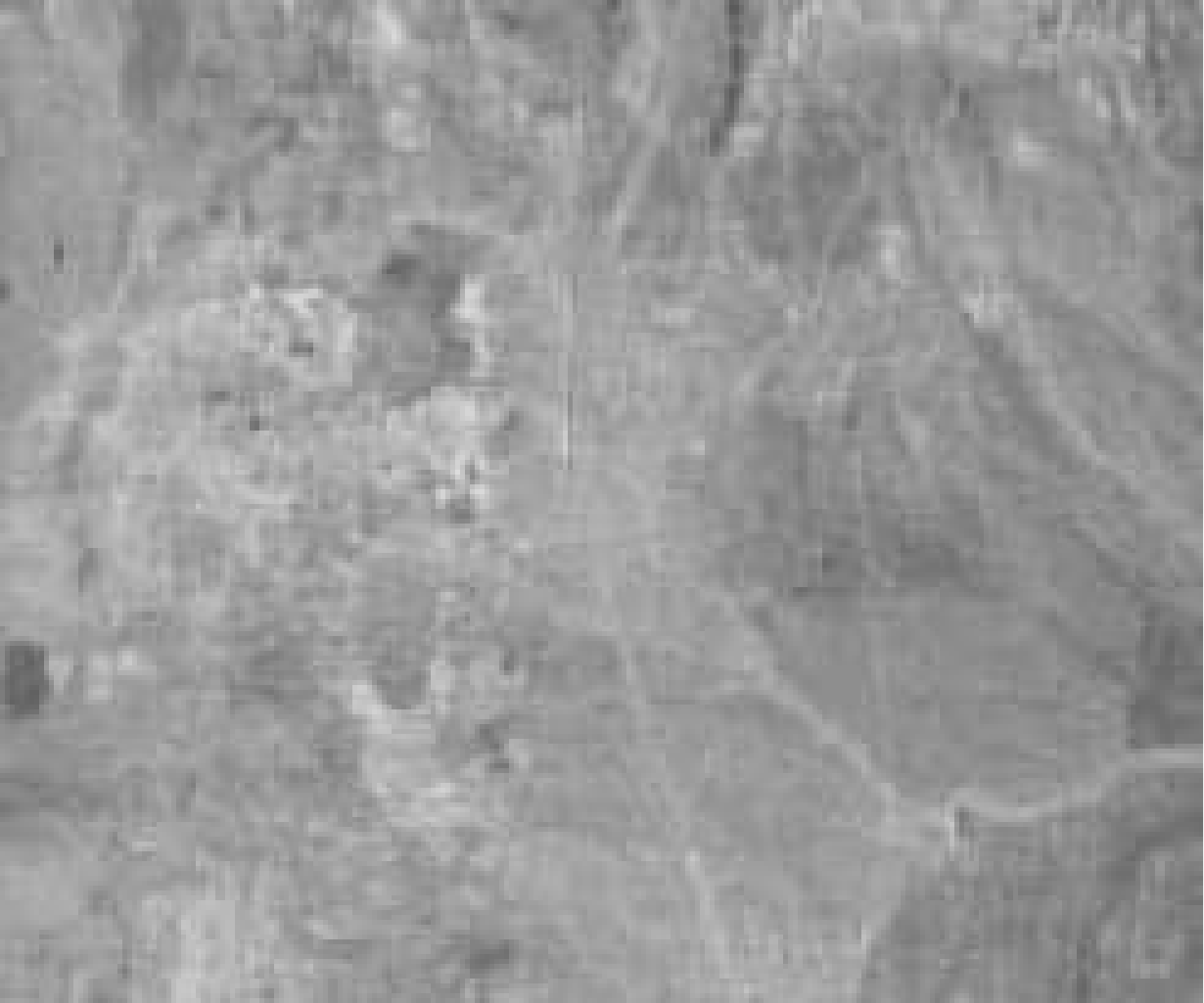}&
\includegraphics[width=0.25\textwidth]{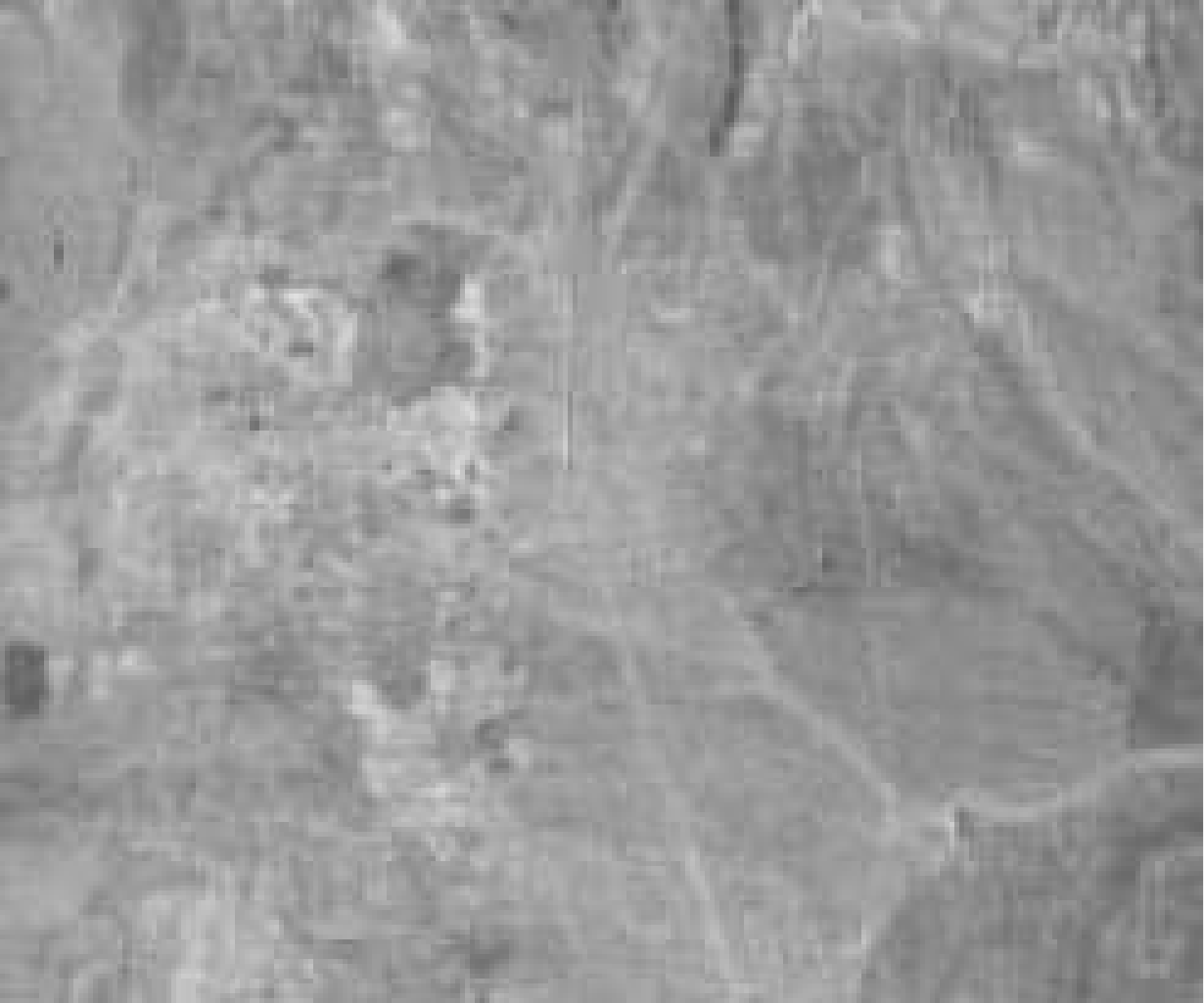}
\end{tabular}}
\end{center}
\end{figure}

\begin{table}[H]
\captionsetup{width=0.99\textwidth}
\caption{Hyperspectral images: average results of 5 independent runs by different tensor completion methods for different settings of noise level $\sigma$'s and sample ratio SR's.}\label{table:hyp}
\begin{center}{\small
\resizebox{\textwidth}{!}{\begin{tabular}{|c|cc|cc|cc|cc|}
\hline
 & \multicolumn{2}{|c|}{TMac with} & \multicolumn{2}{|c|}{TMac with} & \multicolumn{2}{|c|}{\multirow{2}{*}{MatComp}} & \multicolumn{2}{|c|}{\multirow{2}{*}{FaLRTC}}\\
& \multicolumn{2}{|c|}{dynamic $\alpha_n$'s} & \multicolumn{2}{|c|}{fixed $\alpha_n$'s} & \multicolumn{2}{|c|}{} & \multicolumn{2}{|c|}{}\\\hline
SR& relerr & time & relerr & time & relerr & time & relerr & time\\\hline
\multicolumn{9}{|c|}{noise level $\sigma=0$}\\\hline
10\% & 4.00e-02 & 8.88e+01 & 3.91e-02 & 8.28e+01 & 3.15e-01 & 1.22e+01 & 6.48e-02 & 1.70e+02\\
30\% & 2.40e-02 & 5.16e+01 & 3.08e-02 & 5.55e+01 & 6.25e-02 & 1.86e+01 & 3.10e-02 & 1.56e+02\\
50\% & 6.35e-03 & 4.54e+01 & 2.73e-02 & 5.71e+01 & 1.71e-02 & 2.68e+01 & 1.68e-02 & 1.64e+02\\\hline
\multicolumn{9}{|c|}{noise level $\sigma = 0.05$}\\\hline
10\% & 4.14e-02 & 9.40e+01 & 4.12e-02 & 8.91e+01 & 3.16e-01 & 1.34e+01 & 6.65e-02 & 1.61e+02\\
30\% & 2.98e-02 & 5.76e+01 & 3.43e-02 & 6.08e+01 & 7.65e-02 & 2.71e+01 & 3.52e-02 & 1.53e+02\\
50\% & 2.25e-02 & 5.50e+01 & 3.01e-02 & 5.93e+01 & 4.14e-02 & 3.87e+01 & 2.45e-02 & 1.38e+02\\\hline
\multicolumn{9}{|c|}{noise level $\sigma = 0.10$}\\\hline
10\% & 4.52e-02 & 9.18e+01 & 4.53e-02 & 9.04e+01 & 3.25e-01 & 1.47e+01 & 6.99e-02 & 1.60e+02\\
30\% & 3.53e-02 & 5.63e+01 & 3.77e-02 & 6.31e+01 & 1.00e-01 & 4.07e+01 & 4.35e-02 & 1.35e+02\\
50\% & 3.23e-02 & 5.49e+01 & 3.41e-02 & 5.10e+01 & 8.03e-02 & 4.08e+01 & 3.62e-02 & 1.23e+02\\\hline
\end{tabular}}}
\end{center}
\end{table}

\begin{figure}[H]
\captionsetup{width=0.99\textwidth}
\caption{Grayscale video: one original frame, the corresponding frame with 70\% pixels missing and 5\% Gaussian noise, and the recovered frames by different tensor completion methods.}\label{fig:grayvideo}
\begin{center}
{\small
\begin{tabular}{ccc}
Original & 70\% masked and 5\% noise & MatComp\\
\includegraphics[width=0.25\textwidth]{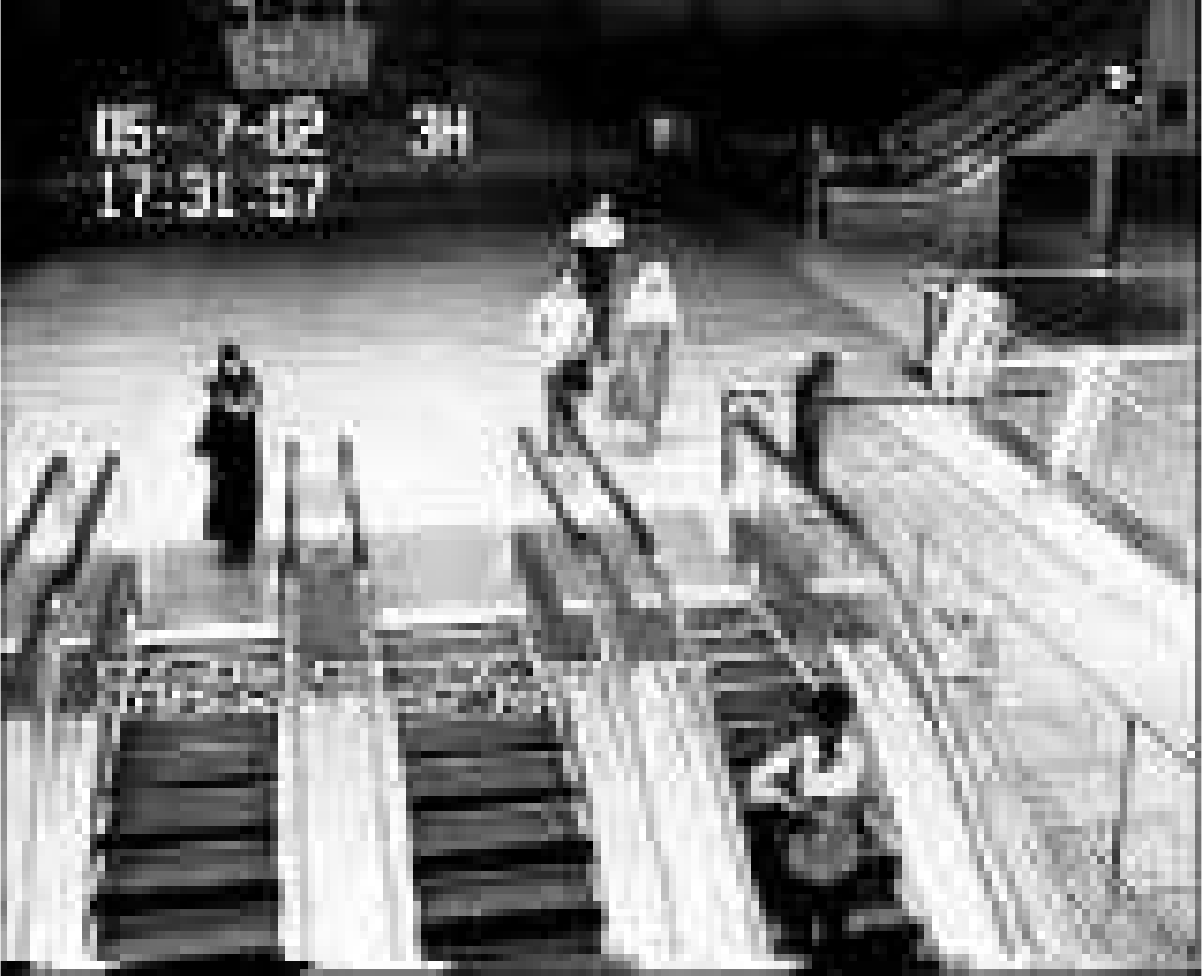}&
\includegraphics[width=0.25\textwidth]{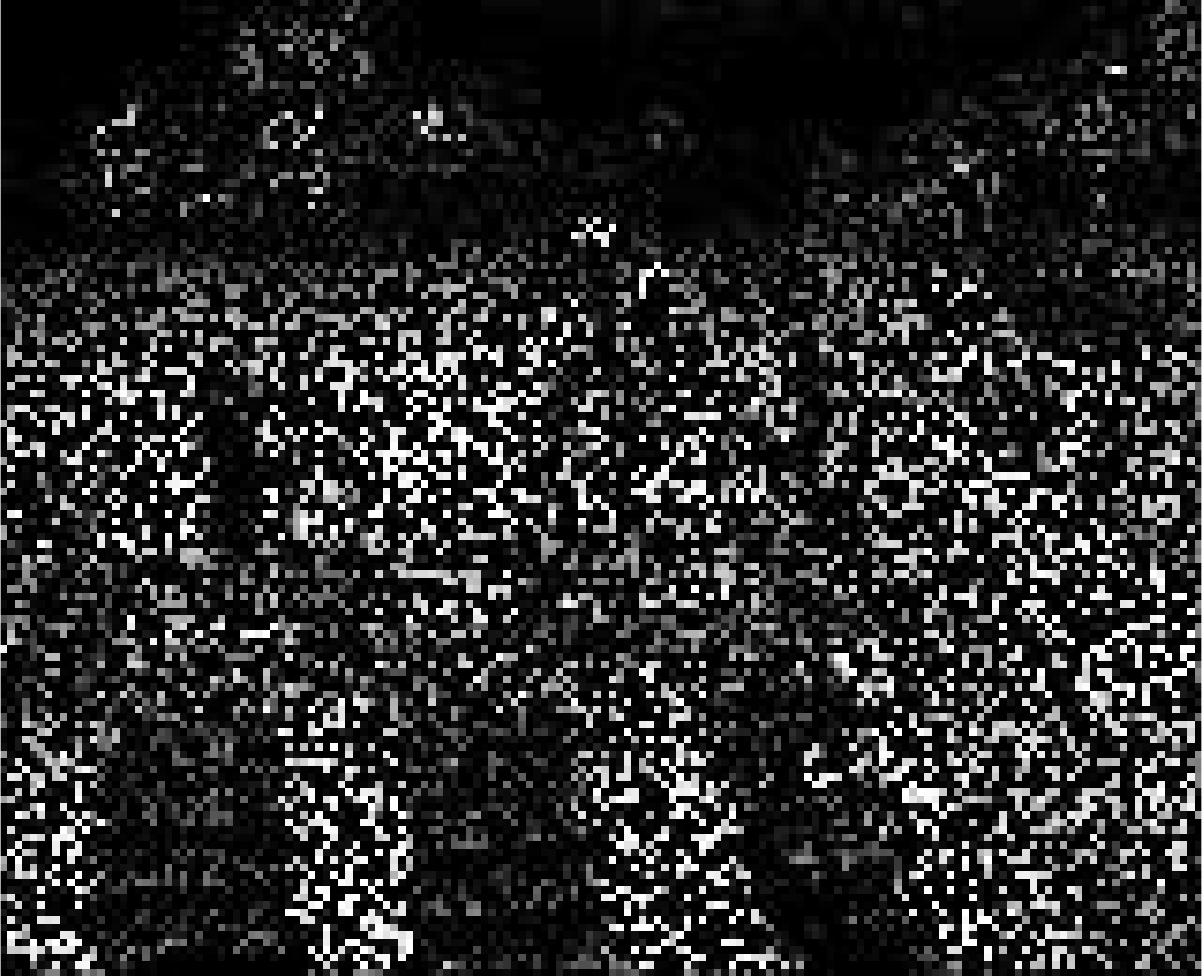}&
\includegraphics[width=0.25\textwidth]{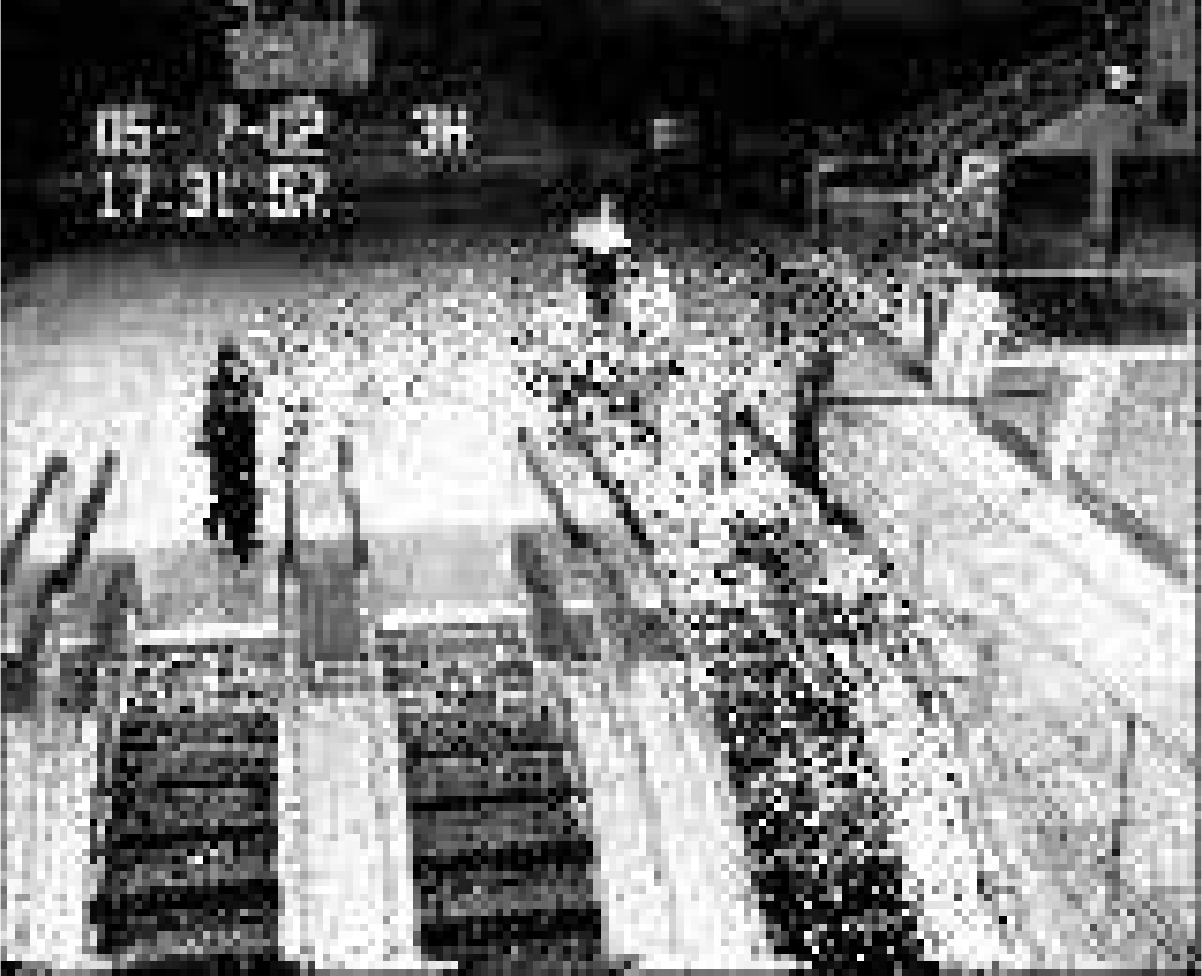}\\
FaLRTC & TMac with dynamic $\alpha_n$'s & TMac with fixed $\alpha_n$'s\\
\includegraphics[width=0.25\textwidth]{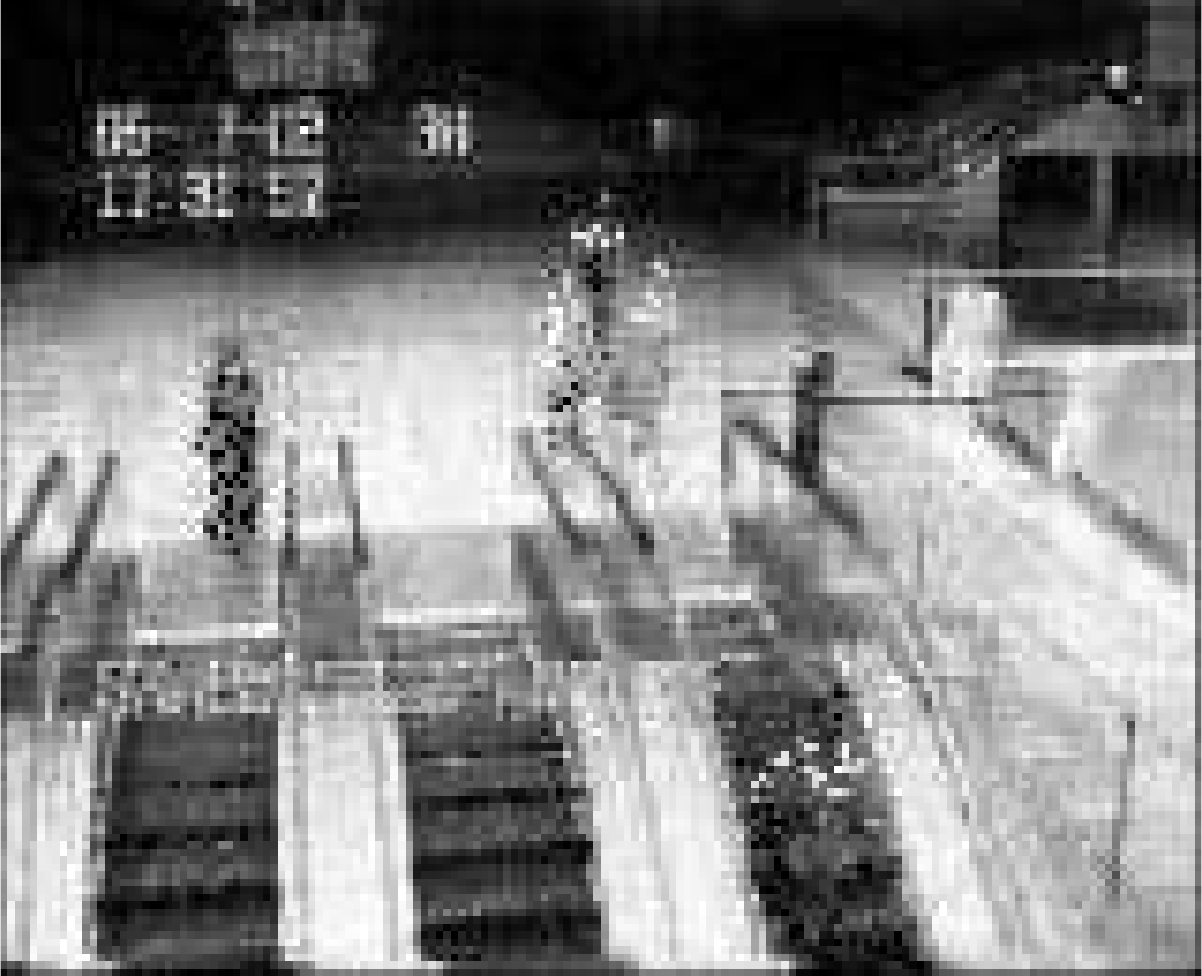}&
\includegraphics[width=0.25\textwidth]{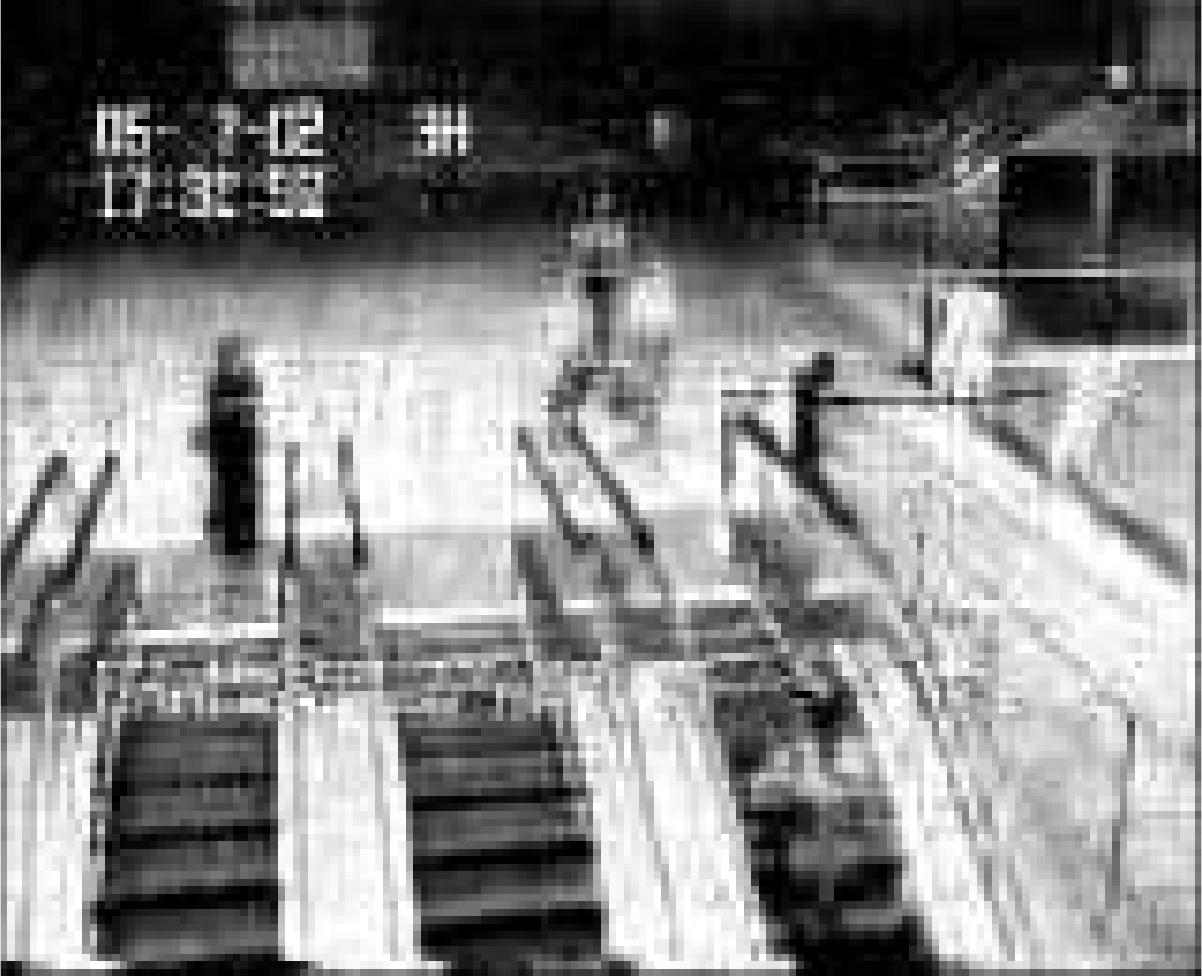}&
\includegraphics[width=0.25\textwidth]{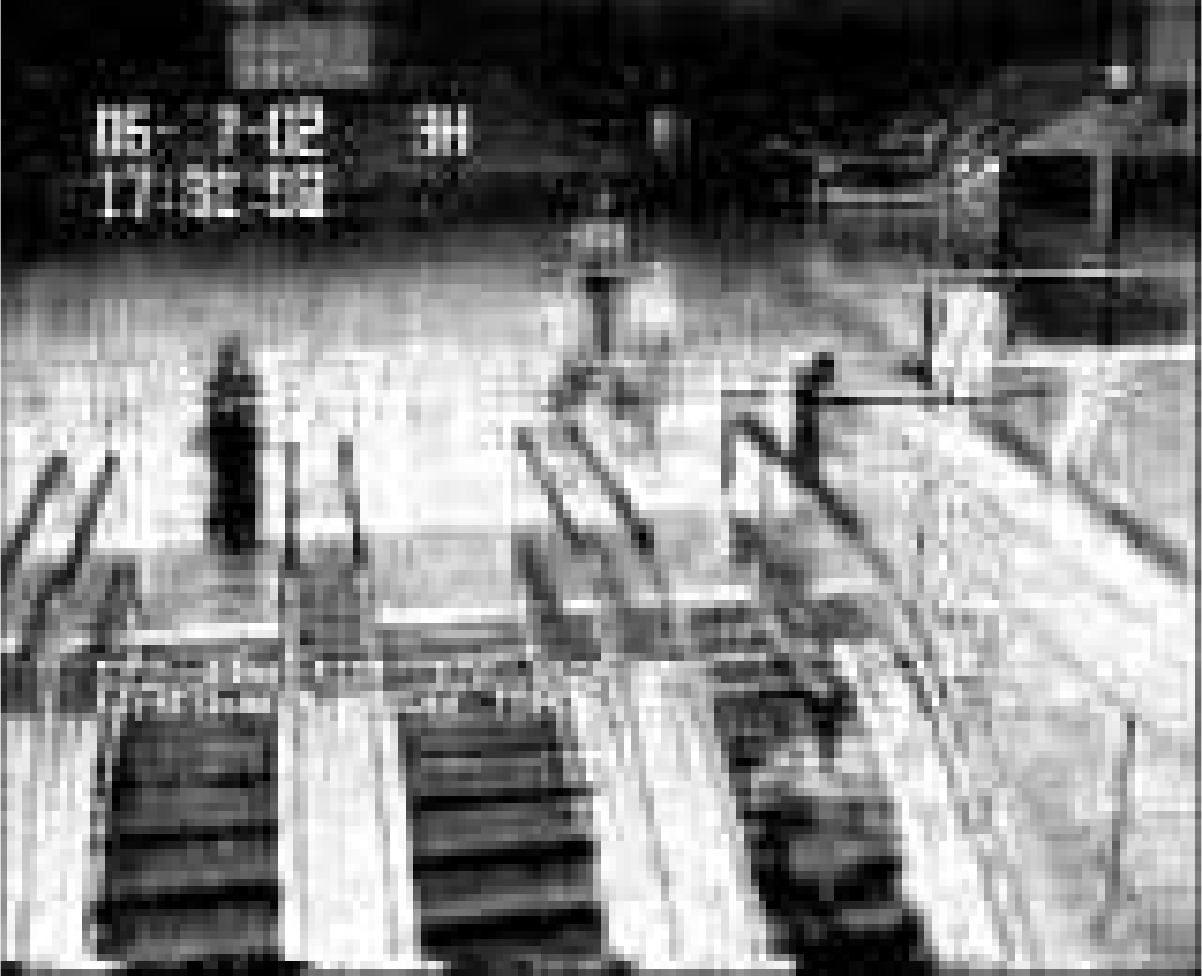}
\end{tabular}}
\end{center}
\end{figure}

\begin{table}[H]
\captionsetup{width=0.99\textwidth}
\caption{Grayscale video: average results of 5 independent runs by different tensor completion methods for different settings of noise level $\sigma$'s and sample ratio SR's.}\label{table:grayvideo}
\begin{center}{\small
\resizebox{\textwidth}{!}{\begin{tabular}{|c|cc|cc|cc|cc|}
\hline
 & \multicolumn{2}{|c|}{TMac with} & \multicolumn{2}{|c|}{TMac with} & \multicolumn{2}{|c|}{\multirow{2}{*}{MatComp}} & \multicolumn{2}{|c|}{\multirow{2}{*}{FaLRTC}}\\
& \multicolumn{2}{|c|}{dynamic $\alpha_n$'s} & \multicolumn{2}{|c|}{fixed $\alpha_n$'s} & \multicolumn{2}{|c|}{} & \multicolumn{2}{|c|}{}\\\hline
SR& relerr & time & relerr & time & relerr & time & relerr & time\\\hline
\multicolumn{9}{|c|}{noise level $\sigma=0$}\\\hline
10\% & 1.45e-01 & 9.23e+01 & 1.41e-01 & 9.77e+01 & 2.51e-01 & 1.88e+01 & 2.35e-01 & 1.22e+02\\
30\% & 9.85e-02 & 5.37e+01 & 1.01e-01 & 5.77e+01 & 1.41e-01 & 2.24e+01 & 1.28e-01 & 6.55e+01\\
50\% & 7.78e-02 & 4.75e+01 & 8.51e-02 & 5.27e+01 & 8.86e-02 & 1.80e+01 & 8.07e-02 & 7.38e+01\\\hline
\multicolumn{9}{|c|}{noise level $\sigma = 0.05$}\\\hline
10\% & 1.45e-01 & 9.79e+01 & 1.41e-01 & 9.49e+01 & 2.47e-01 & 2.13e+01 & 2.36e-01 & 1.36e+02\\
30\% & 9.89e-02 & 5.62e+01 & 1.01e-01 & 5.59e+01 & 1.40e-01 & 1.99e+01 & 1.29e-01 & 6.95e+01\\
50\% & 7.80e-02 & 5.03e+01 & 8.54e-02 & 5.45e+01 & 8.90e-02 & 1.66e+01 & 8.27e-02 & 7.35e+01\\\hline
\multicolumn{9}{|c|}{noise level $\sigma = 0.10$}\\\hline
10\% & 1.46e-01 & 9.68e+01 & 1.43e-01 & 5.45e+01 & 2.47e-01 & 1.93e+01 & 2.38e-01 & 1.42e+02\\
30\% & 1.00e-01 & 5.72e+01 & 1.02e-01 & 5.64e+01 & 1.57e-01 & 2.24e+01 & 1.32e-01 & 6.83e+01\\
50\% & 8.00e-02 & 5.29e+01 & 8.64e-02 & 5.02e+01 & 9.12e-02 & 1.81e+01 & 8.81e-02 & 6.79e+01\\\hline
\end{tabular}}}
\end{center}
\end{table}

\begin{figure}[H]
\captionsetup{width=0.99\textwidth}
\caption{Color video: one original frame, the corresponding frame with 70\% pixels missing and 5\% Gaussian noise, and the recovered frames by different tensor completion methods.}\label{fig:colorvideo}
\begin{center}
{\small
\begin{tabular}{ccc}
Original & 70\% masked and 5\% noise & MatComp\\
\includegraphics[width=0.25\textwidth]{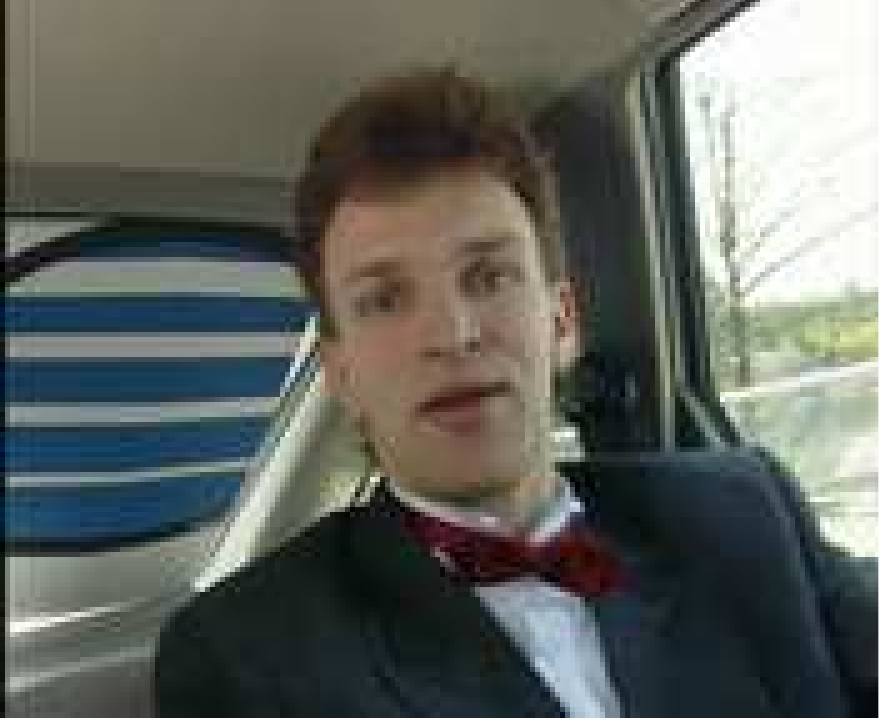}&
\includegraphics[width=0.25\textwidth]{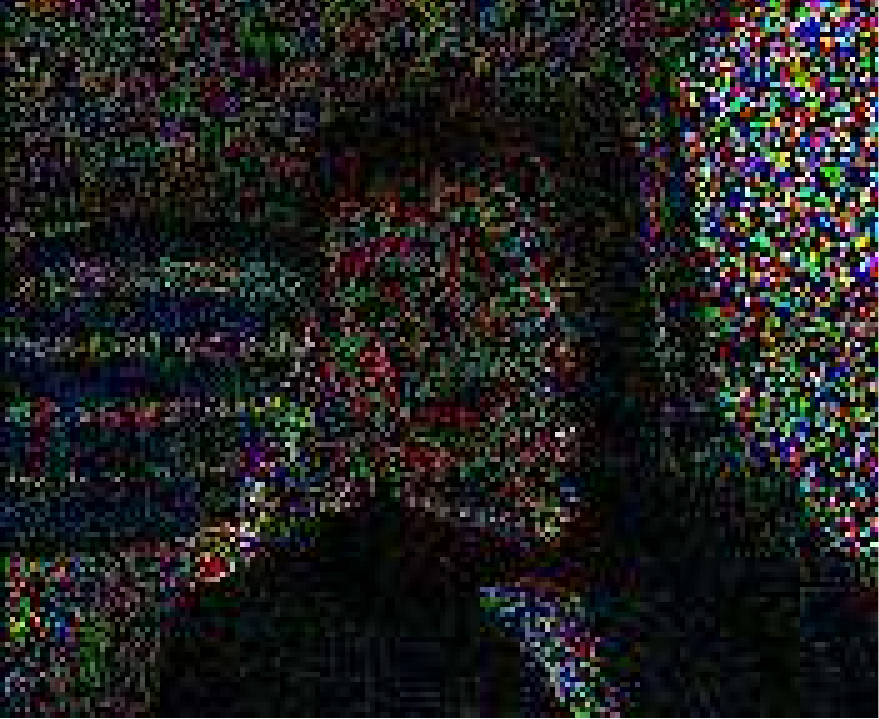}&
\includegraphics[width=0.25\textwidth]{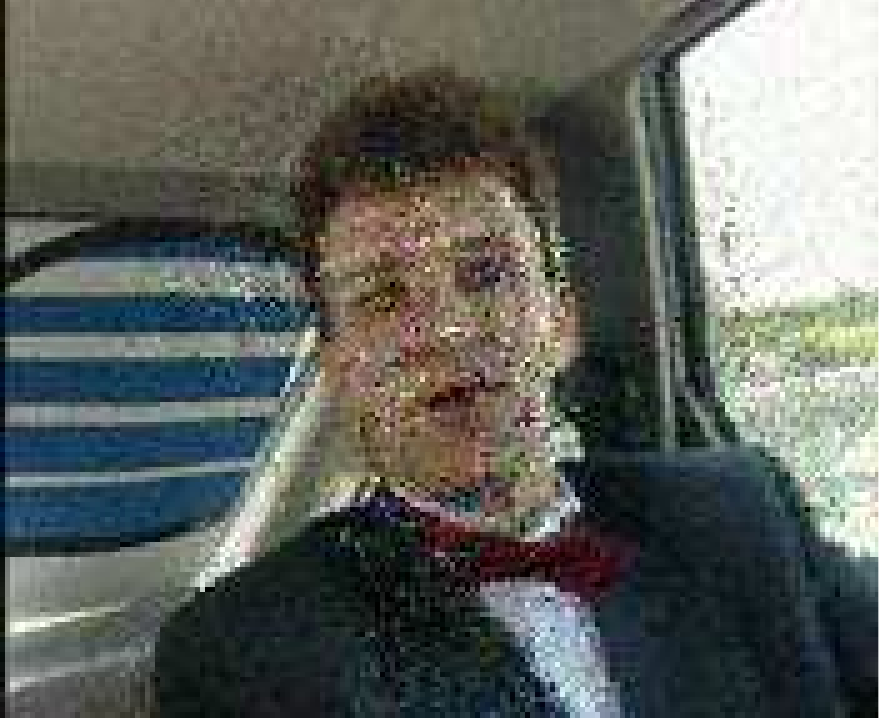}\\
FaLRTC & TMac with dynamic $\alpha_n$'s & TMac with fixed $\alpha_n$'s\\
\includegraphics[width=0.25\textwidth]{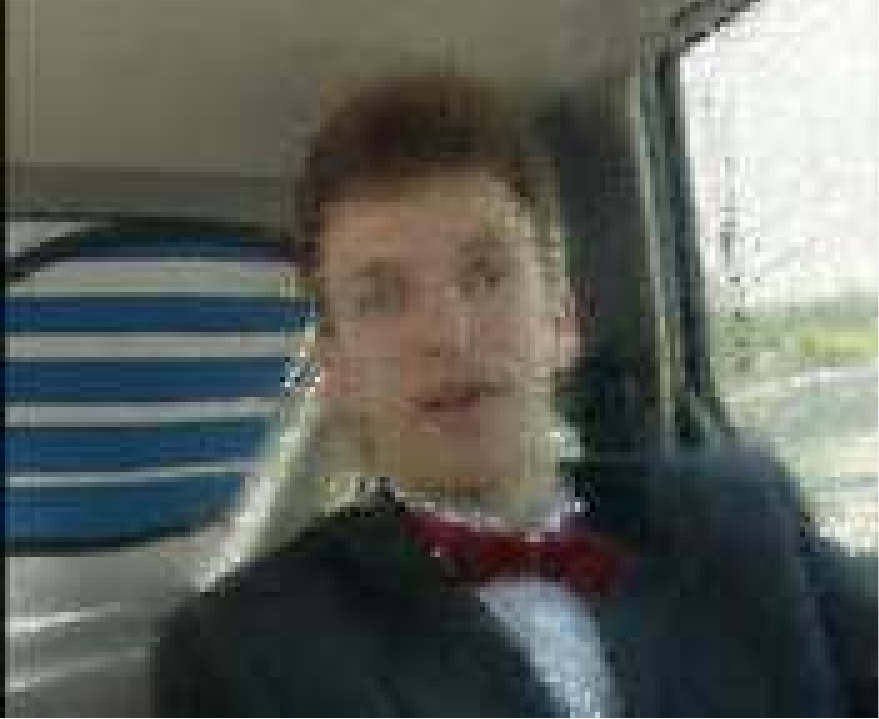}&
\includegraphics[width=0.25\textwidth]{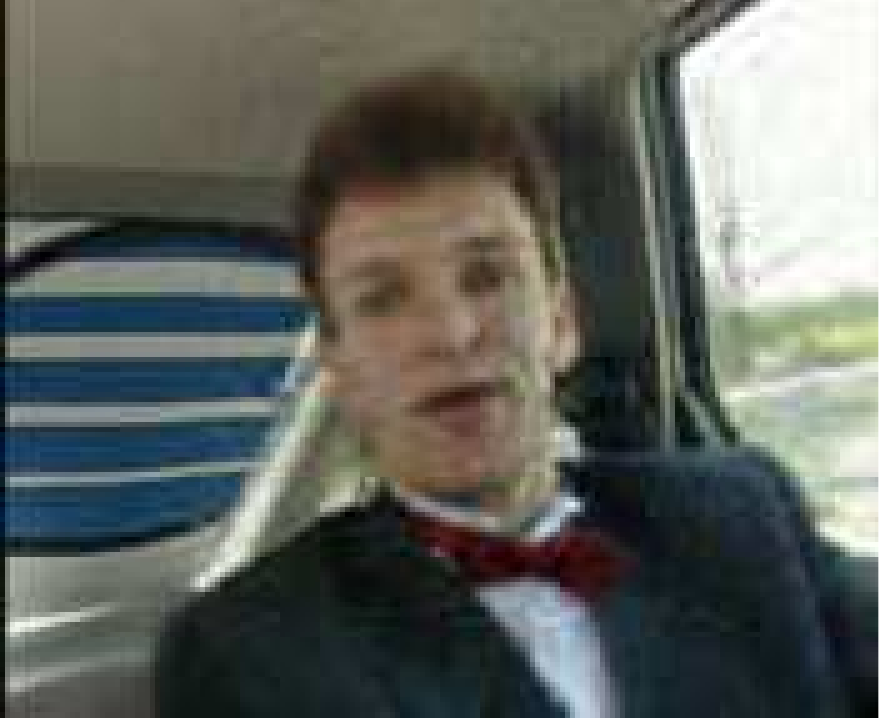}&
\includegraphics[width=0.25\textwidth]{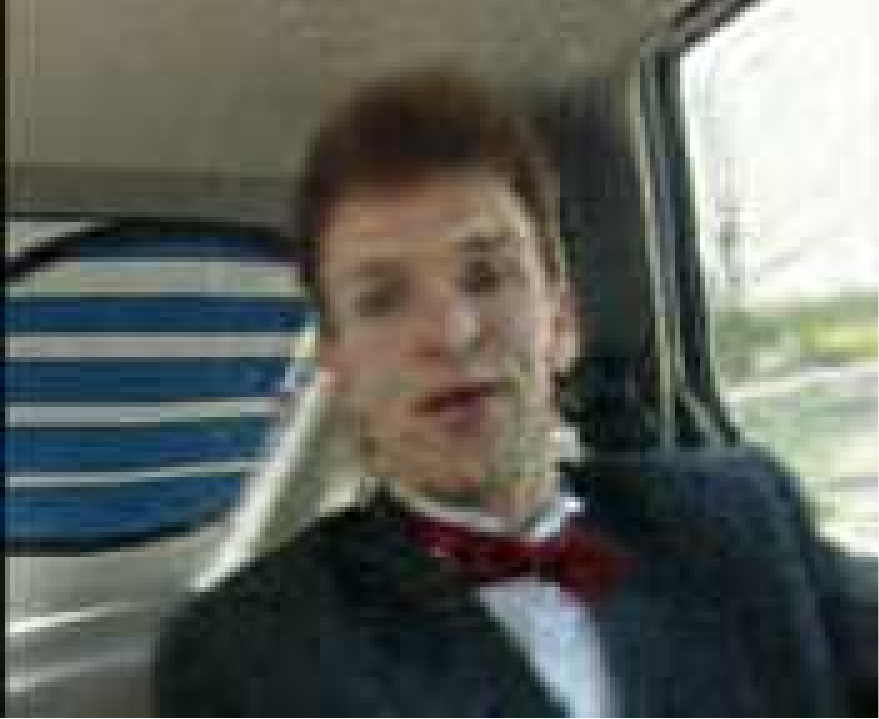}
\end{tabular}}
\end{center}
\end{figure}

\begin{table}[H]
\captionsetup{width=0.99\textwidth}
\caption{Color video: average results of 5 independent runs by different tensor completion methods for different settings of noise level $\sigma$'s and sample ratio SR's.}\label{table:colorvideo}
\begin{center}{\small
\resizebox{\textwidth}{!}{\begin{tabular}{|c|cc|cc|cc|cc|}
\hline
 & \multicolumn{2}{|c|}{TMac with} & \multicolumn{2}{|c|}{TMac with} & \multicolumn{2}{|c|}{\multirow{2}{*}{MatComp}} & \multicolumn{2}{|c|}{\multirow{2}{*}{FaLRTC}}\\
& \multicolumn{2}{|c|}{dynamic $\alpha_n$'s} & \multicolumn{2}{|c|}{fixed $\alpha_n$'s} & \multicolumn{2}{|c|}{} & \multicolumn{2}{|c|}{}\\\hline
SR& relerr & time & relerr & time & relerr & time & relerr & time\\\hline
\multicolumn{9}{|c|}{noise level $\sigma=0$}\\\hline
10\% & 8.27e-02 & 2.64e+02 & 8.17e-02 & 2.65e+02 & 2.42e-01 & 4.78e+01 & 1.82e-01 & 3.19e+02\\
30\% & 5.75e-02 & 1.57e+02 & 5.81e-02 & 1.56e+02 & 1.04e-01 & 1.09e+02 & 8.40e-02 & 2.70e+02\\
50\% & 4.47e-02 & 1.46e+02 & 4.85e-02 & 1.52e+02 & 5.63e-02 & 5.41e+01 & 5.02e-02 & 2.68e+02\\\hline
\multicolumn{9}{|c|}{noise level $\sigma = 0.05$}\\\hline
10\% & 8.38e-02 & 2.73e+02 & 8.22e-02 & 2.52e+02 & 2.43e-01 & 4.92e+01 & 1.83e-01 & 3.14e+02\\
30\% & 5.76e-02 & 1.52e+02 & 5.86e-02 & 1.58e+02 & 1.12e-01 & 8.68e+01 & 8.65e-02 & 2.39e+02\\
50\% & 4.56e-02 & 1.39e+02 & 4.90e-02 & 1.50e+02 & 6.00e-02 & 5.14e+01 & 5.38e-02 & 2.29e+02\\\hline
\multicolumn{9}{|c|}{noise level $\sigma = 0.10$}\\\hline
10\% & 8.75e-02 & 2.46e+02 & 8.70e-02 & 2.43e+02 & 2.50e-01 & 4.77e+01 & 1.86e-01 & 2.89e+02\\
30\% & 5.94e-02 & 1.52e+02 & 6.06e-02 & 1.49e+02 & 1.38e-01 & 8.44e+01 & 9.27e-02 & 2.01e+02\\
50\% & 4.77e-02 & 1.40e+02 & 5.07e-02 & 1.47e+02 & 7.13e-02 & 4.92e+01 & 6.23e-02 & 2.01e+02\\\hline
\end{tabular}}}
\end{center}
\end{table}

\section*{Acknowledgements} The authors thank two
anonymous referees and the associate editor for their very valuable comments. Y. Xu is supported by NSF grant ECCS-1028790. R. Hao and Z. Su are supported by NSFC grants 61173103 and U0935004. R. Hao's visit to UCLA is supported by China Scholarship Council. W. Yin is partially supported by NSF grants DMS-0748839 and DMS-1317602, and ARO/ARL MURI grant FA9550-10-1-0567.


\begin{thebibliography}{10}
\bibitem{candes2009exact}
{\sc E. Cand{\`e}s and B. Recht}, {\em Exact matrix completion
  via convex optimization}, Foundations of Computational mathematics, 9 (2009),
  pp.~717--772.

\bibitem{chen2012matrix}
{\sc C.~Chen, B.~He, and X.~Yuan}, {\em Matrix completion via an alternating
  direction method}, IMA Journal of Numerical Analysis, 32 (2012),
  pp.~227--245.

\bibitem{gandy2011tensor}
{\sc S. Gandy, B. Recht, and I. Yamada}, {\em Tensor completion and
  low-n-rank tensor recovery via convex optimization}, Inverse Problems, 27
  (2011), pp.~1--19.

\bibitem{GrippoSciandrone2000}
{\sc L.~Grippo and M.~Sciandrone}, {\em On the convergence of the block
  nonlinear {G}auss-{S}eidel method under convex constraints}, Oper. Res.
  Lett., 26 (2000), pp.~127--136.

\bibitem{MC-SVT2008}
{\sc J.~Cai, E.~Candes, and Z.~Shen}, {\em A singular value
  thresholding algorithm for matrix completion}, SIAM J. Optim., 20 (2010),
  pp.~1956--1982.
  
\bibitem{jiang2014tensor}
{\sc B.~Jiang, S.~Ma, and S.~Zhang}, {\em Tensor principal component analysis via convex optimization}, Mathematical Programming (2014), pp.~1--35.

\bibitem{kiers2000towards}
{\sc H.A.L. Kiers}, {\em Towards a standardized notation and terminology in
  multiway analysis}, Journal of Chemometrics, 14 (2000), pp.~105--122.

\bibitem{kilmer2013third}
{\sc M.~Kilmer, K. Braman, N. Hao, and R. Hoover}, {\em
  Third-order tensors as operators on matrices: A theoretical and computational
  framework with applications in imaging}, SIAM Journal on Matrix Analysis and
  Applications, 34 (2013), pp.~148--172.

\bibitem{kolda2009tensor}
{\sc T.G. Kolda and B.W. Bader}, {\em Tensor decompositions and applications},
  SIAM review, 51 (2009), pp.~455--500.

\bibitem{kolda2005higher}
{\sc T.G. Kolda, B.W. Bader, and J.P. Kenny}, {\em Higher-order web
  link analysis using multilinear algebra}, in Data Mining, Fifth IEEE
  International Conference on, IEEE, 2005.

\bibitem{kreimer2012tensor}
{\sc N. Kreimer and M.D. Sacchi}, {\em A tensor higher-order singular
  value decomposition for prestack seismic data noise reduction and
  interpolation}, Geophysics, 77 (2012), pp.~V113--V122.

\bibitem{kressner2013low}
{\sc D. Kressner, M. Steinlechner, and B. Vandereycken}, {\em
  Low-rank tensor completion by riemannian optimization}, tech. report, Tech.
  rept. {\'E}cole polytechnique f{\'e}d{\'e}rale de Lausanne, 2013.

\bibitem{lai2013improved}
{\sc M. Lai, Y. Xu, and W. Yin}, {\em Improved iteratively
  reweighted least squares for unconstrained smoothed $\ell_q$ minimization},
  SIAM Journal on Numerical Analysis, 51 (2013), pp.~927--957.

\bibitem{li2010tensor}
{\sc N. Li and B. Li}, {\em Tensor completion for on-board compression of
  hyperspectral images}, in Image Processing (ICIP), 2010 17th IEEE
  International Conference on, IEEE, 2010, pp.~517--520.

\bibitem{Ling-Xu-Yin-mtx-conf-11}
{\sc Q.~Ling, Y.~Xu, W.~Yin, and Z.~Wen}, {\em Decentralized low-rank matrix
  completion}, International Conference on Acoustics, Speech, and Signal
  Processing (ICASSP), SPCOM-P1.4 (2012).

\bibitem{liu2013tensor}
{\sc J.~Liu, P.~Musialski, P.~Wonka, and J.~Ye}, {\em Tensor
  completion for estimating missing values in visual data}, IEEE Transactions
  on Pattern Analysis and Machine Intelligence,  (2013), pp.~208--220.

\bibitem{ma2011fixed}
{\sc S. Ma, D. Goldfarb, and L. Chen}, {\em Fixed point and
  bregman iterative methods for matrix rank minimization}, Mathematical
  Programming, 128 (2011), pp.~321--353.

\bibitem{mu2013square}
{\sc C. Mu, B. Huang, J. Wright, and D. Goldfarb}, {\em Square deal:
  Lower bounds and improved relaxations for tensor recovery}, arXiv preprint
  arXiv:1307.5870,  (2013).

\bibitem{patwardhan2007video}
{\sc K.A. Patwardhan, G. Sapiro, and M. Bertalm{\'\i}o}, {\em
  Video inpainting under constrained camera motion}, Image Processing, IEEE
  Transactions on, 16 (2007), pp.~545--553.

\bibitem{recht2010guaranteed}
{\sc B.~Recht, M.~Fazel, and P.A. Parrilo}, {\em Guaranteed minimum-rank
  solutions of linear matrix equations via nuclear norm minimization}, SIAM
  review, 52 (2010), pp.~471--501.

\bibitem{romera2013new}
{\sc B. Romera-Paredes and M. Pontil}, {\em A new convex
  relaxation for tensor completion}, arXiv preprint arXiv:1307.4653,  (2013).

\bibitem{sauve19993d}
{\sc A.C. Sauve, A.O. Hero~III, W~Leslie Rogers, SJ~Wilderman, and
  NH~Clinthorne}, {\em 3d image reconstruction for a compton spect camera
  model}, Nuclear Science, IEEE Transactions on, 46 (1999), pp.~2075--2084.

\bibitem{sun2005cubesvd}
{\sc J. Sun, H. Zeng, H. Liu, Y. Lu, and Z. Chen}, {\em
  Cubesvd: a novel approach to personalized web search}, in Proceedings of the
  14th international conference on World Wide Web, ACM, 2005, pp.~382--390.

\bibitem{toh2010accelerated}
{\sc K.C. Toh and S.~Yun}, {\em An accelerated proximal gradient algorithm for
  nuclear norm regularized linear least squares problems}, Pacific Journal of
  Optimization, 6 (2010), pp.~615--640.

\bibitem{Tseng-01}
{\sc P.~Tseng}, {\em Convergence of a block coordinate descent method for
  nondifferentiable minimization}, Journal of Optimization Theory and
  Applications, 109 (2001), pp.~475--494.

\bibitem{tucker1966some}
{\sc L.R. Tucker}, {\em Some mathematical notes on three-mode factor analysis},
  Psychometrika, 31 (1966), pp.~279--311.

\bibitem{wen2012lmafit}
{\sc Z.~Wen, W.~Yin, and Y.~Zhang}, {\em Solving a low-rank factorization model
  for matrix completion by a nonlinear successive over-relaxation algorithm},
  Mathematical Programming Computation,  (2012), pp.~1--29.

\bibitem{xing2012dictionary}
{\sc Z.g Xing, M. Zhou, A. Castrodad, G. Sapiro, and
  L. Carin}, {\em Dictionary learning for noisy and incomplete
  hyperspectral images}, SIAM Journal on Imaging Sciences, 5 (2012),
  pp.~33--56.

\bibitem{xu-yin-multiconvex}
{\sc Y.~Xu and W.~Yin}, {\em A block coordinate descent method for regularized
  multi-convex optimization with applications to nonnegative tensor
  factorization and completion}, SIAM Journal on Imaging Sciences, 6 (2013),
  pp.~1758--1789.
  
\bibitem{xu-yin-nonconvex}
{\sc Y.~Xu and W.~Yin}, {\em A globally convergent algorithm for nonconvex optimization based on block coordinate update}, arXiv:1410.1386 (2014).

\bibitem{Xu-Yin-Wen-Zhang-11}
{\sc Y.~Xu, W.~Yin, Z.~Wen, and Y.~Zhang}, {\em An alternating direction
  algorithm for matrix completion with nonnegative factors}, Journal of
  Frontiers of Mathematics in China, Special Issue on Computational
  Mathematics, 7 (2011), pp.~365--384.

\bibitem{zhang2013novel}
{\sc Z. Zhang, G. Ely, S. Aeron, N. Hao, and M. Kilmer}, {\em
  Novel factorization strategies for higher order tensors: Implications for
  compression and recovery of multi-linear data}, arXiv preprint
  arXiv:1307.0805v3,  (2013).

\end{thebibliography}

\end{document}